\documentclass[11pt, letterpaper]{extarticle}
\usepackage[T1]{fontenc}

\usepackage{amsfonts,amsthm,latexsym,amsmath,amssymb,amscd,amsmath, mathrsfs, epsf, xypic, tikz-cd}
\usepackage[left=25mm, right= 25mm, top=20mm, bottom=20mm, includefoot, includehead]{geometry}

\usepackage{tcolorbox}
\newtcolorbox{mybox}{colback=blue!5!white,colframe=blue!75!black}

\usepackage{color}
\usepackage[utf8]{inputenc}
\usepackage{picture}
\usepackage{tikz}
\usepackage{graphicx}
\usepackage{verbatim}
\usepackage[textsize=tiny]{todonotes}
\usepackage{hyperref}
\hypersetup{
    colorlinks,
    citecolor=black,
    filecolor=black,
    linkcolor=black,
    urlcolor=black
}
   
\usepackage[nottoc,numbib]{tocbibind}

 \newtheorem{theorem}{Theorem}[section]
 \newtheorem*{theo*}{Theorem}
 \newtheorem{corollary}[theorem]{Corollary}
 \newtheorem{lemma}[theorem]{Lemma}
 \newtheorem{proposition}[theorem]{Proposition}
 
 \newtheorem{claim}{Claim}
 \newtheorem{ex*}{Example}[section]
 \theoremstyle{definition}
 \newtheorem{definition}[theorem]{Definition}
 \theoremstyle{remark}
 \newtheorem{remark}[theorem]{Remark}
 \theoremstyle{definition}
 \newtheorem{example}[theorem]{Example}
 \newtheorem{question}[theorem]{Question}
  \newtheorem{opquestion}[theorem]{Open-ended Question}

\newcommand{\rk}{\mathrm{rk}}
\newcommand{\crk}{\mathrm{crk}}
\newcommand{\brk}{{\underline{{\bf rk}}}}
\newcommand{\bb}[1]{\mathbf{#1}}

\newcommand{\RR}{\mathbb{R}}
\newcommand{\CC}{\mathbb{C}}

\newcommand{\PP}{\mathbb{P}}

\newcommand{\NN}{\mathbb{N}}
\newcommand{\ZZ}{\mathbb{Z}}
\newcommand{\VSP}{\mathrm{VSP}}
\newcommand{\VPS}{\mathrm{VPS}}
\newcommand{\VSPb}{\underline{\mathrm{VSP}}}
\newcommand{\Ann}{\mathrm{Ann}}
\newcommand{\HF}{\mathrm{HF}}

\title{Border apolarity and varieties of sums of powers}
\author{Tomasz Ma\'ndziuk and Emanuele Ventura}

\date{}

\begin{document}

\maketitle{}

\begin{abstract}
\noindent We study border varieties of sums of powers ($\underline{\mathrm{VSP}}$'s for short), recently introduced by Buczy\'nska and Buczy\'nski, parameterizing border rank decompositions of a point (e.g. of a tensor or a homogeneous polynomial) with respect to a smooth projective toric variety
and living in the Haiman-Sturmfels multigraded Hilbert scheme. Their importance stems from the role of border tensor rank in theoretical computer science, especially in the estimation of the exponent of matrix multiplication, a fundamental and still unknown quantity in the theory of computation. We compare $\underline{\mathrm{VSP}}$'s to other well-known loci in the Hilbert scheme, parameterizing 
scheme-theoretic versions of decompositions. The latter ones are crucial in that they naturally explain the existing severe {\it barriers} to giving good lower bounds on ranks. 
We introduce the notion of border identifiability and provide sufficient criteria for its appearance, which rely on the multigraded regularity of Maclagan and Smith. We link border identifiability to wildness of points. Finally, we determine $\underline{\mathrm{VSP}}$'s in several instances and regimes, in the contexts of tensors and homogeneous polynomials. These include 
concise $3$-tensors of minimal border rank and in particular of border rank three.
\end{abstract}

\medskip
{\footnotesize
\noindent\textbf{Addresses:}\\
Tomasz Ma\'ndziuk, \href{mailto:tomasz.mandziuk@unitn.it}{tomasz.mandziuk@unitn.it}, Universit\`{a} di Trento, Dipartimento di Matematica, Via Sommarive 14,  38123 Povo (Trento), Italy; Faculty of Mathematics, Informatics and Mechanics, University of Warsaw, Banacha 2, 02-097 Warsaw, Poland \\
Emanuele Ventura, \href{mailto:emanuele.ventura@polito.it}{emanuele.ventura@polito.it}, Politecnico di Torino, Dipartimento di Scienze Matematiche ``G. L. Lagrange'', Corso Duca degli Abruzzi 24, 10129 Torino, Italy\\
\vspace{1mm}

\noindent\textbf{Keywords:} 
Toric variety,
Cox ring, Tensors, Forms, 
Border rank, 
Border variety of sums of powers,
Apolarity, Multigraded Hilbert scheme, Multigraded Castelnuovo-Mumford regularity. 

\noindent\textbf{AMS Mathematical Subject Classification 2020:}
Primary 14C05; Secondary 14M25, 15A69, 68Q17.
}

\tableofcontents

\section{Introduction}

The last two decades have witnessed steady progress on the theory and applications of tensor and Waring ranks. Perhaps the strongest driving forces behind these fast developments were, on one hand, uncovering the rich geometry of special projective varieties and, on the other, exploiting the fundamental connection between tensors and theoretical computer science. This second link is very deep and goes back to the works of Strassen \cite{Str87,Str69} and Bini \cite{Bin80}: tensor and border tensor rank of the matrix multiplication tensor are, up to a constant, equal to the asymptotic number of arithmetic operations (digits multiplications) needed to optimally compute the product of two matrices; see \cite[\S 12.3.2]{W19} or the extensive discussions in \cite{Landsberg17} and \cite{BCS97}. 

On the geometric and algebraic side, there have been intense research efforts to understand the subtleties of tensor and Waring ranks, entailing secant varieties \cite[Chapter 1]{russo} and Macaulay's theory of apolarity and inverse systems \cite[\S 1.1]{ik}. 

Strikingly, scheme-theoretic versions of ranks have been an important tool to understand the previous ranks. These schematic ranks take into account more general zero-dimensional schemes, besides the reduced ones featured in the tensor and Waring ranks. The latter more general framework naturally leads to new notions: the {\it smoothable rank} and the {\it cactus rank}, originally called {\it scheme length} \cite[Definition 5.1]{ik}. We recall their definitions in \S\ref{sec:prelim}.

In a $2019$ groundbreaking work, 
Buczy\'nska and Buczy\'nski \cite{bb19} introduced a new method for border estimation, called {\it border apolarity}, see Theorem \ref{mainbb} below. This result opened up a way to potentially overcome the well-known {\it barriers} affecting vector bundle methods for lower bounds on border ranks \cite{Ga17, EGOW18}. These barriers are naturally explained in the context of schematic ranks, 
as cactus rank tends to be much lower than border rank (but not always!), whereas vector bundle 
methods, such as flattening constructions, give lower bounds on the former.  

The strong effectiveness of the new method was demonstrated by Conner, Harper and Landsberg \cite{chl19} who 
proved the following lower bounds on the border rank of rectangular matrix multiplication (of indicated sizes): $\brk(M_{\langle 2,n,n\rangle})\geq n^2 + 1.32n$ and $\brk(M_{\langle 3,n,n\rangle})\geq n^2 + 1.6n$. This is remarkable as the previously known bounds had no linear terms in $n$. 

As fundamental aspect of their border apolarity theorem,  Buczy\'nska and Buczy\'nski discovered an {\it algebraic} way to describe and parameterize border rank decompositions with respect to a smooth toric projective variety. This gives rise to the central geometric objects lurking in the theory, that are the main characters of our article: the {\it border varieties of sums of powers} (called $\VSPb$'s for short); see Definition \ref{vspbar}. As yet, to the best of our knowledge, only few examples and results about these varieties are known; see e.g. \cite{CHL22} and \cite{HMV}. More recently, Jelisiejew, Ranestad and Schreyer \cite{JRS} studied
loci inside multigraded and usual Hilbert schemes that are linked to $\VSPb$'s of quadratic forms; see Remark \ref{remark on JRS}. Our main motivation is then to start a systematic study of these interesting objects that could shed light on the nature of border rank. We believe this has the potential to reverberate in explicit and powerful results in the theory of computation, as border varieties of sums of powers are arguably the ultimate 
geometric objects governing border rank phenomena, that are of great relevance in algebraic complexity. \\

\noindent {\bf Main results.}\\
\noindent The original set-up of border apolarity is for smooth projective toric varieties and their (finitely generated and multigraded) Cox rings. 
In \S\ref{sec:Hilbert schemes}, we make use of the multigraded regularity of Maclagan and Smith, which extends the classical Castelnuovo-Mumford regularity to Cox rings, {\it both} as a framework and as a tool for border apolarity. In \cite{MS05}, the authors employ this machinery to sketch the construction of Hilbert schemes of toric varieties, each parameterized by the datum of a multigraded polynomial, mimicking Grothendieck-Hilbert schemes which depend on a univariate polynomial instead. This is an important result for toric varieties, so we embed this natural construction into the border apolarity theory of Buczy\'nska and Buczy\'nski. In doing so, we give detailed proofs of the ingredients along the way. We believe this is useful for both the toric and the border apolarity literatures. 

Let $S = S[X]$ be the Cox ring of a smooth projective toric variety $X$ which is multigraded by $\ZZ^s$ for some $s$. One defines a similarly graded dual ring $T$; those are equipped with a pairing given by differentiation of $S$ on $T$. 
Inside the Haiman-Sturmfels multigraded Hilbert scheme $\mathrm{Hilb}^{h_{r,X}}_{S}$, Buczy\'nska and Buczy\'nski look at the irreducible component defined by all the limits of all saturated ideals of $r$ points with generic Hilbert function $h_{r,X}$. This projective variety is called $\mathrm{Slip}_{r,X}$. For $F\in T$, all the ideals $J$ in the following closed locus
\[
\VSPb(F,r) = \left\lbrace J\in \mathrm{Slip}_{r,X} \mbox{ such that } J\subset \mathrm{Ann}(F)\right\rbrace 
\]
govern the border decompositions of a homogeneous element $F\in T$ (see Theorem \ref{mainbb}). One natural approach is contrasting the latter with well-known loci attached to $F$: $\mathrm{VSP}(F,r)$ and $\VPS(F,r)$; see Definition~\ref{vspclassic} and Definition~\ref{vps}. In the description of the relationships amongst $\VSPb, \VPS$ and $\VSP$ there are several subtleties lurking. First of all, 
$\VPS$ could be non closed, a fact leading to perhaps unintuitive phenomena at the boundary, informally known as {\it bad limits} in \cite{Ranestad22}; see Remark~\ref{nonclosed VPS}. We give a sufficient condition for the closedness of $\VPS$ in Proposition \ref{prop:VPS_closed}.  

Although there is a proper surjective morphism $\phi_{r,X}: \mathrm{Slip}_{r,X}\rightarrow \mathrm{Hilb}_{sm}^r(X)$, this {\it does not} usually descent to a map 
from $\VSPb(F,r)$ to $\VPS(F,r)$. However, one has that $\phi_{r,X}^{-1}(\VPS(F,r)\cap \mathrm{Hilb}_{sm}^r(X))\subset \VSPb(F,r)$ (Lemma \ref{lem:trivial_containment}). When equality holds, we say that $\VSPb(F,r)$ is {\it of fiber type}.  
A special role in the theory is played by homogeneous elements $F\in T$ such that $\mathrm{srk}(F) > \brk(F)$; these are called {\it wild}. 
Proposition \ref{lem:fiber_type_is_non_wild} proves that whenever $\VSPb(F, \brk(F))$ is of fiber type, then $F$ cannot be wild. 
Wild elements have the remarkable property that their border varieties $\VSPb(F, \brk(F))$ do {\it not} have points corresponding to saturated ideals. 

Given $X$, we fix an embedding of $X\subset \PP(T_{\bb{v}})$ for some multidegree $\bb{v}\in \ZZ^s$. One may wonder how $\VSPb(F,r)$ behaves as we move inside $\sigma_r(X)$. Proposition~\ref{prop:vspbar_of_general}  
states there exists a dense open set $W\subset \sigma_r(X)$ such that $\VSPb(F,r)$ contains a saturated ideal (a similar conclusion holds when we search 
for a radical ideal). A corollary to this result (Corollary~\ref{twocaseswherenonsatareproper}) is that whenever $\sigma_r(X)$ is nondefective or fills up without excess the ambient space, then for a general $F\in \sigma_r(X)$, $\VSPb(F,r)$ contains {\it only} saturated ideals. However without these assumptions, this might fail as observed in Remark \ref{rmk:mayfailfordefective}.

The openness of the locus of saturated ideals in $\mathrm{Hilb}^h_S(X)$ \cite{JM,bb21}, \cite[Proposition 3.9]{JM} and its slight generalization in our Theorem \ref{thm:restriction on saturable locus is closed immersion} all suggest that we should compare  $\VSPb, \VPS$ and $\VSP$
from a birational perspective. Relying on these results we prove the following. 

\begin{theo*}[Theorem \ref{birmap1}]
Let $F\in T_\bb{v}$ be a homogeneous polynomial of degree $\bb{v}$ and $r$ be a positive integer.
Then $\phi_{r,X}$ induces a bijection between the set of those irreducible components of the closure of
$\VPS(F,r)\cap \mathrm{Hilb}_{sm}^r(X)$ that contain a scheme with the generic Hilbert function $h_{r,X}$ and the set of those irreducible components of $\VSPb(F,r)$ that contain a saturated ideal. Under this bijection, the irreducible components in correspondence are birational. 
\end{theo*}

When $X=\PP^n$, this provides a way to transfer birational results of Ranestad-Schreyer \cite{rs13} and Massarenti-Mella \cite{mm} on irreducible components of $\VSP(F,r)$, for some ranges of $r$ and general $F$, to some irreducible components of $\VSPb(F,r)$; see Corollary \ref{cor: bir quadrics} and Corollary \ref{cor: bir mm}. 

Let $X=\PP^n$ and $S$ be its Cox ring, i.e. $S = \CC[y_0,\ldots, y_n]$ with the standard grading. Let $J\subset S$ be a complete intersection
ideal of codimension $n$ and degree $r=\prod_{i=1}^n a_i$. Let $d=\sum_{i=1}^n a_i - (n+1)$. In this set-up we show a partial converse to \cite[Theorem~3.4~and~Corollary~3.3]{JM},
as a consequence of their Ext criterion for membership of an ideal in $\mathrm{Slip}_{r,\PP^n}$. 

\begin{theo*}[Theorem \ref{prop:complete_intersection}]
Let $W$ be a codimension one subspace of $J_d$. There exists $I\in \mathrm{Slip}_{r, \PP^n}$ with $I_{\geq d} = (W) + J_{\geq d+1}$ if and only if $(J^2)_d\subseteq W$.
\end{theo*}

This has applications to the varieties $\VSPb(F,r)$ for elements $F$ such that the apolar ideal $\Ann(F)$ contains a complete intersection ideal $J\subset \Ann(F)$ of degree
$r$. The main point is that we are able to show that $\VSPb(F,r)$ is a projective space or that it has some rational irreducible components; see Proposition \ref{prop:vsp_ci_easy}, Remark \ref{rmk: complete intersections inside ann} and Proposition \ref{prop:application to vspbar of mons}.

We say that $F\in T_{\bb{v}}$ is {\it border identifiable} if $\VSPb(F,\brk(F))$ is a single point. We establish a criterion for border identifiability employing the machinery of multigraded regularity of Maclagan and Smith. Let $\mathcal K = \NN \bb{c}_1+\cdots + \NN\bb{c}_{l}$ be the integral nef cone of $X$. Our main theorem here reads as follows: 

\begin{theo*}[Theorem \ref{thm:idenfiability via toric regularity}] 
Let $X\subset \PP(T_{\bb v})$ and $r=\brk_X(F)$ be the border rank of $F\in T_{\bb{v}}$. Suppose that there exists $\bb{u} \in \mathbf{\mathcal  K}$ such that
\[
\HF(S/\Ann(F), \bb{u}) = \HF(S/\Ann(F), \bb{u}+\bb{c}_1+\cdots + \bb{c}_l) = r.
\]
If there exists a $B$-saturated ideal $I\in \VSPb(F, r)$, then $\VSPb(F, r) = \{I\}$.
\end{theo*}

\noindent For $X = \PP^n$, $d=2s+1$ and $r = \binom{n+s}{s}$, this implies that 
a general $F\in \sigma_{r}(\nu_d(\PP^n))$ is border identifiable (Corollary \ref{cor: p^n and deg 2s+1}).  

A fundamental playground for border apolarity is the challenging world of tensors. These constitute one of the first motivations behind the very conception of border apolarity. Matrix multiplication is a $3$-tensor, so this class of tensors is particularly important -- besides just being the next case after matrices -- and already encapsulate richness of structure in sharp contrast with matrices. Let $X = \PP^{m-1}\times \PP^{m-1}\times \PP^{m-1}$, $S$ be its Cox ring and $T$ be its dual. Then $T_{\bb{1}}\cong \CC^m\otimes \CC^m\otimes \CC^m$, where $\bb{1}=(1,1,1)$. We focus on concise minimal border rank tensors, i.e. $F\in T_{\bb{1}}$ with $\brk(F) = m$ and that are 
not annihilated by any multigraded linear polynomial. Our result in this direction is very much related to work of Jelisiejew, Landsberg and Pal \cite{JLP} on
this class of tensors.  

\begin{theo*}[Theorem \ref{thm:minimalBRTensors}]
Let $F\in T_{\bb{1}}$ be concise and of minimal border rank, i.e. $\brk(F)=m$. Let $I = (\Ann(F)_{(1,1,0)})+(\Ann(F)_{(1,0,1)}) + (\Ann(F)_{(0,1,1)})\subset S$ and $K=(I\colon B^\infty)$. Then the following statements hold:
\begin{enumerate}
\item[(i)] If $\HF(S/I, \bb{1}) \neq m$, then $F$ is wild. 

\item[(ii)] If $\HF(S/I, \bb{1}) =m$, then $F$ is not wild if and only if $I_{(a,b,c)} = K_{(a,b,c)}$ for every $(a,b,c) \in \mathcal{S}$, where $\mathcal{S} = \{(1,0,0), (0,1,0), (0,0,1), (1,1,0), (1,0,1), (0,1,1),(1,1,1)\}.$
\end{enumerate}
\end{theo*}
\noindent Note that statement (i) is \cite[Theorem~9.2]{JLP}. The corollary to this is that: if $F$ is a nonwild concise minimal border rank tensor in $T_{\bb{1}}$, then $\VSPb(F, m) = \{K\}$ where $K$ is the saturation of $I = (\Ann(F)_{(1,1,0)}) + (\Ann(F)_{1,0,1}) + (\Ann(F)_{(0,1,1)})$ (Corollary \ref{cor:minimalBRWild}). In other words, much of the complexity of concise minimal border rank $3$-tensors is due to wild tensors, at least from the perspective of $\VSPb$'s. 

Minimal border rank tensors $F\in T_{\bb{1}}$ with $m=3$ are classified \cite[Theorem~1.2]{BL14}. Therefore, in this case, 
we can actually improve the previous result to an explicit description of all $\VSPb(F, 3)$ for all such $F$'s. We prove the ensuing. 

\begin{theo*}[Theorem \ref{theo:min border rank three}]
Let $X = \PP^2\times \PP^2\times \PP^2$ and let $F$ be a border rank three concise tensor in $\CC^3\otimes \CC^3\otimes \CC^3\cong T_{\bb{1}}$.
The variety $\VSPb(F, 3)$ is either a single point, or $\VSPb(F, 3)\cong \PP^3$ when $F$ is wild. 
\end{theo*}
\noindent This answers a question of Buczy\'nska and Buczy\'nski \cite[\S 5.2]{bb19} about the geometry of $\VSPb(F, 3)$. 

As mentioned, whenever $F$ is wild, $\VSPb(F, \brk(F))$ consists only of nonsaturated ideals. Does the converse hold? We give a negative answer to this question, providing a monomial counterexample; see Example \ref{nonwild_no_saturated_ideals}.  Elaborating more on wildness, we prove a result that is similar in the spirit to that for tensors explained above. When $X = \PP^n$ and $d=3$ or  $d\geq n+2$, we prove that for a minimal border rank $F\in T_d$ one has $\mathrm{Hess}(F)\neq 0$ if and only if $\VSPb(F,n+1)$ consists of a unique saturated ideal. When this holds, the unique saturated ideal is $(\Ann(F)_2)$ (Corollary \ref{vspwithonlysat}). This follows from our Theorem \ref{thm:idenfiability via toric regularity} on multigraded regularity and \cite[Theorem 4.9]{HMV}, which characterizes wildness for minimal border rank forms.

We investigate binary forms, proving that $\VSPb(F,\brk(F))$ are either one point or $\PP^1$ (Proposition~\ref{prop:binary forms}). Leveraging the classifications of ternary cubic forms and reducible cubic forms, we describe $\VSPb(F,\brk(F))$ for each of these classes in Theorem \ref{thm:ternary_cubic} and in \S\ref{ssec: reducible cubics}. Note that reducible cubics are particularly interesting in algebraic complexity theory, as they correspond to symmetrizations of the big and small Coppersmith-Winograd tensors. 

The next results that we showcase are for  $X = \PP^2$. For ternary monomials $F$, we show that $\VSPb(F,\brk(F))$ are of fiber type and we explicitly describe the corresponding $\VPS(F, \brk(F))$; these are Proposition \ref{prop:case c>= a+b}, Theorem \ref{thm:case c=a+b-1} and Theorem~\ref{prop:vps_vspb_>=-2}. Taking momentarily a break from $X=\PP^2$ and setting $X=\PP^n$ for any $n\geq 1$, Theorem \ref{borderrankmonimprove} shows that if $F = x_0^{a_0}\cdots x_n^{a_n}$ with $0<a_0\leq a_1\leq \cdots \leq a_n$ and $a_n\geq \left(\sum_{i=0}^{n-1} a_i\right)-2$, then 
$\brk(F) = \prod_{i=0}^{n-1}(a_i+1)$. This improves the lower bound on $a_n$ given in \cite[Example 6.8]{bb19} by one.

Following the terminology of \cite[\S 1]{rs00}, a {\it nondegenerate} ternary form of degree $d=2p-2$ is a form $F\in T_d$ such that $\Ann(F)_{\leq p-1}=0$. This is a generality-type assumption.  We show that, for $2 \leq p \leq 5$, we have an isomorphism between $\VPS(F, r_p)=\VSP(F,r_p)$ and $\VSPb(F,r_p)$ given by $\phi_{r_p,\PP^2}$ with $r_p =\binom{p+1}{2} = \brk(F)$ (Theorem \ref{plane low even degree}). 

In \S\ref{ssec: reducible vspbar}, we exhibit an example of a reducible $\VSPb(F,\brk(F))$. This addresses and solves the question whether $\VSPb(F,\brk(F))$ can be reducible even in $\PP^2$. Note that this is not the case when $X=\PP^1$. Finally, in \S\ref{sec:Schubert variety}, Proposition \ref{prop: schubert} describes an example of $\VSPb(F, \brk(F))$ that is isomorphic to the Schubert variety $\Sigma_1 \subset \mathbb{G}(3, 5)$.   \\

\noindent{\bf Organisation of this paper.} \\
In \S \ref{sec:prelim}, we discuss preliminaries. We recall toric varieties and their Cox rings where apolarity will take place, and the notions of ranks we shall be concerned with. In \S \ref{sec:Hilbert schemes}, we describe the framework of multigraded regularity, we recall multigraded Hilbert schemes of 
Haiman-Sturmfels and Grotendieck-Hilbert schemes, and finally construct the morphim $\phi_{r,X}$. In \S \ref{sec:VSPb versus VSP vers VPS}, we recall the definitions of the loci we are interested in and show all the results based on comparing $\VSPb$, $\VPS$ and $\VSP$. In \S \ref{sec:Membership in Slip}, we discuss the Ext criterion in the case of complete intersections. 
In \S \ref{sec:Multigraded reg}, we prove Theorem \ref{thm:idenfiability via toric regularity} and corollaries thereof. 
In \S \ref{sec:minimal border rank}, we show Theorem \ref{thm:minimalBRTensors} and its Corollary \ref{cor:minimalBRWild}. The title of \S\ref{sec:botany of forms} is {\it botany of forms}. The name refers to the great variety of plants in Nature (in this terminology we were influenced 
once again by Buczy\'nska and Buczy\'nski). Of course, we are able to just see some glimpses of it: we describe the $\VSPb$'s for binary 
forms (\S\ref{ssec:binary}), cubic ternary forms and reducible cubics (\S\ref{ssec:cubics}). In \S\ref{ssec:monomials}, we discuss ternary monomials; in Theorem \ref{borderrankmonimprove} we provide the improvement by one on the border rank of monomials. 
In \S\ref{ssec: reducible vspbar} and \S\ref{sec:Schubert variety}, we describe the instances  when $\VSPb$ is reducible and when $\VSPb$ is a Schubert variety, respectively. Finally, in \S\ref{sec:Outlook}, we record natural follow-up questions to this work. \\

\noindent {\bf Acknowledgements.}\\
We thank the organizers of the semester  ``Tensors: geometry, complexity and quantum entanglement'' at the University of Warsaw and at the Institute of Mathematics of the Polish Academy of Sciences, held in the Fall of 2022, for their warm hospitality and support. 
We thank J. Buczy\'nski, A. Conca, H. Huang  (especially for conversations on \S\ref{ssec: ternary of low even deg}), J. Jelisiejew, A. Massarenti, G. Ottaviani, and K. Ranestad for very useful discussions. The first author is a member of TensorDec laboratory of the Mathematical Department of Trento. The second author is a member of the GNSAGA group of INdAM and was partially supported by the INdAM-GNSAGA project ``Classification Problems in Algebraic Geometry: Lefschetz Properties and Moduli Spaces'' (CUP\textunderscore E55F22000270001).   
In this paper, we often relied (for experiments or parts of proofs) on the invaluable help of the algebra software \texttt{Macaulay2} \cite{M2}.   

\section{Cox rings, ranks and border apolarity}\label{sec:prelim}

\subsection{Cox rings and Picard groups}
For details on toric varieties, their combinatorial construction and their properties we refer to the textbooks by Cox, Little and Schenck \cite{CLS11} and by Fulton \cite{Fulton}. 
Let $X = X_{\Sigma}$ be a $d$-dimensional complex smooth projective toric variety corresponding to the fan $\Sigma\subset N\cong \ZZ^d$. Let 
$n+1$ be the number of one-dimensional cones (the {\it rays}) and let ${\bf b}_0,\ldots, {\bf b}_{n}$ be the unique minimal lattice vectors generating 
the rays such that their real span is $N_{\RR} = N\otimes_{\ZZ} \ZZ^d \cong \RR^d$. Let $s = n+1-d$ and let $A = [{\bf a}_0 \cdots {\bf a}_{n}]$ be an $s\times (n+1)$ integral 
matrix such that there is a short exact sequence 
\[
0\longrightarrow \ZZ^d \xrightarrow{{\bf b}^T=[{\bf b}_0 \cdots {\bf b}_{n}]^T} \ZZ^{n+1} \xrightarrow{[{\bf a}_0 \cdots {\bf a}_{n}]} \ZZ^s \longrightarrow 0.
\]
Such an integral matrix $A$ (called the {\it Gale dual of ${\bf b}$}) exists and it is uniquely determined up to $\mathrm{GL}(s,\ZZ)$. The group of Cartier divisors or Picard group $\mathrm{Pic}(X)$ is isomorphic to $\ZZ^s$ because $X$ is smooth \cite[\S 3.4]{Fulton}. For simplicity of exposition, we identify a cone $\sigma\in \Sigma$ with 
a subset of $[n]=\lbrace 0,\ldots, n\rbrace$. 

The {\it Cox ring} of $X$ is the polynomial ring $S[X] = \CC[y_0,\ldots,y_n]$ with a $\mathrm{Pic}(X)\cong \ZZ^s$-grading defined by $\deg(x_i) = {\bf a}_i\in \ZZ^s$ for $0\leq i\leq n$ \cite[\S5.2]{CLS11}. Alternatively, one may write $S[X] = \bigoplus_{D\in \mathrm{Pic}(X)} H^0(\mathcal O_X(D))$ \cite[Proposition 1.1]{Cox95}, where each summand is the $\CC$-vector space of global sections of the corresponding Cartier divisor. In the following, we shall identify Cartier divisors with integral vectors. Every Cox ring $S[X]$ has a {\it positive grading}, i.e. the only monomial of degree ${\bf 0}\in \ZZ^s$ is $1$. The combinatorics of the fan $\Sigma$  encodes the {\it irrelevant ideal} $B = \left(\prod_{i\notin \sigma} x_i \ | \sigma\in \Sigma\right)$.
\begin{example}
The first example is when $X = \PP^d$ and $d=n$. In this case, $S[X] = \CC[y_0,\ldots, y_n]$ with the standard $\ZZ$-grading given by $\deg(x_i) = 1$. 
The ideal $B$ is $(y_0,\ldots,y_n)\subset S[X]$. 
\end{example}

The irrelevant ideal $B$ allows one to obtain a quotient construction of $X$ as a geometric quotient $X= \left(\mathbb A^{n+1}\setminus \mathrm{Spec}(S[X]/B)\right)/G$, where 
$G = \mathrm{Hom}_{\ZZ}(\mathrm{Pic}(X), \CC^{*})$, in the sense of geometric invariant theory \cite[Theorem 5.1.11]{CLS11}. This quotient construction generalizes to provide a toric ideals-subscheme correspondence between $S[X]$ and $X$. More generally, Cox shows that there is a categorical equivalence between coherent 
$\mathcal O_X$-modules and finitely generated $\ZZ^s$-graded $S$-modules {\it modulo $B$-torsion} \cite[Proposition 3.3]{Cox95}. For a homogeneous ideal $I\subset S[X]$,
being $B$-torsion free means that $I = (I:B^{\infty}) = \lbrace g\in S[X] \ | \  \exists N\in \mathbb N : gB^N\subset I\rbrace$. 
For ideal sheaves on $X$, the Cox correspondence is as follows: there is a bijection between $\ZZ^s$-homogeneous ideal $I\subset S[X]$ that are $B$-saturated and closed subschemes $Z\subset X$. From the fan $\Sigma\subset \ZZ^s$, one can construct another semigroup. 

\begin{definition}[{\bf Nef Cone}]\label{nefcone}
For each cone $\sigma\in \Sigma$, let $\mathbb N A_{\widehat{\sigma}} = \lbrace \sum_{i\notin \sigma} \lambda_i {\bf a}_i, \lambda_i \in \mathbb N\rbrace$. The semigroup $\mathbf{\mathcal  K}$ is defined to be 
\[
\mathbf{\mathcal  K} = \bigcap_{\sigma \in \Sigma} \NN A_{\widehat{\sigma}}\subset \ZZ^s.
\]
The points of the cone $\mathbf{\mathcal K}\otimes_{\ZZ} \RR$ correspond to numerically effective (nef)  $\RR$-divisors on $X$, and it is called the {\it nef cone} of $X$. When $X$ is projective, this is a pointed $s$-dimensional cone \cite[Theorem 6.3.22]{CLS11}, where $s$ as above is the Picard rank of $X$. The nef cone is the closure of the so-called {\it K\"{a}hler} or {\it ample cone} of $X$ by the well-known Nakai-Moishezon-Kleiman criterion and it is the dual cone to the (closed) {\it Mori cone of curves}. Note that a Cartier divisor 
on $X$ is nef if and only if it is globally generated \cite[Theorem 6.3.12]{CLS11}. So the points in $\mathcal K$ correspond to divisors with nonzero global sections.
\end{definition}

\subsection{Toric apolarity and ranks}

For this part of the exposition, we follow \cite[\S 3.1]{bb19}. 
Given the $\ZZ^s$-graded Cox ring $S[X] = \CC[y_0,\ldots,y_n]$, define
the dual graded polynomial ring $T[X]$ defined as
\[
T[X] = \bigoplus_{D\in \ZZ^s} H^0(\mathcal{O}_X(D))^{*}\cong \CC[x_0,\ldots,x_n].
\]
The ring $T[X]$ is a graded $S[X]$ module with the differential graded action induced by the formula 
$y_i \circ x_j = \delta_{i,j}$. As mentioned before, we identify Cartier divisors with integral vectors $\bb{v}$, writing 
$S_{\bb{v}}$ and $T_{\bb{v}}$ for the degree $\bb{v}$ summands of these rings. 

\begin{example}\label{ex: P^n}
When $X = \PP^n$,  
$S[X]=\CC[y_0,\ldots,y_n]$ and its dual graded ring $T[X] = \CC[x_0,\ldots,x_n]$ are standard $\ZZ$-graded rings and the action described above is the classical {\it Macaulay apolarity action} \cite[\S 1.1]{ik}. 
\end{example}

\begin{example}\label{ex:tensors}
When $X = \PP^{n_1}\times \cdots \times \PP^{n_s}$, $S[X] = \CC[y_{1,0},\ldots, y_{1,n_1}]\otimes \cdots \otimes \CC[y_{s,0},\ldots, y_{s,n_s}]$; its dual graded ring $T[X]$ has a similar description. The grading is given by $\deg(y_{i,j}) = \bb{e}_i\in \ZZ^s$, the $i$th standard basis vector. 
When $s=2$ or $3$, we shall use $\alpha_i, \beta_j$ and $\gamma_k$ to denote 
the generators of the corresponding Cox ring. 
\end{example}

\noindent The elements of $T_d$ from Example \ref{ex: P^n} are homogeneous polynomials of degree $d$ and are classically called {\it forms} (or {\it symmetric tensors}). When $n=2$, they are {\it binary} forms; when $n=3$, they are {\it ternary} forms. The elements of $T_{\bb{1}}\cong \CC^{n_1+1}\otimes \cdots \otimes \CC^{n_s+1}$ from Example \ref{ex:tensors}, where $\bb{1}=(1,\ldots, 1)\in \ZZ^s$, are called {\it tensors}. 

In the rest of the article, once the toric variety $X$ is fixed, we drop this symbol from the notations $S[X]$ and $T[X]$. For a homogeneous ideal $J\subset S$, let $J_{\bb{v}}$ denote the complex finitely generated vector space consisting of its degree $\bb{v}$ homogeneous elements. For $J\subset T$, let $\overline{J}$ denote its saturation with respect to $B$, i.e. 
\[
\overline{J} = (J:B^{\infty})\subset S. 
\]
The {\it multigraded Hilbert function} 
of a homogeneous $J\subset S$ is the numerical function $\mathrm{HF}(S/J,\cdot): \ZZ^s\to \NN$ defined by 
\[
\mathrm{HF}(S/J,\bb{v}) = \dim_{\CC} S_{\bb{v}} - \dim_{\CC} J_{\bb{v}}. 
\]
A numerical function
$h: \ZZ^s\to \NN$ is said to be an {\it admissible multigraded Hilbert function} if there exists an ideal $J\subset S$ such that $\mathrm{HF}(S/J,\cdot) = h$. Given an element $0\neq F\in T_{\bb{v}}$, denote $[F]\in \PP(T_{\bb{v}})$ to be the corresponding point in projective space. 

\begin{definition}[{\bf Apolar ideals}]
Let $F\in T_{\bb{v}}$. Then its {\it apolar} or {\it annihilator ideal} is 
\[
\Ann(F) = \lbrace \psi\in S \ | \ \psi\circ F = 0\rbrace \subset S. 
\]
This is a $\ZZ^s$-homogeneous ideal. 
\end{definition}
Given $X$, we fix an ample line bundle $\mathcal L=\mathcal O_X(D)$ on $X$, whose divisor corresponds to a vector $\bb{v}\in\ZZ^s$, embedding 
$X$ in the projective space $\PP(H^0(\mathcal L)^{*})=\PP(T_{\bb{v}})$. 
Let $Z\subset X\subset \PP(T_{\bb{v}})$ be a closed subscheme. Then its {\it projective span}  $\langle Z\rangle$ the smallest linear space containing $Z$. Equivalently, $\langle Z\rangle = \PP((S_\mathbf{v}/I_{Z,\bb{v}})^*)\subset \PP(T_{\bb{v}})$, where $I_{Z,\bb{v}}$ is the degree $\bb{v}$ homogeneous summand of the $B$-saturated ideal $I_Z$ of $Z$. 

\begin{definition}[{\bf Rank}]
The $X$-rank of $[F]\in \PP(T_{\bb{v}})$ is the minimal integer $r\geq 1$ such that there exist $r$ points $p_1,\ldots, p_r\in X\subset \PP(T_{\bb{v}})$
such that $[F]\in \langle p_1,\ldots, p_r\rangle$. Equivalently, let $Z$ be the smooth scheme consisting of the union of the points $p_i$. Then $[F]\in \langle p_1,\ldots, p_r\rangle$ is equivalent to the condition $I_Z\subset \Ann(F)$. This is the same as the classical {\it apolarity lemma} \cite[Lemma 1.15]{ik}, but here $X$ is any smooth projective toric variety. This equivalence is then called {\it multigraded apolarity}. 
\end{definition}

Given a homogeneous ideal $I\subset S$ and $F\in T_{\bb{v}}$, the following equivalence is well-known and very useful 
to put apolarity to work. 

\begin{proposition}[{\cite[Proposition 3.5]{bb19}}]\label{inclusion in onedeg}
$I\subset \Ann(F) \Longleftrightarrow I_{\bb{v}}\subset \Ann(F)_{\bb{v}}$. 
\end{proposition}

\begin{definition}[{\bf Border rank}]
For a point $[F]\in \PP(T_{\bb{v}})$, the $X$-{\it border rank} of $F$ is the minimal integer $r\geq 1$ such that $[F]\in \sigma_r(X)$, the $r$-th secant variety of $X\subset \PP(T_{\bb{v}})$. When $X$ and its embedding are fixed, the border rank of $F$ is denoted $\underline{{\bf rk}}_X(F)$ (or simply $\underline{{\bf rk}}(F)$, whenever dropping $X$ should not cause confusion). 
\end{definition}

\begin{definition}[{\bf Smoothable rank}]
The $X$-{\it smoothable rank} of $[F]\in \PP(T_{\bb{v}})$ is the minimal integer $r\geq 1$ such that there exists a finite scheme $Z\subset X$ of length $r$ which is {\it smoothable}
(in $X$) and $[F]\in \langle R\rangle$. Equivalently, there exists a finite smoothable scheme $Z\subset X$ of length $r$ whose $B$-saturated ideal satisfies $I_Z\subset \mathrm{Ann}(F)\subset S$. This is in analogy with the classical Apolarity lemma cited above, when $X=\PP^n$ and $Z$ is a smooth finite scheme. When $X$ and its embedding are fixed, the smoothable rank of $F$ is denoted $\mathrm{srk}_X(F)$ (or simply $\mathrm{srk}(F)$, whenever dropping $X$ should not cause confusion).
\end{definition}

\begin{remark}
Smoothable and border ranks satisfy $\underline{{\bf rk}}(F)\leq \mathrm{srk}(F)$; see \cite[\S 2.1]{bb15}. Equality holds
for general points of any secant variety. The difference arises in general from the {\it failure} of the equality $\langle \lim_{t\rightarrow 0} Z(t)\rangle = \lim_{t\rightarrow 0} \langle Z(t) \rangle$, where $Z(t)$ is a family of finite schemes say over the base $\mathrm{Spec}(\CC[t^{\pm 1}])$  and $\lim_{t\rightarrow 0} R(t)$ denotes its flat limit.
\end{remark}

Motivated by this exceptional behaviour, Buczy\'nska and Buczy\'nski introduced {\it wildness}:

\begin{definition}[{\bf Wildness}]\label{def:wild}
An element $F\in T_{\bb{v}}$ is {\it wild} if $\mathrm{srk}(F) > \brk(F)$. 
\end{definition}

\begin{definition}[{\bf Cactus rank}]\label{def:cactus}
The $X$-{\it cactus rank} of $[F]\in \PP(T_{\bb{v}})$ is the minimal integer $r\geq 1$ such that there exists a finite scheme $Z\subset X$ of length $r$  such that $[F]\in \langle Z\rangle$. 
Equivalently, there exists a finite scheme $Z\subset X$ of length $r$ whose $B$-saturated ideal satisfies $I_Z\subset \mathrm{Ann}(F)$. When $X$ and its embedding are fixed, the cactus rank of $F$ is denoted $\mathrm{crk}_X(F)$ (or simply $\mathrm{crk}(F)$, whenever dropping $X$ should not cause confusion).\end{definition}

We finish off this subsection recalling the notion of conciseness and minimal border rank elements. 

\begin{definition}[{\bf Conciseness}]
A homogeneous $F\in T_{\bb{v}}$ is said to be {\it concise} whenever $\Ann(F)_{{\bf a}_i} = 0$ for all $0\leq i\leq n$. 
\end{definition}

The following lemma is again well-known and motivates the next definition. 
\begin{lemma}
Let $F\in T_{\bb{v}}$ be concise. Then $\brk(F)\geq  \max_{i\in [n]}\lbrace \dim_{\CC} T_{{\bf a}_i}\rbrace$. 
\end{lemma}

\begin{definition}
Let $F\in T_{\bb{v}}$ be concise. If $\brk(F) = \max_{i\in [n]}\lbrace \dim_{\CC} T_{{\bf a}_i}\rbrace$, then $F$ is said to be of {\it minimal border rank}. 
\end{definition}

\subsection{Multigraded Hilbert schemes and border apolarity}

Given an admissible numerical function $h: \mathrm{Pic}(X)=\ZZ^s\rightarrow \NN$, let $\mathrm{Hilb}^h_S$ be the scheme represent 
the functor whose points are the homogeneous ideals $I\subset S$ with $\mathrm{HF}(S/I,\cdot) = h$. 
The functor in question is indeed representable and the scheme $\mathrm{Hilb}^h_S$ is the {\it Haiman-Sturmfels multigraded Hilbert scheme}. This natural 
object was introduced by Haiman and Sturmfels \cite{hs}. The topological properties of this scheme are somewhat wilder than its classical
cousin, i.e. the {\it Grothendieck-Hilbert scheme}, parameterizing schemes with a prescribed Hilbert polynomial. 

Since $S$ is the Cox ring of $X$ and hence is positively graded, $\mathrm{Hilb}^h_S$ is a projective scheme for any Hilbert function $h$. 
We usually ignore the scheme structure of $\mathrm{Hilb}^h_S$ and look at the underlying reduced scheme $(\mathrm{Hilb}^h_S)_{\mathrm{red}}$. 
A closed point of $\mathrm{Hilb}^h_S$ corresponds to an ideal $I\subset S$ with Hilbert function $h$: we express this membership 
in a set-theoretic fashion as $I\in \mathrm{Hilb}^h_S$. 

For any integer $r\geq 0$, define the numerical function $h_{r,X}: \mathrm{Pic}(X)\rightarrow \NN$ to be 
\[
h_{r,X}(\bb{v}) = \min\lbrace r, \dim H^0(D)\rbrace = \min\lbrace r, \dim_{\CC} S_{\bb{v}}\rbrace, 
\]
where the Cartier divisor $D$ is identified with $\bb{v}$. 
The function $h_{r,X}$ is the {\it generic Hilbert function of $r$ points on $X$} \cite[\S 3.2]{bb19}. By \cite[Lemma 3.9]{bb19} the equality $\mathrm{HF}(S/I_Z,\cdot) = h_{r,X}$ holds for a very general collection of $r$ points $Z\subset X$. 

Buczy\'nska-Buczy\'nski prove that the scheme $\mathrm{Hilb}^{h_{r,X}}_S$ contains a unique irreducible component, called $\mathrm{Slip}_{r,X}$,
that is the closure of the locus of all $I\in \mathrm{Hilb}^{h_{r,X}}_S$ that are the $B$-saturated ideals of $r$ distinct points in $X$ \cite[Proposition 3.13]{bb19}.  

We are now ready to state the following groundbreaking result of Buczy\'nska-Buczy\'nski \cite[Theorem 3.15]{bb19}, which is a recent effective tool at our disposal to estimate border ranks of forms or tensors; as already mentioned, a major achievement of this method was demonstrated 
in \cite{chl19}.

\begin{theorem}[{\bf Border apolarity of  Buczy\'nska-Buczy\'nski}]\label{mainbb}
Keep the notation from above and let $F\in T_{\bb{v}}$. Then the following assertions are equivalent: 
\begin{enumerate}

\item[(i)] The $X$-border rank of $F$ satisfies $\brk_X(F)\leq r$;

\item[(ii)] there exists a homogeneous ideal $I\subset \Ann(F)\subset S$ which lies in the irreducible component $\mathrm{Slip}_{r,X}$. 
\end{enumerate}
\end{theorem}

One of the main contributions of 
Buczy\'nska-Buczy\'nski is the realization of a (projective and so topologically compact) parameter space of border decompositions of a given $F\in T_{\bb{v}}$. These are the {\it border varieties of sums of powers}, the main characters of this article. 

\begin{definition}\label{vspbar}
Let $h_{r,X}$ be the generic Hilbert function and $F\in T_{\bb{v}}$. The {\it border variety of sums of $r$ powers} $\VSPb(F, r)$ is the closed subvariety of $\mathrm{Slip}_{r,X}$ defined as follows: 
\[
\VSPb(F,r) = \left\lbrace J\in \mathrm{Slip}_{r,X} \mbox{ such that } J\subset \mathrm{Ann}(F)\right\rbrace. 
\]
Note that this is a closed subscheme of the projective scheme $\mathrm{Slip}_{r,X}$ and so projective. Although the definition is for an arbitrary $r\in \NN$, particularly interesting and meaningful is $\VSPb(F,\brk(F))$. 
\end{definition}

\begin{definition}[{\bf Border identifiability}]\label{def:border identifiable}
Let $F\in T_{\bb{v}}$. We say that $F$ is {\it border identifiable} if $\VSPb(F,\brk(F))$ is a single point.  
\end{definition}

\section{Multigraded regularity and Hilbert schemes of toric varieties}\label{sec:Hilbert schemes}
\subsection{Maclagan-Smith construction and the map to the Hilbert scheme} 

Recall that $X$ is a smooth projective complex toric variety with Cox ring $S$, a graded ring graded by $\mathrm{Pic}(X)\cong \ZZ^s$, with irrelevant ideal $B \subseteq S$. The semigroup $\mathbf{\mathcal  K}$ denotes the points in $\ZZ^s$ corresponding
to nef divisors on $X$; let us fix a minimal generating set $\mathcal C=\{\bb{c}_1,\ldots, \bb{c}_l\}$ of $\mathbf{\mathcal  K}$, i.e.  
$\mathbf{\mathcal  K} = \NN \bb{c}_1+ \cdots+\NN \bb{c}_l$.

Given a $\CC$-algebra $R$, let $S_R$ be the ring $R\otimes_\CC S$ and $B_R$ be the extension of the irrelevant ideal, i.e. $B_R = B\cdot S_R$; define $X_R$ to be the fiber product $X\times_{\mathrm{Spec}(\CC)} \mathrm{Spec}(R)$. Given a $\ZZ^s$-graded ideal $I_R$ of $S_R$, $\overline{I_R}$ denotes the saturation of $I_R$ with respect to $B_R$. Given any $\ZZ^s$-graded $S_R$-module $M$ and a subset $\mathcal{D}\subset \ZZ^s$, $M_{|\mathcal{D}}$ is the graded $R$-module $\bigoplus_{\mathbf{v}\in \mathcal D} M_\mathbf{v}$.

As in the case $R=\mathrm{Spec}(\CC)$ explained in \cite{Cox95}, there is an exact functor $M\mapsto \widetilde{M}$ from the category of
$\ZZ^s$-graded $S_R$-modules to the category of quasicoherent sheaves on $X_R$. Given a cone $\sigma$ in the fan of $X$, let 
$U_{R,\sigma}$ be the corresponding affine open subset of $X_R$. Its coordinate ring
is $(S_R)_{(\sigma)}$--the degree zero part of the localization $(S_R)_\sigma$ of $S_R$ at the product
of variables corresponding to rays outside $\sigma$. The restriction of $\widetilde{M}$ 
to $U_{R,\sigma}$ is given by the degree zero part of 
$M\otimes_{S_R} (S_R)_{\sigma}$.
In particular, every $\ZZ^s$-homogeneous ideal $I\subset S_R$ defines a quasicoherent sheaf of ideals on $X_R$ and thus we have a corresponding closed subscheme of $X_R$.

\begin{lemma}\label{lem:saturation_defines_the_same_subscheme}
    The ideals $I$ and $\overline{I}$ define the same subscheme of $X_R$.
\begin{proof}
    We have an exact sequence $0\to I\to \overline{I} \to \overline{I}/I \to 0$ so it is enough to show that $\widetilde{\overline{I}/I}$ is the zero sheaf. This can be checked on the affine open covering 
    $\{U_{R,\sigma}\}_{\sigma}$. By definition every element of $\overline{I}/I$
    is annihilated by a certain power of $B_R$. In particular, it is annihilated by a certain power
    of the product of variables of $S$ corresponding to rays outside of $\sigma$. This product is invertible
    in $(S_R)_{\sigma}$, so $(\widetilde{\overline{I}/I})_{|U_{R,\sigma}} = 0$.
\end{proof}
\end{lemma}

Cox showed that the functor $M\mapsto \widetilde{M}$ induces a bijection between closed subschemes of $X$ and $\ZZ^s$-homogeneous $B$-saturated ideals of $S$.
The following relative version in the smooth case is well-known.

\begin{theorem}[{\cite[Corollary~3.8]{Cox95}}]\label{thm:Cox_bijection}
For any noetherian $\CC$-algebra $R$, there is a bijection between closed subschemes of $X_R$ and $\ZZ^s$-homogeneous $B_R$-saturated ideals of $S_R$. Furthermore, if $\mathcal{I}$ is the ideal sheaf of a subscheme $Z$, then the ideal corresponding to $Z$ is $\bigoplus_{\bb{v}\in \mathrm{Pic}(X)} H^0(X_R, \mathcal{I}(\bb{v}))$.
\end{theorem}

The statement generally fails in the simplicial non smooth case \cite[\S 2]{Ga23}.

\begin{definition}
Given a field extension $\CC \subseteq K$ and a finitely generated $\ZZ^s$-graded $S_K$-module $M$, the multigraded Hilbert function $\mathrm{HF}(M,-)\colon \ZZ^s \to \NN$ is defined by $\mathrm{HF}(M, \bb{v}) = \dim_{K} M_{\bb{v}}$. 
\end{definition}

\begin{theorem}[{\cite[Proposition 2.10]{MS05}}]
There exists a unique polynomial $P_M({\bf t})\in \mathbb Q[t_1,\ldots, t_s]$ such that $P_M(\bb{v}) = \mathrm{HF}(M,\bb{v})$ for all $\bb{v}\in {\bf \mathcal{K}}$ sufficiently far from the boundary of ${\bf \mathcal{K}}$. 
\end{theorem}

\noindent This polynomial $P_M({\bf t})$ is the {\it multigraded Hilbert polynomial} of $M$. 

Let $\CC-${\bf Alg} and {\bf Set} be the category of noetherian $\CC$-algebras and the category of sets, respectively.
Given a polynomial $P\in \mathbb Q[t_1,\ldots, t_s]$, the {\it toric Hilbert functor} is the functor 
\[
\mathcal{H}\mathrm{ilb}^P(X): \CC{\bf -Alg}\longrightarrow  {\bf Set}
\]
defined by
\begin{align*}
R\mapsto \{ Y\subseteq X\times_{\CC} \mathrm{Spec}(R) \ &| \ Y \mbox{ is flat over } \mathrm{Spec}(R) \mbox{ and } \mbox{ for every }r\in \mathrm{Spec}(R) \\
& \mbox{ the multigraded Hilbert polynomial of } S_{\kappa(r)} / I_{Y_r} \mbox{ is } P\} 
\end{align*}
where $Y_r$ is the scheme-theoretic fiber of $Y$ over $r\in \mathrm{Spec}(R)$, $\kappa(r)$ is the residue field of $r$ and $I_{Y_r}$ is the unique $B_{\kappa(r)}$-saturated $\ZZ^s$ homogeneous ideal of $S_{\kappa(r)}$ defining $Y_r$ (see Theorem~\ref{thm:Cox_bijection}). Maclagan and Smith showed that the toric Hilbert functor is represented by a projective scheme over $\CC$ \cite[Theorem 6.2]{MS05}. We denote the scheme representing it $\mathrm{Hilb}^P(X)$. 

In the $\mathbb{Z}$-graded case, the saturation of an ideal is determined by the intersection of this ideal with any power of the irrelevant ideal. Similarly, we have the following fact.

\begin{lemma}[{\cite[Lemma~6.8]{MS}}]\label{lem:equality_of_sat}
    Let $R$ be a noetherian $\CC$-algebra. For every $\ZZ^s$-graded ideal $I \subseteq S_R$ and every $\bb{u}\in \mathbf{\mathcal  K}$, the ideals $\overline{I}$ and $\overline{(I_{|\bb{u}+\mathbf{\mathcal  K}})} $ coincide. 
\end{lemma}
\begin{proof}
    Let $J = (I_{|\bb{u}+\mathbf{\mathcal  K}})$. Since $J\subseteq I$ we have $\overline{J} \subseteq \overline{I}$. Assume that $f\in \overline{I}$ is a homogeneous element. There exists a positive integer $k$ such that $fg \in I$ for every $g\in B_R^k$. Let $\{g_1, \ldots, g_p\}$ be a set of homogeneous generators of $B_R^k$.
    Let $\bb{v}\in \mathbf{\mathcal  K}$ be such that $\deg (f) + \bb{v} + \deg(g_i) \in \bb{u}+\mathbf{\mathcal  K}$ for every $i$. 
    
    We claim that there is $h\in (S_R)_{\bb{v}}$ that is a nonzerodivisor on $S_R/\overline{J}$. Otherwise, the union of degree $\bb{v}$ parts of all the associated primes of $S_R/\overline{J}$ would be $(S_R)_{\bb{v}}$. Since there are only finitely many associated primes, by prime avoidance we conclude that there exists an associated prime $P$ such that $P_\bb{v} = (S_R)_\bb{v}$. From 
    \cite[Lemma~2.4]{MS} we conclude that $B_R \subseteq P$ which contradicts the assumption that $\overline{J}$ is $B_R$-saturated.  
    
    Let $g\in B_R^k$. We have $fgh \in (I_{|\bb{u}+\mathbf{\mathcal  K}})$ so $fh \in \overline{J}$. Since $h$ is a nonzerodivisor on $S_R/\overline{J}$ we conclude that $f\in \overline{J}$.
\end{proof}

The multigraded Castelnuovo-Mumford regularity of Maclagan and Smith \cite[Definition 1.1]{MS} has the following definition. 

\begin{definition}[{\bf Multigraded regularity}]\label{def:CM-regularity of MS}
Let $\CC\subseteq K$ be a field extension, $M$ be a graded $S_K$-module and fix $\mathcal C$ as above. Given $\bb{m}\in \mathrm{Pic}(X) \cong \ZZ^s$, one says that $M$ is $\bb{m}$-regular if its graded local cohomology module satisfies the vanishing $H^i_{B_K}(M)_\bb{p}=0$ 
for all of the following $i$ and $\bb{p}$: 
\begin{enumerate}
\item[(i)] $i\geq 1$ and 
$\bb{p}$ is of the form $\bb{p}=\bb{m}-\lambda_1\bb{c}_1-\cdots-\lambda_{l}\bb{c}_{l}+\bb{u}$,
where the coefficients $\lambda_j\in \NN$ satisfy $\lambda_1+\cdots 
+ \lambda_{l}=i-1$ and $\bb{u}\in \mathbf{\mathcal  K}$. 

\item[(ii)] $i=0$ and $\bb{p}$ is of the form $\bb{p} = \bb{m} + \bb{c}_j + \bb{u}$, where $1\leq j\leq l$ and $\bb{u}\in \mathbf{\mathcal  K}$. 

\end{enumerate}

Define the {\it multigraded regularity} of $M$, $\mathrm{reg}(M)\subset \mathrm{Pic}(X)\cong \ZZ^s$, to be the set 
\[
\mathrm{reg}(M) = \lbrace \bb{m}\in \ZZ^s \ | \ M \mbox{ is $\bb{m}$-regular} \rbrace.
\]

\end{definition}

\noindent The following uniform bound on multigraded regularity by Maclagan and Smith is fundamental in their construction of the Hilbert scheme $\mathrm{Hilb}^P(X)$ recalled below.

\begin{theorem}[{\cite[Theorem~4.11]{MS05}}]\label{thm:uniform_bound}
There exists $\bb{k}\in \mathbf{\mathcal{K}}$ such that for every field extension $\CC\subseteq K$ and every $\ZZ^s$-homogeneous  $B_K$-saturated ideal of $S_K$ defining a subscheme of $X_K$ with multigraded Hilbert polynomial $P$ is $\bb{k}$-regular.
\end{theorem}

\begin{lemma}\label{lem:Hilbert_polynomials_of_the_family}
  Let $\bb{k}\in \mathbf{\mathcal{K}}$ be as in Theorem~\ref{thm:uniform_bound}. Let $R$ be a noetherian $\CC$-algebra and let $Z\in \mathcal{H}ilb^P(X)(R)$. If $I_Z$ is the unique $B_R$-saturated ideal defining $Z$, then for every $\bb{u}\in \bb{\mathcal{K}}$ the $R$-module
  $(S_R/I_Z)_{\bb{u+k}}$ is locally free of rank $P(\bb{u}+\bb{k})$.
\end{lemma}
\begin{proof}
  Let $\pi\colon X_R\to \mathrm{Spec}(R)$ denote the projection. We start with showing that $H^i(X_R, \mathcal{I}_Z(\bb{u}+\bb{k})) = 0$ for all $i>0$. Assume initially that $R=K$ for a field extension $\CC\subseteq K$. It follows from the definition of $\bb{k}$ and \cite[Proposition~6.4]{MS} that $\mathcal{I}_Z$ is $(\bb{u}+\bb{k})$-regular. In particular, $H^i(X_K, \mathcal{I}_Z(\bb{u}+\bb{k})) = 0$ for all $i>0$.

  Since $R^i\pi_*(\mathcal{I}_Z(\bb{u}+\bb{k})) \cong \widetilde{H^i(X_R, \mathcal{I}_Z(\bb{u}+\bb{k}))}$ (see \cite[Proposition~III.8.5]{h}), it follows from \cite[Theorem~III.12.11]{h} that 
  the fiber of $R^i\pi_*(\mathcal{I}_Z(\bb{u}+\bb{k}))$ over every point of $\mathrm{Spec}(R)$ is zero. By Nakayama's lemma we get that the localization of $H^i(X_R, \mathcal{I}_Z(\bb{u}+\bb{k}))$ at every prime ideal of $R$ is the zero $R$-module. Hence the module itself is zero. \\
\noindent In particular, using the above for $i=1$ we get $(S_R/I_Z)_{\bb{u}+\bb{k}}= H^0(X_R, \mathcal{O}_Z(\bb{u}+\bb{k}))$. Similarly, we show that $H^i(X_R, \mathcal{O}_{X_R}(\bb{u}+\bb{k})) = 0$ for all $i>0$. For $R=K$ we use \cite[Corollary~3.6]{MS} and then proceed as for $\mathcal{I}_Z(\bb{u}+\bb{k})$.

We conclude that $H^i(X_R, \mathcal{O}_{Z}(\bb{u}+\bb{k})) = 0$ for all $i>0$. Let $\mathfrak{U}$ be the affine cover of $X_R$ obtained by pulling back the standard affine cover of $X$ corresponding to full-dimensional cones in the fan of $X$.
By the above the \v{C}ech complex $\mathcal{C}^\bullet(\mathfrak{U}, \mathcal{O}_Z(\bb{u}+\bb{k}))$ is a free resolution of $(S_R/I_Z)_{\bb{u}+\bb{k}}$. Each $\mathcal{C}^i(\mathfrak{U}, \mathcal{O}_Z(\bb{u}+\bb{k}))$ is a flat $R$-module since $\mathcal{O}_Z(\bb{u}+\bb{k})$ is flat over $\mathrm{Spec}(R)$. Therefore, by splitting the resolution into short exact sequences and using \cite[Proposition~III.9.1A]{h} we conclude that $(S_R/I_Z)_{\bb{u}+\bb{k}}$ is a flat $R$-module. Since it is finitely generated it is locally free. 

We are left with showing that $(S_R/I_Z)_{\bb{u}+\bb{k}}$ is locally free of rank $P(\bb{u}+\bb{k})$. Since tensor product preserves the rank of a locally free module we may reduce to the case $R=K$ for a field extension $\CC\subseteq K$.
By \cite[Proposition~2.14]{MS05} it is sufficient to show that $H^i_{B_K}(S_K/I_Z)_{\bb{u}+\bb{k}} = 0$ for every $i$. Using \cite[Corollary~3.6]{MS} and the long exact sequence of local cohomology groups, we reduce the problem to showing that
$H^{i+1}_{B_K}(I_Z)_{\bb{u}+\bb{k}} = 0$ for every $i$. This follows from the definition of $\bb{k}$ given in Theorem~\ref{thm:uniform_bound} and the definition of $\bb{k}$-regularity. 
\end{proof}

\begin{lemma}\label{lem:flat_ring_implies_flat_scheme}
    If $I$ is a $\ZZ^s$-homogeneous ideal of $S_R$ such that $S_R/I$ is a flat $R$-module, then the subscheme of $X_R$ defined by $I$ is flat over $\mathrm{Spec}(R)$.
\begin{proof}
    Due to the exact sequence $0\to I\to S_R \to S_R/I \to 0$, it is enough to show that $\widetilde{S_R/I}$ is a flat sheaf over $\mathrm{Spec}(R)$. It is sufficient to check this after restricting to an affine open subset $U_{R,\sigma}$.
    The localization functor is exact and $S_R/I$ is flat over $R$. It follows that $(S_R/I) \otimes_{S_R} (S_R)_{\sigma}$ is flat over $R$ and therefore so is its degree zero part.
\end{proof}
\end{lemma}

Given a subset $\mathcal D\subset \ZZ^s$ and $S=S[X]$ the Cox ring of $X$, let $F_{\mathcal D} = \bigcup_{\bb{u},\bb{v}\in \mathcal D} F_{\bb{u},\bb{v}}$, where $F_{\bb{u},\bb{v}}$ is the set of multiplication maps $S_{\bb{u}}\rightarrow S_{\bb{v}}$ by a monomial in $S_{\bb{v}-\bb{u}}$. Given a $\CC$-algebra $R$, let $R\otimes S_{\mathcal D}$ be the 
graded $R$-module $\bigoplus_{\bb{v}\in \mathcal D} R\otimes_{\CC} S_\bb{v}$ equipped with maps $F^R_{\bb{u},\bb{v}}=1_R\otimes F_{\bb{u},\bb{v}}$. A homogeneous submodule $L = \bigoplus_{\bb{v}\in \mathcal D} L_\bb{v}\subset R\otimes S_{\mathcal D}$ is an $F$-submodule if it satisfies $F^R_{\bb{u},\bb{v}}(L_\bb{u})\subset L_{\bb{v}}$ for 
all $\bb{u},\bb{v}\in \mathcal D$.

Let $h: \mathcal D\rightarrow \NN$ be a function. The {\it multigraded Hilbert functor} of Haiman and Sturmfels is the functor
\[
\mathcal{H}\mathrm{ilb}^h_{S_{\mathcal D}}: \CC{\bf -Alg}\longrightarrow  {\bf Set}
\]
defined by 
\begin{align*}
R\mapsto  \{ J\subset R\otimes S_{\mathcal D} \ &| \ 
J \mbox{ is an } F \mbox{-submodule}, R\otimes S_\bb{v}/J_\bb{v} \mbox{ is a locally free } 
 \\ & \mbox{ } R \mbox{-module of rank } h(\bb{v}), \forall \bb{v}\in \mathcal D\}.
\end{align*}
This functor is represented by a projective scheme over $\CC$ \cite[\S6.1]{hs}. The scheme representing it when $\mathcal D = \ZZ^s$ and $h=h_{r,X}$ is the multigraded Hilbert scheme $\mathrm{Hilb}^h_S$ introduced before. For any $\mathcal D\subset \ZZ^s$, when the function $h: \mathcal D\rightarrow \NN$ is evaluating a polynomial $P$, then the representing scheme above is denoted $\mathrm{Hilb}^{P}_{S_{\mathcal D}}$. 

We prove the following theorem, whose proof was only sketched in \cite{MS05}, as it is of central 
importance in our approach. 

\begin{theorem}[{\bf Maclagan-Smith}]\label{sat vs. truncation}
There exists $\bb{k}\in {\bf {\mathcal K}}$ such that the functors 
$\mathcal{H}\mathrm{ilb}^P(X)$ and $\mathcal{H}\mathrm{ilb}^P_{S_{\bb{k}+\bb{{\mathcal K}}}}$ are naturally isomorphic. In particular, the schemes $\mathrm{Hilb}^{P}_{S_{\bb{k}+\bb{\mathcal K}}}$ and $\mathrm{Hilb}^P(X)$ are isomorphic. 
\begin{proof} 
Let $\bb{k}$ be as in Theorem~\ref{thm:uniform_bound} and $R$ be a noetherian $\CC$-algebra. By Lemma~\ref{lem:flat_ring_implies_flat_scheme} there is a map of sets 
\[
 \mathcal{H}\mathrm{ilb}^P_{S_{\bb{k}+\bb{\mathcal K}}}(R)\stackrel{\tau_R}{\longrightarrow}\mathcal{H}\mathrm{ilb}^P(X)(R), 
\]
defined by $\tau_R(I) = Z$, where $Z$ is the subscheme defined by $I$.
By Lemma~\ref{lem:Hilbert_polynomials_of_the_family} we have a map of sets
\[
\mathcal{H}\mathrm{ilb}^P(X)(R)\stackrel{\rho_R}{\longrightarrow}\mathcal{H}\mathrm{ilb}^P_{S_{\bb{k}+{\mathcal K}}}(R)
\]
defined by $\rho_R(Z) = (I_Z)_{|\bb{k}+\bb{\mathcal K}}$ where $I_Z$ is the unique $\ZZ^s$-homogeneous $B_R$-saturated ideal of $S_R$ defining $Z$.
It follows from Lemma~\ref{lem:saturation_defines_the_same_subscheme} and Lemma~\ref{lem:equality_of_sat} that $\tau_R\circ \rho_R(Z) = Z$ for every $Z$. Let $I \in \mathcal{H}\mathrm{ilb}^P_{S_{\bb{k}+\bb{\mathcal K}}}(R)$ and denote $\tau_R(I)$ by $Z$.
By Lemma~\ref{lem:saturation_defines_the_same_subscheme} and Theorem~\ref{thm:Cox_bijection} we have $\overline{I}  = I_Z$. 
Hence $I \subseteq \rho_R\circ \tau_R(I)$. Since they have the same Hilbert functions they are equal.
These maps are inverses of each other for every noetherian $\CC$-algebra $R$. Furthermore, they are natural in $\mathrm{Spec}(R)$ by \cite[Proposition~III.9.3]{h}. Hence the two functors are isomorphic. Therefore so are the schemes representing them.
\end{proof}
\end{theorem}

\begin{theorem}\label{there exists a morphism}
Let $S = S[X]$ be the Cox ring of a smooth projective toric variety with $\mathrm{Pic}(X)\cong \ZZ^s$. Let $h: \ZZ^s\rightarrow \NN$ be an admissible Hilbert function that coincides for all $\bb{v}\in \ZZ^s$ sufficiently far from the boundary of $\mathcal K$ with a polynomial $P({\bf t})\in \mathbb{Q}[t_1,\ldots, t_s]$. 
Then there exists a morphism 
\[
\psi_{\bb{k}+\bb{\mathcal K}}: \mathrm{Hilb}^h_{S}\longrightarrow \mathrm{Hilb}^P(X). 
\]
\begin{proof}
Let $\tau_{\bb{k}+\bb{\mathcal K}}$ be the isomorphism of schemes $\mathrm{Hilb}^{P}_{S_{\bb{k}+\bb{\mathcal K}}}\cong \mathrm{Hilb}^P(X)$ established in Theorem \ref{sat vs. truncation}. Then $\psi_{\bb{k}+\bb{\mathcal K}} =  \tau_{\bb{k}+\bb{\mathcal K}} \circ \mathrm{Pr}_{|\bb{k}+\bb{\mathcal K}}$, where $\mathrm{Pr}_{|\bb{k}+\bb{\mathcal K}}$ is the projection
of an ideal $I\in \mathrm{Hilb}^h_{S}$ to the graded vector space spanned by the degree $\bb{k}+\bb{\mathcal K}$ pieces of $I$. This gives the desired morphism. 
\end{proof}
\end{theorem}

\begin{proposition}\label{Tangent space to Hilb_k+K}
The tangent space at a graded vector space $I_{|\bb{k}+\bb{\mathcal{K}}}\in \mathrm{Hilb}^{P}_{S_{\bb{k}+\bb{\mathcal K}}}$ to the scheme  $\mathrm{Hilb}^{P}_{S_{\bb{k}+\bb{\mathcal K}}}$
is isomorphic to the space of degree zero $F$-module homomorphisms
\[
\mathrm{Hom}_F(I_{|\bb{k}+\bb{\mathcal{K}}}, S_{|\bb{k}+\bb{\mathcal{K}}}/I_{|\bb{k}+\bb{\mathcal{K}}})_0,
\]
i.e. $\CC$-linear maps $\phi$ preserving the degree such that for every $\bb{v}, \bb{u}\in \bb{k}+\bb{\mathcal{K}}$, every monomial $G$ in $S_{\bb{v}-\bb{u}}$ and every element $x\in I_{\bb{u}}$ we have $\phi(Gx) = G\phi(x)$.
\begin{proof}
Let $R=\CC[\epsilon]/(\epsilon^2)$. 
The tangent space to $\mathrm{Hilb}^{P}_{S_{\bb{k}+\bb{\mathcal K}}}$ is by definition given by evaluating 
its functor of points on the local $\CC$-algebra $R$. Let $\phi\in \mathrm{Hom}_F(I_{|\bb{k}+\bb{\mathcal K}}, S_{|\bb{k}+\bb{\mathcal K}}/I_{|\bb{k}+\bb{\mathcal K}})_0$ be an $F$-module homomorphism. 
Define 
\[
I'_{|\bb{k}+\bb{\mathcal K}} = \left\lbrace x + \epsilon y \ | \ x\in I_{|\bb{k}+\bb{\mathcal K}}, y\in S_{|\bb{k}+\bb{\mathcal K}} \mbox{ and } \overline{y} = \phi(x)\in S_{|\bb{k}+\bb{\mathcal K}}/I_{|\bb{k}+\bb{\mathcal K}}\right\rbrace. 
\]
Note that $I'_{|k+\mathcal K}$ is an $F$-submodule of $R\otimes  S_{|\bb{k}+\bb{\mathcal K}}$. The definitions show that there is a short exact sequence of graded $R$-modules $0\rightarrow I_{|\bb{k}+\bb{\mathcal K}}\stackrel{\cdot \epsilon}{\rightarrow} I'_{|\bb{k}+\bb{\mathcal K}}\rightarrow I_{|\bb{k}+\bb{\mathcal K}}\rightarrow 0$. 
This exactness implies that the following sequence of graded $R$-modules is exact: 
\[
0\longrightarrow S_{|\bb{k}+\bb{\mathcal K}}/I_{|\bb{k}+\bb{\mathcal K}}\stackrel{\cdot \epsilon}{\longrightarrow} (R\otimes S_{|\bb{k}+\bb{\mathcal K}})/I'_{|\bb{k}+\bb{\mathcal K}}\longrightarrow S_{|\bb{k}+\bb{\mathcal K}}/I_{|\bb{k}+\bb{\mathcal K}}\longrightarrow 0. 
\]
It follows from \cite[Proposition~2.2]{h10} that $(R\otimes S_{|\bb{k}+\bb{\mathcal K}})/I'_{|\bb{k}+\bb{\mathcal K}}$ is a flat $R$-module. Therefore, each graded summand $((R\otimes S_{|\bb{k}+\bb{\mathcal K}})/I'_{|\bb{k}+\bb{\mathcal K}})_\bb{v}$ is a locally free $R$-module of rank equal to the dimension of the corresponding graded summand of $S_{|\bb{k}+\bb{\mathcal K}}/I_{|\bb{k}+\bb{\mathcal K}}$, i.e. $P(\bb{v})$.

Conversely, suppose we are given an $F$-submodule $I'_{|\bb{k}+\bb{\mathcal K}}\subset R\otimes S_{|\bb{k}+\bb{\mathcal K}}$ such that each graded summand of $(R\otimes S_{|\bb{k}+\bb{\mathcal K}})/I'_{|\bb{k}+\bb{\mathcal K}}$ is a locally free $R$-module of the required rank and such that 
its image in $S_{|\bb{k}+\bb{\mathcal K}}$ (under the quotient map by $(\epsilon)$) is $I_{|\bb{k}+\bb{\mathcal K}}$. By \cite[Proposition~2.2]{h10}, this amounts to say that we have the short exact sequence of graded $R$-modules $0\rightarrow S_{|\bb{k}+\bb{\mathcal K}}/I_{|\bb{k}+\bb{\mathcal K}}\stackrel{\cdot \epsilon}{\rightarrow} (R\otimes S_{|\bb{k}+\bb{\mathcal K}})/I'_{|\bb{k}+\bb{\mathcal K}}\rightarrow S_{|\bb{k}+\bb{\mathcal K}}/I_{|\bb{k}+\bb{\mathcal K}}\rightarrow 0$. This exactness implies that the following sequence of graded $R$-modules is exact: 
\[
0\longrightarrow I_{|\bb{k}+\bb{\mathcal K}}\stackrel{\cdot \epsilon}{\longrightarrow} I'_{|\bb{k}+\bb{\mathcal K}}\longrightarrow I_{|\bb{k}+\bb{\mathcal K}}\longrightarrow 0.
\]

Hence two liftings $x+\varepsilon y,x+\varepsilon y'$ of degree $\bb{v}\in {\bb{k}+\bb{\mathcal K}}$ of a given element $x\in I_{\bb{v}}$ differ by an element in $\epsilon I_{\bb{v}}$. Then we may define $\phi(x) = \overline{y}=\overline{y'}\in S_{|\bb{k}+\bb{\mathcal K}}/I_{|\bb{k}+\bb{\mathcal K}}$, which is a degree zero $F$-module homomorphism. 

The two constructions above are inverses of each other, showing a bijection between the sets $\mathcal{H}\mathrm{ilb}^{P}_{S_{\bb{k}+\bb{\mathcal K}}}(R)$ and 
$\mathrm{Hom}_F(I_{|\bb{k}+\bb{\mathcal K}}, S_{|\bb{k}+\bb{\mathcal K}}/I_{|\bb{k}+\bb{\mathcal K}})_0$.
One can check that when $\mathcal{H}\mathrm{ilb}^{P}_{S_{\bb{k}+\bb{\mathcal K}}}(R)$ is given its natural $\CC$-vector space structure 
described in \cite[\S VI.1.3]{hs}, then this bijection is a $\CC$-vector space isomorphism.
\end{proof}
\end{proposition}

By the description of the tangent space at $I_{|\bb{k}+\bb{\mathcal K}}$ to $\mathrm{Hilb}^{P}_{S_{k+\mathcal K}}$ given in Proposition \ref{Tangent space to Hilb_k+K} and by the well-known description of the tangent space at $I$ to $\mathrm{Hilb}^h_{S}$ (see \cite[Proposition~1.6]{hs}) we obtain the following.

\begin{lemma}\label{tangentmap}
The tangent map of the morphism  $\mathrm{Pr}_{|{\bb{k}+\bb{\mathcal K}}}: \mathrm{Hilb}^h_{S}\longrightarrow \mathrm{Hilb}^{P}_{S_{\bb{k}+\bb{\mathcal K}}}$ at a point $I$
is the map of vector spaces 
\[
(d \mathrm{Pr}_{|{\bb{k}+\bb{\mathcal K}}})_I: \mathrm{Hom}_S (I, S/I)_0 \longrightarrow \mathrm{Hom}_F(I_{|\bb{k}+\bb{\mathcal K}}, S_{|\bb{k}+\bb{\mathcal K}}/I_{|\bb{k}+\bb{\mathcal K}})_0
\]
given by restriction of the corresponding $S$-module map $\phi\in \mathrm{Hom}_S (I, S/I)_0$. 
\end{lemma}

The next ingredient we are going to need is the locus of saturated ideals inside the multigraded Hilbert scheme $\mathrm{Hilb}^h_S$. 

\begin{definition}[{\bf Saturable locus}] 
The subset of closed points of $\mathrm{Hilb}^h_S$ 
corresponding to $B$-saturated ideals $I\subset S$ is denoted
$\mathrm{Hilb}^{h,\mathrm{sat}}_S$. Its closure is the {\it saturable locus}. 
\end{definition}

\begin{theorem}[{\bf Buczy\'nska-Buczy\'nski} \cite{bb21}, {\bf Jelisiejew-Ma\'ndziuk} \cite{JM}]\label{satlocus is open}
Let $h$ be an admissible Hilbert function for $S$. Then the locus $\mathrm{Hilb}^{h,\mathrm{sat}}_S$
is a Zariski open subset of the multigraded Hilbert scheme $\mathrm{Hilb}^h_S$. 
\end{theorem}

The following is a slight generalization of \cite[Proposition 3.9]{JM} to the case
of smooth projective toric varieties. 

\begin{theorem}\label{thm:restriction on saturable locus is closed immersion}
Let $h$ be the Hilbert function of the quotient algebra of a $B$-saturated ideal of $S$ and $P$ be the corresponding multigraded Hilbert polynomial. Then the projection morphism $\psi_{\bb{k}+\bb{\mathcal K}}: \mathrm{Hilb}^{h}_{S}\longrightarrow \mathrm{Hilb}^P(X)$ restricts to a locally closed immersion $\mathrm{Hilb}^{h,\mathrm{sat}}_S \to \mathrm{Hilb}^P(X)$.
\begin{proof}
Since $ \tau_{\bb{k}+\bb{\mathcal K}}$ is an isomorphism, it is enough to prove that the restriction of $\mathrm{Pr}_{|{\bb{k}+\bb{\mathcal K}}}$ to  $\mathrm{Hilb}^{h,\mathrm{sat}}_S$ is a locally closed immersion.
By Theorem \ref{satlocus is open} the subset $U = \mathrm{Hilb}^{h,\mathrm{sat}}_S$ is Zariski open. Let $Z$ be the complement 
of $U\subset  \mathrm{Hilb}^{h}_{S}$. Then the restriction of $\mathrm{Pr}_{|{\bb{k}+\bb{\mathcal K}}}$ to $U$ gives a map
\[
\mathrm{Pr}_{U |{\bb{k}+\bb{\mathcal K}}}: U\longrightarrow \mathrm{Hilb}^{P}_{S_{\bb{k}+\bb{\mathcal K}}}\setminus \mathrm{Pr}_{|\bb{k}+\bb{\mathcal K}}(Z). 
\]
We prove that $\mathrm{Pr}_{U |{\bb{k}+\bb{\mathcal K}}}$  is a closed immersion. By \cite[Proposition~12.94]{GW}, it is enough to show that
\begin{itemize}
    \item $\mathrm{Pr}_{U |{\bb{k}+\bb{\mathcal K}}}$ is injective on closed points;
    \item the tangent map to $\mathrm{Pr}_{U |{\bb{k}+\bb{\mathcal K}}}$ at every closed point is injective. 
\end{itemize}
Let $J\in U$ be in the fiber of $I_{|\bb{k}+\bb{\mathcal K}}$ for some $I\in U$. Then we have the inclusion
of ideals $(I_{|\bb{k}+\bb{\mathcal K}})\subset J=\overline{J}$. Hence $\overline{(I_{|\bb{k}+\bb{\mathcal K}})} = \overline{I} = I\subset \overline{J}=J$. Since $I$ and $J$ have the same Hilbert function $h$, we conclude that $I=J$.

Let $\phi\in \ker((d \mathrm{Pr}_{U |{\bb{k}+\bb{\mathcal K}}})_J)$ where $(d \mathrm{Pr}_{U |{\bb{k}+\bb{\mathcal K}}})_J$ is the tangent map at $J$, featured in Lemma~\ref{tangentmap}, and let $\bb{u}\in \ZZ^s$. Then there exists $\bb{v}\in \mathcal K$ such that $\bb{u}+\bb{v}\in {\bb{k}+\bb{\mathcal K}}$, because the cone $\mathcal K\otimes_{\mathbb Z} \mathbb{R}$ is full-dimensional in $\mathrm{Pic}(X)$ \cite[Proposition 6.3.24]{CLS11}. Let $h\in S_{\bb{v}}$ be a nonzerodivisor on $S/I$. This exists in every degree $\bb{v}\in \mathcal K$ because $I = \overline{I}$ (see the proof of Lemma~\ref{lem:equality_of_sat}). For every $g\in I_{\bb{u}}$, we have $\phi(hg) = 0$, because the restriction of $\phi$ is the zero map. Since $\phi$ is an $S$-module map, this implies $h\phi(g) = 0$. Hence $\phi(g) = 0$ for all $g\in I_{\bb{u}}$, for $h$ is nonzerodivisor. Thus $\phi$ is the zero map itself, establishing the desired injectivity.  
\end{proof}
\end{theorem}

\subsection{The morphism from $\mathrm{Slip}_{r,X}$ to $\mathrm{Hilb}_{sm}^r(X)$}

\begin{proposition}\label{surjective map from Slip to Hilb^r}
Let $h=h_{r,X}$ be the generic Hilbert function and $P$ be the constant polynomial $P(\bb{t}) = r$. 
Then morphism $\psi_{\bb{k}+\bb{\mathcal K}}: \mathrm{Hilb}^h_{S}\longrightarrow \mathrm{Hilb}^r(X)$
descends to a surjective morphism 
\[
\phi_{r,X}: \mathrm{Slip}_{r,X}\longrightarrow \mathrm{Hilb}_{sm}^r(X), 
\]
where $\mathrm{Hilb}_{sm}^r(X)$ is the smoothable component. 
\begin{proof}
By \cite[Theorem 1.4]{bb21}, there exists a dense open subset $U$ of $r$-tuples 
$(x_1,\ldots, x_r)\in X^{\times r}$ (with distinct entries) such that the Hilbert function of
the ideal of $Z = \lbrace x_1,\ldots, x_r\rbrace$ is $\mathrm{HF}(S/I_Z) = h_{r,X}$. 
Let $X^{(r)}$ be the quotient $X^r/\mathfrak S_r$. This is a geometric quotient 
which is a projective variety since $X$ is itself projective \cite[Theorem 7.1.2]{Fan05}. We have a surjective projective morphism $\pi_{\mathfrak S_r}: X^{\times r}\rightarrow X^{(r)}$. 
Then $\pi_{\mathfrak S_r}(U)$ is constructible and so it contains an open dense subset $W$ of its closure 
$\overline{\pi_{\mathfrak S_r}(U)} = \overline{\pi_{\mathfrak S_r}(\overline{U})} =\pi_{\mathfrak S_r}(\overline{U}) = X^{(r)}$. 
Let $\pi_{\mathrm{HC}}: \mathrm{Hilb}^r(X)_{red}\rightarrow X^{(r)}$ be the Hilbert-Chow morphism 
\cite[\S 7.1]{Fan05}. Thus $V=\pi_{\mathrm{HC}}^{-1}(W)\subset \mathrm{Hilb}^r_{sm}(X)$ is a dense open subset.
The morphism given in Proposition \ref{there exists a morphism} induces a map $\phi_{r,X}: \mathrm{Slip}_{r,X}\rightarrow \mathrm{Hilb}^r_{sm}(X)$, by the construction of $\mathrm{Slip}_{r,X}$. 
Moreover, the image of $\phi_{r,X}$ contains $V$ and is therefore dense. Since $\phi_{r,X}$ is projective, it is surjective. 
\end{proof}
\end{proposition}

\begin{lemma}\label{irreducibility of saturablelocus}
If $\mathrm{Hilb}^r(X)$ is irreducible, any $B$-saturated ideal $I\in \mathrm{Hilb}^{h}_S$ is in $\mathrm{Slip}_{r,X}$. Consequently, if $\mathrm{Hilb}^r(X)$ is irreducible then $\mathrm{Slip}_{r,X}$ is the saturable locus. 
\end{lemma}
\begin{proof}
By definition, $\mathrm{Slip}_{r,X}$ is an irreducible component of the saturable locus. To obtain the other inclusion it is sufficient to show the first sentence. Consider the point $\psi_{\bb{k}+\bb{\mathcal K}}(I) \in \mathrm{Hilb}^r(X)$. By assumption, $\psi_{\bb{k}+\bb{\mathcal K}}(I) \in \mathrm{Hilb}_{sm}^r(X)$. By the surjectivity proven in Proposition \ref{surjective map from Slip to Hilb^r}, there exists $J\in \mathrm{Slip}_{r,X}$ with $\psi_{\bb{k}+\bb{\mathcal K}}(J) = \psi_{\bb{k}+\bb{\mathcal K}}(I)$. This implies that $\overline{J} = \overline{I} = I$. Since $I$ and $J$ have the same Hilbert function, it follows that $I=J$ and therefore $I\in \mathrm{Slip}_{r,X}$.
\end{proof}

\subsection{Limiting schemes}
The next theorem is a consequence of Theorem \ref{mainbb}; before we state it, we have to recall the notion of {\it limiting scheme}.

\begin{definition}[{\bf Limiting scheme}]\label{borderdec}\cite[Definition 5.1]{HMV}
Let $F\in T_{\bb{v}}$ and $r=\brk_X(F)$. Suppose we are given a border rank decomposition for $F$, i.e. 
\begin{equation}\label{brdec}
F = \lim_{t\rightarrow 0}\frac{1}{t^s}\left(p_1(t) + \cdots  +p_{r}(t) \right), \ s\geq 0, 
\end{equation}
where $p_j(t)$ are points of $X\subset \PP(T_{\bb{v}})$ \cite[\S 5.2.4]{Lands11}. The reduced zero-dimensional scheme whose (closed) points are the $p_j(t)$ is denoted $Z(t)$. 
 For each $t\neq 0$, the $B$-saturated radical ideal defining $Z(t)$ is denoted $I_{Z(t)}$. The flat limit $Z = \lim_{t\rightarrow 0} Z(t)$ is called the {\it limiting scheme} of \eqref{brdec}.
 \end{definition}

\begin{theorem} [{\cite[Theorems 5.3 and 5.4]{HMV}}]\label{borderdecgivespointinVSPb}
Let $F\in T_{\bb{v}}$ and $r=\brk_X(F)$.
The saturation of any ideal in $\VSPb(F,r)$ coincides with the ideal of a limiting scheme of a border rank decomposition. Vice versa, if $F\in T_\bb{v}$ is concise and of minimal border rank, then every border rank decomposition of $F$ determines 
an ideal in $\VSPb(F,r)$. 
\begin{proof}
Any ideal $J\in \VSPb(F,r)$ comes from some border rank decomposition \eqref{brdec}; see the proof of \cite[Theorem 3.15]{bb19}. Such a border rank decomposition \eqref{brdec} determines a family of zero-dimensional schemes $Z(t)$, each of length $r$, such that their ideals $I_{Z(t)}$ have 
the generic Hilbert function $h_{r,X}$ and $J =  \lim_{t \rightarrow 0} I_{Z(t)}$. 
Using Theorem~\ref{sat vs. truncation} to identify the Hilbert scheme $\mathrm{Hilb}^r(X)$ with a multigraded Hilbert scheme, we claim that 
\[
J_{|\bb{k}+\mathcal{K}} = \lim_{t\to 0} ({I_{Z(t)}}_{|\bb{k}+\mathcal{K}}) = {(I_{\lim_{t\to 0 }Z(t)})}_{|\bb{k}+\mathcal{K}} = {I_Z}_{|\bb{k}+\mathcal{K}}.
\]
The first equality follows from the fact that limits of ideals are computed degree by degree and the last follows from the definition of $Z$. We explain the middle one. The identification of the multigraded Hilbert scheme with the Hilbert scheme takes an ideal $I$ to the subscheme of $X$ defined by $I$ and takes a subscheme $Z$ of $X$ to $I_{|\bb{k}+\mathcal{K}}$ where $I$ is the unique $B$-saturated ideal defining $Z$. Therefore, the limit of ${I_{Z(t)}}_{|\bb{k}+\mathcal{K}}
$ is the degree $(\bb{k}+\mathcal{K})$-part of the unique $B$-saturated ideal defining the limit of $Z(t)$.
From the displayed equalities and Lemma~\ref{lem:equality_of_sat} we get $\overline{J} = I_Z$.

For the second statement, given any border rank decomposition $F = \lim_{t\rightarrow 0} H(t) $, where $H(t) = \frac{1}{t^s}\left(p_1(t) + \cdots  + p_r(t)\right)$ has $X$-rank $r$,
one has that $I_{Z(t)}\subset \mathrm{Ann}(H(t))\subset S$ ($t\neq 0$) by multigraded apolarity. Since $F$ is concise, we can find a neighborhood $U$ of $0$ such that $H(t)$ is concise for all $t\in U$. For such $t$, the ideal $I_{Z(t)}$ has the generic Hilbert function $h_{r,X}$. The limit ideal $J:= \lim_{t\rightarrow 0} I_{Z(t)}$ is in $\Ann(F)$ by construction because 
$\lim_{t\rightarrow 0} \Ann(H(t))\subset \Ann(F)$. To see the last inclusion, let $p = \lim_{t\rightarrow 0} p_t \in \lim_{t\rightarrow 0} \Ann(H(t))$. Then $p\circ F = \lim_{t\rightarrow 0} p_t \circ \lim_{t\rightarrow 0} H(t) = \lim_{t\rightarrow 0} p_t\circ H(t) = 0$. Moreover, $J$ is a limit of a radical ideal of points with Hilbert function $h_{r,X}$ and therefore $J\in \mathrm{Slip}_{r,X}$. Hence $J\in \VSPb(F,r)$. 
\end{proof}
\end{theorem}

 \section{$\VSPb$ versus $\mathrm{VSP}$ versus $\VPS$}\label{sec:VSPb versus VSP vers VPS}

As in \S \ref{sec:prelim}, let $X$ be a smooth projective toric variety over $\CC$. Let $S=S[X]$ be its
Cox ring and let $T$ be its graded dual ring. Both these rings are graded by $\mathrm{Pic}(X)\cong \ZZ^s$. In this section, we shall be concerned with some more notions of varieties of sums of powers
that will be contrasted to the Buczy\'nska-Buczy\'nski border varieties of sums of powers $\VSPb(F,r)$.

\subsection{Basic definitions and first properties}

\begin{definition}\label{vspclassic}
Let $F\in T_\bb{v}$. The {\it (open) variety of sums of $r$ powers} ($\mathrm{VSP}^{0}$) of $F$ in the Hilbert scheme $\mathrm{Hilb}^r(X)$ is defined set-theoretically as follows: 
\[
\mathrm{VSP}^{0}(F,r) = \left\lbrace Z \ | \  \mathrm{length}(Z)=r \mbox{ and } I_Z \subset \Ann(F) \mbox{ is radical}\right\rbrace. 
\]
The {\it (closed) variety of sums of $r$ powers} ($\mathrm{VSP}$) of $F$ is its Zariski closure in  $\mathrm{Hilb}^r(X)$. Note that $\mathrm{VSP}(F,r)$ is then a reduced closed subscheme of the irreducible smoothable component $\mathrm{Hilb}^r_{sm}(X)$. 
\end{definition}

The name is inherited from the classical case: $X = \PP^n$ and $F \in T_d$, a degree $d$ form. Then 
the smooth zero-dimensional scheme $Z$ spanning $F$ may be identified with a presentation of $F$ as a sum of powers of linear forms. Although for another toric variety this might be meaningless, we shall keep this suggestive terminology. 

\begin{definition}\label{vps}
Let $F\in T_\bb{v}$. The {\it variety of apolar schemes of length $r$ to $F$} is 
\[
\VPS(F,r) = \left\lbrace Z\in \mathrm{Hilb}^r(X) \ | \ I_Z\subset \mathrm{Ann}(F) \right\rbrace.
\]
\end{definition}
So $\VSP^0(F,r)\subset \VPS(F,r)\cap \mathrm{Hilb}^{r}_{sm}(X)\subset \VPS(F,r)$. It is important to stress that $\VPS(F,r)$
is {\it not} a closed subset of the Hilbert scheme in general; see Remark \ref{nonclosed VPS} and \cite[Example~1.1]{JRS}. 

\begin{remark}
Let $\mathrm{Gor}^r(X)\subset \mathrm{Hilb}^r(X)$ be the Gorenstein locus of the Hilbert scheme, whose closed points are the zero-dimensional Gorenstein schemes embedded in $X$. Note that, when $r=\mathrm{crk}(F)$ is the cactus rank of $F$,  \cite[Lemma 2.3]{bb13} implies that $\VPS(F,r)\subset \mathrm{Gor}^r(X)$. 
\end{remark}

\begin{lemma}\label{lem:trivial_containment}
We have $\phi_{r,X}^{-1}(\VPS(F,r)\cap\mathrm{Hilb}^r_{sm}(X)) \subseteq \VSPb(F,r)$.
\end{lemma}
\begin{proof}
Let $I\in \mathrm{Slip}_{r,X}$ be such that 
$\phi_{r,X}(I)\in \VPS(F,r)$. This means that $\overline{I} \subseteq \Ann(F)\subset S$. So $I\subseteq \Ann(F)$ as well and hence $I\in \VSPb(F, r)$.
\end{proof}

The lemma motivates the following definition. 

\begin{definition}[{\bf $\VSPb$ of fiber type}]
A border variety of sums of powers $\VSPb(F,r)$ is said to be {\it of fiber type} if equality
in Lemma \ref{lem:trivial_containment} holds true. 
\end{definition}

\begin{remark}\label{nonclosed VPS}
A $\VSPb(F,r)$ need not be of fiber type. If the equality in Lemma~\ref{lem:trivial_containment} holds we get $\VPS(F,r)\cap\mathrm{Hilb}^r_{sm}(X) = \phi_{r,X}(\VSPb(F,r))$. Since the right-hand side  is closed, in the circumstance that $\VPS(F,r) \cap\mathrm{Hilb}^r_{sm}(X)$ is not closed, equality cannot hold.
\cite[Example~1.1]{JRS} shows this failure for a quadric in four variables. 
\end{remark}

\begin{remark}\label{remark on JRS}
The loci studied extensively in \cite{JRS} for full-rank quadrics are related to ours, although they are  different. Let 
$X = \PP^n$, $T=\CC[x_0,\ldots,x_n]$ and $Q\in T_2$ be a full-rank quadric. Our 
$\VSPb(Q,n+1)$ embeds in the {\it Jelisiejew-Ranestad-Schreyer variety} $\VPS(Q,H)$ (using the notation of {\it loc. cit.}), which in turn sits inside the multigraded Hilbert scheme  $\mathrm{Hilb}_S^{h_{n+1, \PP^n}}$. Moreover, the image of their $\VPS^{good}(Q,H)$ under $\phi_{n+1,\PP^n}$ is our $\VPS(Q,n+1)$, sitting inside $\mathrm{Hilb}^{n+1}(\PP^n)$. When $n+1\leq 13$, their variety $\VSP^{sbl}(Q,H)$ coincides with $\VSPb(Q,n+1)$ \cite[Corollary 3.9]{JRS}. 
\end{remark}

Remark~\ref{nonclosed VPS} shows that generally there is no {\it natural} map between $\VSPb(F,r)$ and $\VPS(F,r)$, which is perhaps 
one of the sources of complication and interest of the theory. \\
A most favourable situation is when $F\in T_\bb{v}$ has the property that {\it any} ideal $I_Z\subset \Ann(F)$ of an apolar scheme $Z$ of $F$ has the generic Hilbert function. Then one has the following. 

\begin{lemma}\label{isobetweenVSP-VPSbar}
Suppose that every ideal $I_Z\subset \Ann(F)$ of a length $r$ subcheme $Z$ has generic Hilbert function $h_{r,X}$  and the corresponding subscheme is smoothable and that any ideal $J\subset \Ann(F)$ with generic Hilbert function is $B$-saturated  and belongs to $\mathrm{Slip}_{r,X}$. Then $\phi_{r,X}$ induces an isomorphism $\VSPb(F, r)\cong \VPS(F,r)$.
\end{lemma}
\begin{proof}
    By Theorem~\ref{thm:restriction on saturable locus is closed immersion} the morphism $\phi_{r,X}$ restricts to a (bijective) locally closed immersion $\phi^{-1}_{r,X}(\VPS(F,r)) \to \VPS(F,r)$. Therefore, we have a surjective closed immersion $\VSPb(F,r)\to \VSP(F,r)$. Since the latter scheme is reduced, we claim that this is an isomorphism.
    Indeed, this can be checked locally on the target, so we need to prove that the only ideal $I\subset A$ such that the map of schemes $\mathrm{Spec} (A/I) \to \mathrm{Spec} (A)$ is bijective is $I=0$. Any such $I$ is contained in the intersection of all the prime ideals of $A$, hence it is the zero ideal as $A$ is reduced.
\end{proof}

This situation in Lemma \ref{isobetweenVSP-VPSbar} is what happens for $X=\PP^2$ and general forms $F$ of low even degrees; see Theorem \ref{plane low even degree}. However, this is too much to expect in general, but we shall prove some birationality results, see \S \ref{ssec:bir results}. 

It is now time for a simple example, which leads to some generalizations; see Theorem \ref{thm:idenfiability via toric regularity} and Corollary \ref{vspwithonlysat}. 

\begin{example}\label{minbr}
Let $X = \PP^n$ and let $d=3$. Let $F\in T_d$ be a concise nonwild form with $\mathrm{srk}(F) = \brk(F)=\dim_{\CC} T_1 = n+1$. Then $\phi^{-1}_{n+1,\PP^n}(\VPS(F,n+1)\cap\mathrm{Hilb}^r_{sm}(\PP^n)) = \VSPb(F,n+1)$ is a single point. Moreover, if $\mathrm{rk}(F)=\brk(F)=n+1$ then $\VSP(F,n+1)=\VPS(F,n+1)\cap\mathrm{Hilb}^r_{sm}(\PP^n)$. 
\begin{proof}
There exists a saturated ideal $I=I_Z\subset \Ann(F)$ of a smoothable scheme $Z$ of length $n+1$. This ideal has necessarily the generic Hilbert function $h_{n+1,\PP^n}$. Since $I$ is the saturated ideal of a smoothable scheme and has the generic Hilbert function, $I\in \VSPb(F,n+1)$. In fact, 
such an $I$ is unique and $I = (\Ann(F)_2)$. For the last sentence, 
the equality follows from noting that the scheme $Z$ above must be smooth. 
\end{proof}
\end{example}

Let $\mathbb{G}(\PP^{h(\bb{v})-1},  \PP(T_\bb{v}))$ be the Grassmannian of $(h(\bb{v})-1)$-dimensional linear spaces in $\PP(T_\bb{v})$. There exists a morphism $\rho\colon \mathrm{Slip}_{r,X} \rightarrow \mathbb{G}(\PP^{h(\bb{v})-1},  \PP(T_\bb{v}))$ \cite[Lemma~3.16]{bb19} defined by $\rho(I) = I^{\perp}_\bb{v}$, the degree $\bb{v}$ summand of the annihilator of $I$. 
Given an embedding of our toric variety $X\subset \PP(T_\bb{v})$, let $\sigma_r(X)$ be the $r$-th secant variety of $X$ inside $\PP(T_\bb{v})$. 

\begin{proposition}\label{prop:vspbar_of_general}
There exists a Zariski open dense subset $W_{sat}\subseteq \sigma_r(X)$ such that, if $[F]\in W_{sat}$, then $\VSPb(F,r)$ contains a $B$-saturated ideal.  There exists a Zariski open dense subset $W_{rad}$ such that if $[F]\in W_{rad}$, 
then $\VSPb(F,r)$ contains a radical ideal. 
\end{proposition}
\begin{proof}
Let 
\[
\mathcal{V}\subseteq \mathbb{G}(\PP^{h(\bb{v})-1}, \PP(T_\bb{v}))\times \PP(T_\bb{v}) 
\]
be the universal subbundle on the Grassmannian. Let $\mathcal{U} \subseteq \mathrm{Slip}_{r,X}\times \PP(T_\bb{v})$ be the pullback of $\mathcal{V}$ along the morphism $\rho$. More explicitly, 
\[
\mathcal{U} = \lbrace  (I, F) \ | \   I\in \mathrm{Slip}_{r,X}, [F]\in \PP(T_\bb{v}),  F\in \rho(I) \rbrace.
\]
Then $\sigma_r(X)$ is the image of the projection $\pi_2$ onto the second factor  of $\mathcal{U}$ \cite[Lemma~3.16]{bb19}.

Let $U$ be the locus in $\mathrm{Slip}_{r,X}$ consisting of $B$-saturated ideals.
This is the intersection of $\mathrm{Slip}_{r,X}$ with the open locus $\mathrm{Hilb}_S^{h_{r,X},sat}$ and 
therefore $U$ is open in $\mathrm{Slip}_{r,X}$. Consider the restriction $\mathcal{U}_{|U}$ and its image $\pi_2(\mathcal{U}_{|U}) \subseteq \sigma_r(X)$.
If $(I,F)\in \mathcal{U}_{|U}$, then $F\in (I^\perp)_\bb{v}$ and so $I_\bb{v}\subset \Ann(F)_\bb{v}$. From Proposition~\ref{inclusion in onedeg} we obtain $I\subset \Ann(F)$. Therefore $\VSPb(F, r)$ contains a $B$-saturated ideal. By Chevalley's theorem, $\pi_2(\mathcal{U}_{|U})$ is a constructible subset of $\sigma_r(X)$.
Since $\mathcal{U}_{|U}$ is dense in $\mathcal{U}$, $\pi_2\colon \mathcal{U}\to \sigma_r(X)$
is surjective and $\overline{\pi_2\left(\overline{\mathcal{U}_{|U}}\right)} = \overline{\pi_2(\mathcal{U}_{|U})}$, it follows that $\pi_2(\mathcal{U}_{|U})$ is a dense subset.  Since the latter set is constructible, it contains an open dense subset $W_{sat}$ with the claimed property. 

Let $U'$ be the preimage $\phi_{r,X}^{-1}(V)$, where $V$ is the open set that appeared in the proof of Proposition~\ref{surjective map from Slip to Hilb^r}. The set $U'$ is open and nonempty. The rest of the argument exhibiting the existence of such $W_{rad}$ is completely analogous as the previous paragraph and we omit it. 
 \end{proof}

\begin{remark}\label{rmk: defective case}
Let $X = \PP^n$ be equipped with the Veronese embedding $\nu_d(\PP^n) \subset \PP(T_d)$. If $n,d\geq 1$ and $r\leq 2$, then the claimed open subset $W_{sat}$ in Proposition \ref{prop:vspbar_of_general} coincides with $\sigma_r(\nu_d(\PP^n))$ and if $r=2$ and $d\geq 3$ it strictly contains $W_{rad}$. 
\begin{proof}
 The case $r=1$ or $d=1$ is clear, so let $r=2$ and $d\geq 2$. The locus of points of border Waring rank equal to $2$ is set-theoretically the union $\sigma_2(\nu_d(\PP^n))\setminus \nu_d(\PP^n) = \sigma^{0}_2(\nu_d(\PP^n))\cup \tau(\nu_d(\PP^n))$, where $\sigma^{0}_2(\nu_d(\PP^n))$ is the locus of points of Waring rank equal to $2$ and $\tau(\nu_d(\PP^n))$ is the tangential variety. Up to the action of $\mathrm{PGL}(n+1,\CC)$ on $\PP^n$, we may assume that: if $[F]\in \sigma^{0}_2(\nu_d(\PP^n))$, then $F=x_0^d + x_1^d$; if $[F]\in \tau(\nu_d(\PP^n))$, then $F=x_0x_1^{d-1}$. In the first case the ideal $(y_0y_1, y_2, \ldots,y_n)\subset \mathrm{Ann}(F)$ and in the second the ideal $(y_0^2, y_2, y_3, \ldots, y_n)\subset \mathrm{Ann}(F)$ are saturated ideals with the generic Hilbert function of two points in $\PP^n$. Since $\mathrm{Hilb}^2(\PP^n)$ is irreducible, by Lemma~\ref{irreducibility of saturablelocus}  both ideals are in $\mathrm{Slip}_{2,\PP^n}$ which shows that $W_{sat} = \sigma_2(\nu_d(\PP^n))$. If $d\geq 3$ these are the only ideals  in $\VSPb(F,2)$. The latter ideal is not radical, so $W_{rad} \neq \sigma_2(\nu_d(\PP^n))$.
\end{proof}
\end{remark}

\begin{corollary}\label{closednessofnonsat}
The set of all  $[F]\in\sigma_r(X)$ for which there exists a nonsaturated (respectively nonradical) ideal in $\VSPb(F,r)$ is closed. 
\begin{proof}
In the notation of Proposition~\ref{prop:vspbar_of_general}, let $\mathcal{Z}$ (respectively $\mathcal{Z}'$) be the complement of $\mathcal{U}_{|U}$ (respectively $\mathcal{U}_{|U'}$) in $\mathcal{U}$. It is closed and so is the locus of $[F]\in \sigma_r(X)$ in the image of $\mathcal Z$ (respectively $\mathcal{Z}'$) under $\pi_2$. 
\end{proof}
\end{corollary}

\begin{corollary}\label{twocaseswherenonsatareproper}
The following statements hold true: 
\begin{enumerate}
\item[(i)] If there exists an $[F]\in \sigma_r(X)$ such that $\VSPb(F,r)$ consists only of $B$-saturated (respectively radical) ideals, then the same is true for a general element of this secant variety.

\item[(ii)] If $\sigma_r(X)\subsetneq \PP(T_\bb{v})$ is nondefective or $\sigma_r(X) = \PP(T_\bb{v})$ and $r(\dim X+1) = \dim_\CC T_{\bb{v}}$, then for a general element $[F]\in \sigma_r(X)$, one has that $\VSPb(F,r)$ consists only of $B$-saturated and radical ideals.

\end{enumerate}

\begin{proof}
(i) Let $\mathcal{Y} = \pi_2(\mathcal Z)$ and $\mathcal{Y}' = \pi_2(\mathcal{Z}')$ be the closed subsets introduced in Corollary \ref{closednessofnonsat}. If $\sigma_r(X)\setminus \mathcal{Y}\neq \emptyset$ (respectively $\sigma_r(X)\setminus \mathcal{Y}'\neq \emptyset)$, then $\mathcal{Y}$ (respectively $\mathcal{Y}'$)  is a proper subset, so its complement is open and dense. 

\noindent (ii) Let $d=\dim X$. The dimension of $\mathcal{Z}$ and $\mathcal{Z}'$ is at most $(rd-1)+(r-1) = r(d+1)-2$. Therefore, the dimension of $\mathcal{Y}$ and $\mathcal{Y}'$ is at most $r(d+1)-2$. If $\sigma_r(X) \subsetneq \PP(T_\bb{v})$ and is nondefective then its dimension is the expected count of parameters $r(d+1)-1$. Hence, in both cases described in the statement, the closed subsets $\mathcal{Y}$ and $\mathcal{Y}'$ are proper subsets of the secant variety.
\end{proof}
\end{corollary}

\begin{remark}
Let $X=\PP^n$. Then a radical ideal $I\neq B=(y_0,\ldots,y_n)\subset S$ is $B$-saturated. 
However, this generally fails for other toric varieties. For instance, let $X=\PP^1\times \PP^1$
and $B=(\alpha_0\beta_0,\alpha_0\beta_1,\alpha_1\beta_0,\alpha_1\beta_1)\subset S=\CC[\alpha_0,\alpha_1,\beta_0,\beta_1]$. 
The ideal $I = (\alpha_0\beta_0,\alpha_0\beta_1,\beta_0-\beta_1)\neq B$
is radical but not $B$-saturated. Its saturation $\overline{I} = (\alpha_0,\beta_0-\beta_1)$ is the radical $B$-saturated ideal corresponding to the point $([0:1],[1:1])\in X$. 
\end{remark}

\begin{remark}\label{rmk:mayfailfordefective}
Consider again the case $X=\PP^n$ equipped with the Veronese embedding $\nu_d$. If $\sigma_r(\nu_d(\PP^n))$ fills the ambient space and $r(n+1) > \binom{n+d}{d}$ or $\sigma_r(\nu_d(\PP^n))$ is defective, then it is possible that the $\VSPb$ of a general element of it consists only of saturated ideals as in Corollary~\ref{twocaseswherenonsatareproper}. However, it is not always the case. 

Let $n=2, d=2$ and $r=3$ and take a full-rank (equivalently, concise) quadric $Q$. We show that $\VSPb(Q, 3)$ does not contain a nonsaturated ideal. Assume that $I\subset \Ann(Q)\subset S$ is a nonsaturated ideal. Then there exists $\ell\in \overline{I}_1$ and hence $\ell\cdot S_1 \subset I \subset \Ann(Q)$. It follows that $\ell\in \Ann(Q)$ which contradicts the assumption that $Q$ is full-rank.

On the other hand, \cite[Corollary~3.10]{JRS} shows that for a general (and hence, {\it for every}, by Corollary \ref{closednessofnonsat}) quadric $Q$ in four variables there exists a nonsaturated ideal in $\VSPb(Q,4)$. 

Now we proceed to the defective cases of quadrics of rank $n$ on $\PP^n$. We may change coordinates so that $Q=x_0^2+\cdots + x_{n-1}^2$.
If $n=3$ and $I\in \VSPb(Q, 3)$, then  $I = (y_3)+K^e$ where the second summand is the extension of $K\subseteq \CC[y_0,y_1,y_2]$, which is an apolar ideal of $Q \in \CC[x_0,x_1,x_2]$ and it has Hilbert function $h_{3, \PP^2}$. By the second paragraph, $K$ is saturated and thus so is $I$.  

Let $n=4$ and let $K\in \VSPb(Q, 4) \subseteq \mathrm{Hilb}_S^{h_{4,\PP^3}}$ be a nonsaturated ideal. Then $I=(y_4)+K^e$ is a nonsaturated apolar ideal to $Q$, having the generic Hilbert function $h_{4,\PP^4}$. 
Since $K\in \mathrm{Slip}_{4,\PP^3}$, it follows from \cite[Proposition~3.1]{CEVV09} that $I\in \mathrm{Slip}_{4,\PP^4}$. We showed that for a general (and hence, for every) quadric $Q$ in $5$ variables with $\mathrm{rk}(Q)=4$ there exists a nonsaturated ideal in $\VSPb(Q,4)$. 
\end{remark}

For further questions in the direction of establishing existence or non-existence of nonsaturated ideals and irreducible components of such in a $\VSPb$, see \S \ref{sec:Outlook}. 

Using a similar construction as in \cite[Lemma~3.16]{bb19} and Proposition \ref{prop:vspbar_of_general}, we give a sufficient condition for the closedness 
of the loci $\VPS(F,r)$. Recall that the closedness of $\VPS(F,r)\cap\mathrm{Hilb}^r_{sm}(X)$ is a necessary condition in order to have a $\VSPb(F,r)$ of fiber type. 

\begin{proposition}\label{prop:VPS_closed}
Let $r\in \NN$. Then there exists $\bb{k}\in \mathrm{Pic}(X)=\ZZ^s$ such that for every $\bb{v}\in \bb{k}+\mathcal K$ and every $F\in T_{\bb{v}}$, 
one has that $\VPS(F,r)$ is closed (possibly empty). 
\begin{proof}
By Lemma \ref{lem:Hilbert_polynomials_of_the_family} applied to the constant Hilbert polynomial $P=r$ and $R=\CC$, there exists $\bb{k}\in \ZZ^s$
such that for any $\bb{v}\in \bb{k}+\mathcal K$ and any scheme $Z\in \mathrm{Hilb}^r(X)$, one has 
$\mathrm{HF}(S/I_Z, \bb{v}) = r$. There exists a natural morphism from $\mathrm{Hilb}^r_{S_{\bb{k}+\mathcal K}}(X)$ to the Grassmannian
$\mathbb G(\PP^{r-1}, \PP(T_{\bb{v}}))$. Composing this map with the isomorphism established in Theorem \ref{sat vs. truncation} gives the morphism $\delta: \mathrm{Hilb}^r(X)\rightarrow \mathbb G(\PP^{r-1}, \PP(T_{\bb{v}}))$ defined by $Z\mapsto (I_Z)^{\perp}_{\bb{v}}$. As in Proposition \ref{prop:vspbar_of_general}, let $\mathcal V$ be the tautological subbundle 
on the Grassmannian, $\mathcal V\subseteq \mathbb G(\PP^{r-1}, \PP(T_{\bb{v}}))\times \PP(T_{\bb{v}})$. Let $\mathcal W$ be the pullback of $\mathcal V$ along the morphism $\delta$. 
Then $\mathcal W =  \lbrace  (Z, F) \ | \   Z\in \mathrm{Hilb}^r(X), [F]\in \PP(T_\bb{v}),  (I_Z)_{\bb{v}}\subset \Ann(F)_{\bb{v}}\rbrace \subseteq \mathrm{Hilb}^r(X)\times \PP(T_{\bb{v}})$. 
Let $\pi_1$ and $\pi_2$ be the first and second projections of the latter product. Let $[F]\in \PP(T_{\bb{v}})$. Then $\VPS(F,r) = \pi_1(\pi_2^{-1}(F)\cap \mathcal W)$. Note that $\mathcal W$ is closed, because $\mathcal V$ is. The latter statement can be checked locally on the Grassmannian. Hence $\pi_2^{-1}(F)\cap \mathcal W$ is closed. Since $\pi_1$ is a proper map, $\VPS(F,r)$ is closed. 
\end{proof}
\end{proposition}

The closedness of $\VPS(F,r)$ prevents the existence of {\it bad limits}, in the terminology of \cite{Ranestad22}.

\begin{corollary}
Let $X = \PP^n$ and let $F\in T_d$. Then, for $d\geq r-1$, $\VPS(F,r)$ is closed. 
\end{corollary}

\begin{proposition}\label{lem:fiber_type_is_non_wild}
Assume that $F\in T_\bb{v}$ has border rank $r$. If $\VSPb(F,r)$ is of fiber type, then $F$ is not wild.
\end{proposition}
\begin{proof}
    If $\VSPb(F,r)$ is of fiber type, then $\phi_{r,X}(\VSPb(F,r)) = \VPS(F,r)\cap \mathrm{Hilb}^{r}_{sm}(X)$ since $\phi_{r,X}$ is surjective. By the border apolarity Theorem \ref{mainbb}, $\VSPb(F,r)$ is nonempty. Hence $\VPS(F,r)\cap \mathrm{Hilb}^r_{sm}(X)$ is nonempty and thus $\mathrm{srk}(F) \leq r$. On the other hand, we have  $\mathrm{srk}(F) \geq \brk(F) = r$. Hence $\brk(F) = \mathrm{srk}(F)$.
\end{proof}

\begin{corollary}
    Let $\bb{k}$ be as in Lemma~\ref{lem:Hilbert_polynomials_of_the_family} applied to the constant Hilbert polynomial $P=r$ and $R=\CC$. If $F\in T_\bb{v}$ for some $\bb{v} \in \bb{k}+\mathcal{K}$, then $\VSPb(F,r)$ is of fiber type. In particular, if $\brk(F) = r$, then $F$ is not wild. 
\end{corollary}
\begin{proof}
    By Proposition~\ref{lem:trivial_containment}, in order to prove that $\VSPb(F,r)$ is of fiber type, it is sufficient to show that 
    $\phi_{r,X}(\VSPb(F,r)) \subseteq \VPS(F,r)\cap \mathrm{Hilb}^r_{sm}(X)$. Let $I\in \VSPb(F,r)$. We need to show that $\overline{I}\subset \Ann(F)$. By Proposition~\ref{inclusion in onedeg} it is enough to establish that $\overline{I}_{\bb{v}} \subseteq \Ann(F)$. By the definition of $\bb{k}$ we have $\overline{I}_{\bb{v}} = I_\bb{v}$ and hence $\overline{I}_{\bb{v}} \subseteq \Ann(F)$ which shows that $\VSPb(F,r)$ is of fiber type. If $\brk(F) = r$, then $F$ is not wild by Proposition~\ref{lem:fiber_type_is_non_wild}.
\end{proof}

\subsection{Birational results}\label{ssec:bir results}

\begin{theorem}\label{birmap1}
Let $F\in T_\bb{v}$ and $r$ be a positive integer.
Then $\phi_{r,X}$ induces a bijection between the set of those irreducible components of the closure of
$\VPS(F,r)\cap \mathrm{Hilb}_{sm}^r(X)$ that contain a scheme with the generic Hilbert function $h_{r,X}$ and the set of those irreducible components of $\VSPb(F,r)$ that contain a $B$-saturated ideal. Under this bijection, the irreducible components in correspondence are birational. 

\begin{proof}
Let $\mathcal{X}_1, \ldots, \mathcal{X}_s$ be the irreducible components of $\VSPb(F,r)$ that contain a $B$-saturated ideal and $\mathcal{X} = \bigcup_{i=1}^s \mathcal{X}_i$. 
Since the set of $B$-saturated ideals is open in $\mathrm{Slip}_{r,X}$ there are ideals $J_1, \ldots, J_s$ such that $J_i \in \mathcal{X}_j$ if and only if $i=j$.
Let $\mathcal{Y}_1, \ldots, \mathcal{Y}_t$ be the irreducible components of the closure of $\VPS(F,r)\cap \mathrm{Hilb}_{sm}^r(X)$ that contain a scheme with the generic Hilbert function $h_{r,X}$ and $\mathcal{Y} = \bigcup_{i=1}^t \mathcal{Y}_i$.
Let  $U \subset \mathrm{Hilb}_{sm}^r(X)$ be the locus of schemes with the generic Hilbert function. It is the complement of the image of the closed subset $\mathrm{Slip}_{r,X}\setminus \mathrm{Hilb}^{h_{r,X},sat}_S$ under the proper map $\phi_{r,X}$. Hence $U$ is open.
Therefore, there are schemes $Z_1, \ldots, Z_t$ with generic Hilbert functions such that $Z_i \in \mathcal{Y}_j$ if and only if $i=j$. Let $i\in \{1,\ldots, s\}$. The image $\phi_{r,X}(\mathcal{X}_i\cap \mathrm{Hilb}_S^{h_{r,X},sat})$ is an irreducible subset of $\VPS(F,r)\cap \mathrm{Hilb}_{sm}^r(X)$, hence $\phi_{r,X} (\mathcal{X}_i) \subseteq \mathcal{Y}_j$ for some $j$. 

On the other hand, for each $j$, the dense open subset $\mathcal{Y}_j\cap U$ of $\mathcal{Y}_j$ is contained in $\phi_{r,X}(\mathcal{X})$, therefore, each $\mathcal{Y}_j$ is the union $\bigcup_{i=1}^s \phi_{r,X}(\mathcal{X}_i)\cap \mathcal{Y}_j$. In particular, there exists $i$ such that $\phi_{r,X}(\mathcal{X}_i)$ contains the generic point of $\mathcal{Y}_j$. Since $\phi_{r,X}(\mathcal{X}_i)$ is closed we get that $\mathcal{Y}_j \subset \phi_{r,X}(\mathcal{X}_i)$. By the above there exists $k$ such that $\phi_{r,X}(\mathcal{X}_i)\subseteq \mathcal{Y}_k$ and we obtain $Z_j \in \mathcal{Y}_k$. Hence $j=k$ and we conclude that $\phi_{r,X}(\mathcal{X}_i) = \mathcal{Y}_j$.

Next we claim that for every $i$ there exists $j$ such that $\phi_{r,X}(\mathcal{X}_i) = \mathcal{Y}_j$.
Let $\mathcal{Y}_j$ be such that $\phi_{r,X}(\mathcal{X}_i) \subset \mathcal{Y}_j$. By the above, there exists $k$ such that $\phi_{r,X}(\mathcal{X}_k)= \mathcal{Y}_j$. In particular, $\phi_{r,X}(J_i) = \phi_{r,X}(K)$ for some $B$-saturated ideal $K\in \mathcal{X}_k$. Since $\phi_{r,X}$ is injective when restricted to $B$-saturated ideals we obtain $J_i=K \in \mathcal{X}_k$ and hence $i=k$.

We showed that $\phi_{r,X}$ induces a bijection between irreducible components of $\mathcal{X}$ and $\mathcal{Y}$. Furthermore, if
$\phi_{r,X}(\mathcal{X}_i) = \mathcal{Y}_j$ then the induced bijective map $\phi_{r,X} \colon \mathcal{X}_i\cap \mathrm{Hilb}_S^{h_{r,X}, sat} \to \mathcal{Y}_j\cap U$ is a locally closed immersion by Theorem~\ref{thm:restriction on saturable locus is closed immersion}. Hence, it is a surjective closed immersion with reduced target and thus an isomorphism. It follows that the morphism $\mathcal{X}_{i} \to \mathcal{Y}_{j}$ is birational.
\end{proof}
\end{theorem}

Note that the locus of elements $[F]\in \sigma_r(X)\subset \PP(T_\bb{v})$ such that there exists a smoothable scheme $Z$ with $I_Z\subseteq \Ann(F)$ and $\HF(S/I_Z) = h_{r,X}$ is dense by Proposition \ref{prop:vspbar_of_general}. Therefore, the sets of the irreducible components under bijection explained in Theorem~\ref{birmap1} are nonempty for a general $[F]\in \sigma_r(X)$.

\begin{proposition}\label{birfromVSPtoVPS} 
Let $F\in T_\bb{v}$ and $r$ be a positive integer. Then $\VSP(F,r)$ is the union of those irreducible components of the closure of $\VPS(F,r)$ that contain a smooth scheme.
\begin{proof}
We claim that the locus $W\subset \mathrm{Hilb}_{sm}^r(X)$ consisting of smooth schemes is a Zariski dense open subset.
It is enough to show that the locus of smooth schemes is open in $\mathrm{Hilb}^r(X)$. Let $\mathcal{U}\subset \mathrm{Hilb}^r(X)\times X$
be the universal family and $\pi\colon \mathcal{U}\to X$ be the projection. The locus $U \subset \mathcal{U}$ of those $x$ such that the fiber of $\pi$ over $\pi(x)$ is smooth is open by \cite[Theorem~12.1.6]{Gro}.
Therefore, its image $W$ under $\pi$ is open since $\pi$ is flat and locally of finite presentation and thus, open by \cite[Theorem~14.35]{GW}. One has $\VSP^0(F,r) = \VPS(F,r)\cap W$. Therefore, the irreducible components of $\VSP(F,r)$ are exactly those irreducible components of the closure of $\VPS(F,r)$ that have nonempty intersection with $W$.
\end{proof}
\end{proposition}

Using similar arguments as the ones in Theorem \ref{birmap1} yields the following. 

\begin{theorem}\label{birmap2}
Let $F\in T_\bb{v}$ and $r$ be a positive integer. 
 Then $\phi_{r,X}$ induces a bijection between the set of irreducible components of 
$\VSP(F,r)\subset\mathrm{Hilb}_{sm}^r(X) $ that contain a scheme with the generic Hilbert function $h_{r,X}$ and the set of irreducible components of $\VSPb(F,r)$ containing a $B$-saturated radical ideal. Under this bijection,
the irreducible components in correspondence are birational. 
\end{theorem}
By Proposition \ref{prop:vspbar_of_general} the bijection from Theorem~\ref{birmap2} is nontrivial for a general $[F]\in \sigma_r(X)$.

A birational result of Massarenti-Mella \cite[Theorem~1]{mm} and 
Ranestad-Schreyer \cite[Theorem~1.1]{rs13} on $\VSP$'s of quadrics, in conjunction with our Theorem \ref{birmap2}, gives an application to $\VSPb$'s of quadrics. 

\begin{corollary}\label{cor: bir quadrics}
Let $X = \PP^n$ and let $F\in T_2$ be a full-rank quadric. Then the unique irreducible component of $\VSPb(F,n+1)$ containing a radical ideal is rational. 
\begin{proof}
Note that every smooth scheme $Z\in \VSP^0(F,n+1)$ has the generic Hilbert function $h_{n+1,\PP^n}$. 
Hence smooth schemes with Hilbert function $h_{n+1,\PP^n}$ are dense in $\VSP(F,n+1)$. In particular, 
any irreducible component of that contains such a $Z$. Under the birational correspondence established by Theorem \ref{birmap2}, every irreducible component $\mathcal X$ of $\VSPb(F,n+1)$ containing a radical ideal (and so saturated, as we work on $X=\PP^n$) is birational to a corresponding irreducible component of $\VSP(F,n+1)$. The latter is irreducible and rational by \cite[Theorem 1.1]{rs13} or by \cite[Theorem 1]{mm}, hence so is $\mathcal X$. 
\end{proof}
\end{corollary}

We observe that another very remarkable birational result of Massarenti and Mella applies to 
$\VSPb$'s, but in a higher range than the generic border rank.  

\begin{corollary}\label{cor: bir mm}
Let $X = \PP^n$ and $d\geq 3$. Suppose that for some positive integer $k< n$, the rational number $c_0=\left(\binom{d+n}{d}-1\right)/(k+1)$ is an integer. Then there exists a dense subset of $[F]\in \PP(T_d)$ such that for any $c\geq c_0$ every irreducible component of $\VSPb(F, c)$ containing a radical ideal is rationally connected.
\begin{proof}
By Proposition \ref{prop:vspbar_of_general} there is a dense subset of $[F]\in \PP(T_d)$ such that there exists a radical $I_Z\subset \Ann(F)$ with generic Hilbert function $h_{c,\PP^n}$.
Under the birational correspondence established by Theorem \ref{birmap2}, every irreducible component of $\VSPb(F,c)$ containing a radical ideal is birational to a corresponding irreducible component of $\VSP(F,c)$. By \cite[Theorem 2]{mm}, every such a component is rationally connected, a property that is stable under birational
maps. 
\end{proof}
\end{corollary}

\section{Membership in $\mathrm{Slip}_{r,\PP^n}$, complete intersections and monomials}\label{sec:Membership in Slip}

\subsection{Ext criterion for complete intersections}
Let $S=\CC[y_0,\ldots, y_n]$ be the Cox ring of $\PP^n$. We consider a homogeneous complete intersection ideal $J = (f_1, \ldots, f_n)$ of $S$ where $\deg(f_i) = a_i$. Let $d=\sum_{i=1}^n a_i - (n+1)$ and $r=\prod_{i=1}^n a_i$. In this section, for the ease of notation, $h$ denotes the Hilbert function $h_{r, \PP^n}$. We require that $\HF(S/J) \neq h$.

We have $\HF(S/J, s) = r$ for every $s\geq d+1$ and $\HF(S/J, d) = r-1$. Therefore, if $I\in \mathrm{Hilb}_S^{h}$ satisfies $\overline{I} = J$ and $I\in \mathrm{Slip}_{r,\PP^n}$, then by \cite[Theorem~3.4~and~Corollary~3.3]{JM} we have $(J^2)_d \subseteq I_d$. We show a partial converse--Theorem~\ref{prop:complete_intersection}. We start with a numerical result about the Hilbert function of $S/J$.

\begin{lemma}\label{lem:ci_symmetry_of_HF}
If $0\leq s\leq d$, then $\HF(S/J, s) + \HF(S/J, d-s) = r$.
\end{lemma}
\begin{proof}
 We may assume that $J=(y_1^{a_1}, \ldots, y_n^{a_n})$. Let $R=\CC[y_1,\ldots, y_n]$ and $I = J\cap R$. We have $\HF(S/J,s) = \sum_{i=0}^s \HF(R/I, i)$. Furthermore, $\HF(R/I, i) = \HF(R/I, d+1-1)$ since $I$ is the annihilator of the degree $d+1$ monomial
 $x_1^{a_1-1}\cdots x_n^{a_n-1}$.
Therefore,
\begin{align*}
\HF(S/J,s) + \HF(S/J, d-s) &= \sum_{i=0}^s \HF(R/I, i) + \sum_{i=0}^{d-s} \HF(R/I, i) \\
&= \sum_{i=0}^s \HF(R/I, i) + \sum_{i=s+1}^{d+1} \HF(R/I, i) = r.
\end{align*}
The last equality follows from writing a monomial basis for the quotient. 
\end{proof}

\begin{proposition}\label{lem:exts}
If $I\in \mathrm{Slip}_{r,\PP^n}$ is such that $\overline{I} = J$, then $\dim_\CC \mathrm{Ext}^1_S(J/I_{\geq d}, S/J)_0 = 1$. Equivalently, if $I\in \mathrm{Slip}_{r,\PP^n}$ is such that $\overline{I} = J$, then $(J^2)_d\subseteq I_d$. 
\end{proposition}
\begin{proof}
By \cite[Corollary~3.3]{JM} we get $\dim_\CC \mathrm{Ext}^1_S(J/I_{\geq d}, S/J)_0 =  \dim_\CC \frac{J_d}{I_d + J^2_d} =  \dim_\CC\mathrm{Ext}^1_S(J/I, S/J)_0.$
Since $I_d$ is of codimension $1$ in $J_d$ we get $\dim_\CC \frac{J_d}{I_d + J^2_d} \leq 1$. Since $I\in \mathrm{Slip}_{r,\PP^n}$ and $I\neq \overline{I} = J$, it follows from \cite[Theorem~3.4]{JM} that $\dim_\CC\mathrm{Ext}^1_S(J/I, S/J)_0 > 0$.
\end{proof}

\begin{remark}\label{rmk:com_int}
Let $\delta_1 < \delta_2 < \cdots <\delta_t$ be the distinct integers appearing among $a_j$'s and let $m_i$ be the number of $a_j$ equal to $\delta_i$. In \cite[\S 2.2.2]{Ben12} there is an inductive construction of an integral smooth scheme $H_t$ and a family $\mathcal{X}_t \subseteq H_t\times \PP^n$ over $H_t$ all of whose fibers are complete intersections generated by $m_i$ minimal generators of degree $\delta_i$ with $i\in \{1,2,\ldots, t\}$. Using the inductive definition of $H_t$ one verifies that
$\dim H_s = \sum_{i=1}^s m_i \cdot \HF(S/J, \delta_i)$ for all $1\leq s\leq t$. In particular, $\dim H_t = \sum_{i=1}^n \HF(S/J, a_i)$.
By the universal property of the Hilbert scheme there is a morphism $\psi \colon H_t \to \mathrm{Hilb}^r(\PP^n)$ corresponding to $\mathcal{X}_t$. Let $\mathcal{CI}$ be the (set-theoretic) image. It is irreducible and it is constructible by Chevalley's theorem. Furthermore, if $\Gamma\in \mathcal{CI}$, then $\Gamma$ is a complete intersection of degrees $(a_1,a_2,\ldots, a_n)$. 
Conversely, if $\Gamma$ is such a complete intersection, then $\Gamma\in \mathcal{CI}$ by \cite[Lemma~2.2.2]{Ben12}. The morphism $\psi$ is injective on closed points so $\dim \mathcal{CI} = \dim H_t$.

The Analogously defined constructible subset in $\mathrm{Hilb}^P(\PP^n)$, with $\deg(P)>0$ (i.e. for positive-dimensional subschemes), is open \cite[Lemme 2.2.3]{Ben12}. However this is false for zero-dimensional complete intersections: the source of failure is that the map on global sections $H^0(\mathcal O_{\PP^n}(k)) \rightarrow H^0(\mathcal O_Z(k))$ is not always surjective for $k> 0$ in this case. 

\end{remark}

We use the following observation.
\begin{lemma}\label{lem:ci_is_open}
    Let $g$ be the Hilbert function of $S/J$. There is a smooth irreducible open subset of $\mathrm{Hilb}_{S}^{g}$ of dimension 
    $\sum_{i=1}^n \HF(S/J, a_i)$ containing $J$ and consisting of ideals of complete intersections.
\end{lemma}
\begin{proof}
Pick the minimal graded free resolution of $J$, $\cdots \to F_1 \to F_0 \to J\to 0$. We have 
\[
\dim_\CC \mathrm{Hom}_S(J,S/J)_0 \leq \dim_\CC \mathrm{Hom}_S(F_0, S/J)_0 = \sum_{i=1}^n \dim_\CC \mathrm{Hom}_S(S(-a_i), S/J))_0 = \sum_{i=1}^n \HF(S/J, a_i).
\]
Let $\mathcal{Z}$ be an irreducible component of the saturable locus in $\mathrm{Hilb}_S^g$ such that $J\in \mathcal{Z}$. By \cite[Corollary~3.16]{JM}, the point $J$ is a smooth point of $\mathrm{Hilb}_S^g$ so there is only one such $\mathcal{Z}$. By the above calculation $\dim \mathcal{Z} \leq \sum_{i=1}^n \HF(S/J, a_i)$.  

Let $\mathcal{CI}$ be as in Remark~\ref{rmk:com_int} and let $\mathcal{Z}_0$ be its inverse image under the natural map $\alpha\colon \mathrm{Hilb}_S^g \to \mathrm{Hilb}^r(\PP^n)$. It is an irreducible subset of the saturable locus of dimension $\sum_{i=1}^n \HF(S/J, a_i)$ that contains $J$. Hence its closure is  $\mathcal{Z}$. The set $\mathcal{Z}_0$ is a dense constructible subset of $\mathcal{Z}$ so it contains a nonempty open subset $U$ of $\mathcal{Z}$. Every point of this subset is a smooth point of $\mathrm{Hilb}_S^g$. Hence $U$ is disjoint from other irreducible components. Thus, it is open in $\mathrm{Hilb}_S^g$.

\end{proof}

\begin{theorem}\label{prop:complete_intersection}
Let $W$ be a codimension one subspace of $J_d$. There exists $I\in \mathrm{Slip}_{r, \PP^n}$ with $I_{\geq d} = (W) + J_{\geq d+1}$ if and only if $(J^2)_d\subseteq W$.
\end{theorem}
\begin{proof}
The only if part follows from Proposition \ref{lem:exts}. Let $g$ be the Hilbert function of $S/J$. Let $U$ be the irreducible set from Lemma~\ref{lem:ci_is_open}.
Let $\alpha\colon \mathrm{Hilb}_S^g \to \mathrm{Hilb}^r(\PP^n)$ be the natural map and let $V$ be the set-theoretic image of $U$ under $\alpha$. By Theorem~\ref{thm:restriction on saturable locus is closed immersion} it is a locally closed subset and $\dim V = \dim U$.
Let $K = (W) + J_{\geq d+1}$ and $h_{\geq d}$ be the Hilbert function of $S/K$. Let $\beta\colon \mathrm{Hilb}_S^{h_{\geq d}} \to \mathrm{Hilb}^r(\PP^n)$ be the natural map and $Z$ be the preimage of $V$. By construction $Z$ contains $K$.

We start with computing the dimension of $Z$. The fiber of the proper morphism $\beta_{|Z}\colon Z\to V$ over each point of $V$ is a projective space of dimension $\dim_\CC S_d - r$. Indeed, we need to choose a codimension one subspace of the degree $d$ part of the ideal. Hence a codimension one subspace of a linear space of dimension $\dim_\CC S_d - (r-1)$. 
We conclude that $Z$ is irreducible and
\begin{equation}\label{eq:1}
\dim Z = \dim_\CC S_d - r + \dim V = \dim_\CC S_d - r + \dim U.
\end{equation}

Let $Z_0\subset Z$ denote the closed locus of those $K'$ for which $(\overline{K'})^2_d \subseteq K'_d$.  Each fibre of the proper morphism $\beta_{|Z_0}\colon Z_0 \to V$ is the projective space of codimension one subspaces of $J_d/(J^2)_d$. Thus $Z_0$ is irreducible and has dimension $\dim U + (\dim_\CC S_d - \HF(S/J, d)) - (\dim_\CC S_d - \HF(S/J^2, d)) -1$, i.e.
\begin{equation}\label{eq:2}
\dim Z_0 = \dim U + \HF(S/J^2, d) - r = \sum_{i=1}^n \HF(S/J, a_i) + \HF(S/J^2, d) - r.
\end{equation}

Let $\gamma\colon \mathrm{Hilb}_S^{h}\to \mathrm{Hilb}_S^{h_{\geq d}}$ be the natural map and let $Y = \gamma(\mathrm{Slip}_{r,\PP^n})$. It is an irreducible closed subset of dimension $nr$. 
 We claim that $\overline{Z}\cap Y = \overline{Z_0}$ (set-theoretically). By Proposition \ref{lem:exts}, we have $Z\cap Y \subseteq Z_0$ and thus $Z\cap Y = Z_0\cap Y$. Consequently, $\overline{Z}\cap Y = \overline{Z_0}\cap Y \subseteq \overline{Z_0}$. We show that the opposite inclusion also holds. Pick any $K'\in Z_0\cap Y$ and let $J'=\overline{K'}$. By \cite[Lemma~2.6]{Man}, in order to show that $\overline{Z_0} = \overline{Z}\cap Y$, it is sufficient to show that 
\begin{equation}\label{eq:3}
    \dim_\CC T_{K'}\mathrm{Hilb}_S^{h_{\geq d}} = \dim \overline{Z}+1 
\end{equation}
and 
\begin{equation}\label{eq:4}
    \dim \overline{Z_0} = nr-1.
\end{equation}

\noindent By applying the left-exact covariant functor $\mathrm{Hom}_S(K', \cdot)_0$ to the exact sequence $0\rightarrow J'/K' \rightarrow S/K'\rightarrow S/J'\rightarrow 0$, we find
\begin{align*}
&\dim_\CC \mathrm{Hom}_S(K', S/K')_0 \leq \dim_\CC\mathrm{Hom}_S(K', J'/K')_0 + \dim_\CC \mathrm{Hom}_S(K', S/J')_0 \\
&\leq \dim_\CC \mathrm{Hom}_S(K', J'/K')_0 + \dim_\CC \mathrm{Hom}_S(J', S/J')_0 + \dim_\CC \mathrm{Ext}^1_S(J'/K', S/J')_0
\end{align*}
where the second inequality follows from applying the left-exact contravariant functor $\mathrm{Hom}_S(\cdot, S/J')$ to the short exact sequence $0\to K'\to J' \to J'/K'\to 0$.
We have $\dim_\CC \mathrm{Hom}_S(J', S/J')_0 = \dim U$ by Lemma~\ref{lem:ci_is_open}. Therefore, 
\[
\dim_\CC \mathrm{Hom}_S(K', S/K')_0 \leq (\dim_\CC S_d - r)(r-\HF(S/J, d)) + \dim U + \dim_\CC \mathrm{Ext}^1_S(J'/K', S/J')_0
\]
which is equal to $\dim \overline{Z} + 1$ by \eqref{eq:1} and Lemma~\ref{lem:exts}. 

Finally, we prove \eqref{eq:4}. By \eqref{eq:2} it is equivalent to 
$
\sum_{i=1}^n \HF(S/J, a_i) + \HF(S/J^2, d) - r = nr-1.
$
By \cite[Corollary~2.3]{gvt} we have 
\begin{align*}
\sum_{i=1}^n \HF(S/J, a_i) + \HF(S/J^2, d) -r &= \sum_{i=1}^n \HF(S/J, a_i) + \HF(S/J, d) + \sum_{i=1}^n \HF(S/J, d-a_i) -r\\
&= \sum_{i=1}^n \big(\HF(S/J, a_i) + \HF(S/J, d-a_i)\big) -1 = nr-1
\end{align*}
where the last equality follows from Lemma~\ref{lem:ci_symmetry_of_HF}.
\end{proof}

\subsection{Applications to $\VSPb$'s of monomials}
Let $n\geq 2$ and $T=\CC[x_0, \ldots, x_n]$ be the graded dual ring of $S$.
Let $1\leq a_0\leq a_1\leq\cdots \leq a_n$ and consider the monomial $F=x_0^{a_0}\cdots x_n^{a_n}\in T$.
Let $J = (y_0^{a_0+1}, \ldots, y_{n-1}^{a_{n-1}+1})$, $d=\sum_{i=0}^{n-1} a_i - 1$ and $r=\prod_{i=0}^{n-1} (a_i+1)$.

\begin{proposition}\label{prop:vsp_ci_easy}
    If $\binom{a_0+\cdots +a_{n-1}+n-2}{n} \leq \prod_{i=0}^{n-1} (a_i+1)$ and $a_n + 1 > \sum_{i=0}^{n-1} a_i$, then $\brk(F) = r$ and 
    $\VSPb(F, r)$ is
    \begin{enumerate}
        \item[(i)] a point if $\binom{a_0+\cdots +a_{n-1}+n-1}{n} \leq r$;
        \item[(ii)] isomorphic to $\PP^N$, where $N=\HF(S/J^2, d) - r$, otherwise.
    \end{enumerate}
\end{proposition}
\begin{proof}
We have $\brk(F) = r$ by \cite[Theorem 11.3]{LT10}. Furthermore, if $I\in \VSPb(F,r)$, then $I_{d+1}=\Ann(F)_{d+1} = J_{d+1}$. Since $J$ is generated in degrees at most $a_{n-1}+1 \leq d+1$ we conclude that $\overline{I} = J$. 
If $\binom{d+n}{n}=\binom{a_0+\cdots +a_{n-1}+n-1}{n} \leq r$ we obtain $I_d=0$ and hence $I = J_{\geq d+1}$ is the unique point of $\VSPb(F, r)$. 

Assume that $\binom{a_0+\cdots +a_{n-1}+n-1}{n} > r$. It follows from Theorem~\ref{prop:complete_intersection} that if $I\in \VSPb(F,r)$, then $(J^2)_d \subseteq I_d$. Conversely, we claim that for every codimension one subspace of $J_d$ containing $(J^2)_{d}$ there is a unique $I\in \VSPb(F,r)$ such that $I_d = W$. Let $W$ be such a subspace and $I = W+J_{\geq d+1}$. It follows from $\binom{(d-1)+n}{n}=\binom{a_0+\cdots +a_{n-1}+n-2}{n} \leq r$ that $I\in \mathrm{Hilb}_S^{h_{r,\PP^n}}$. From Theorem~\ref{prop:complete_intersection} we conclude that there exists $I'\in \mathrm{Slip}_{r,\PP^n}$ with $I'_{\geq d} = I_{\geq d}$. Hence
\[
I = I_{\geq d} = I'_{\geq d} = I'
\]
is in $\mathrm{Slip}_{r, \PP^n}$ and therefore, it is a point of $\VSPb(F, r)$. By construction, it is the unique point of $\VSPb(F,r)$ whose degree $d$ part is equal to $W$.
\end{proof}

\begin{remark}\label{rmk: complete intersections inside ann}
Assume that $F$ is a form in $T$ of border rank $r= \prod_{i=0}^{n-1}(a_i+1)$ and there exists a complete intersection ideal $J \subseteq \Ann(F)$ minimally generated by forms of degrees $a_0+1,\ldots, a_{n-1}+1$. 
Then the conclusions of Proposition~\ref{prop:vsp_ci_easy} hold for $\VSPb(F,r)$ with exactly the same proof.
\end{remark}

Assumptions of Proposition~\ref{prop:vsp_ci_easy} are very restrictive. Sometimes it is possible to relax the assumptions a little and still be able to conclude that the border variety of sums of powers has a rational irreducible component of dimension $\HF(S/J^2, d)-r$. We illustrate this in a special case when $n=2$. Consider the monomials of the form $x_0x_1^ax_2^b$ with $b\geq a+1$. Note that in this case $d$ as defined above is equal to $a$. The first inequality from the statement of Proposition~\ref{prop:vsp_ci_easy} holds if and only if $a\in \{1,2,3,4\}$. Furthermore, if $a\in \{1,2\}$, then $\VSPb(F, 2(a+1)) = \lbrace \mathrm{pt} \rbrace$ and if $a\in \{3,4\}$, then $\VSPb(F, 2(a+1))\cong \PP^{2(a-2)}$. 

\begin{proposition}\label{prop:application to vspbar of mons}
    Let $F=x_0x_1^ax_2^b$ with $5\leq a \leq 8$ and $b\geq a+1$. There is an irreducible component of $\VSPb(F,2(a+1))$ that is birational to $\mathbb{P}^{2(a-2)}$.
\end{proposition}
\begin{proof}
As before, let $J=(y_0^2, y_1^{a+1})$. We show that for $5\leq a \leq 8$ there is a nonempty open subset of the space of those codimension one subspaces $W$ of $J_a$ containing $J^2_a$ for which there exists a unique ideal $I\in \mathrm{Hilb}_S^{h_{r,\PP^2}}$ with $I_{\geq a} = (W) + J_{\geq a+1}$. Then, the closure of the locus of those $I$ in $\mathrm{Hilb}_S^{h_{r, \PP^2}}$ is an irreducible closed subset birational to $\mathbb{P}^{2(a-2)}$. Furthermore, by Theorem~\ref{prop:complete_intersection} it is contained in $\VSPb(F, r)$ and hence is one of its irreducible components by construction.

We are left with verifying our claim. Using the pairing between $S_a$ and $T_a$, a codimension $1$ subspace $W$ of $J_a$ containing $J^2_a$ can be equivalently described by $W^\perp_a$ such that $ J^\perp _a \subset W^\perp_a \subset (J^2)^\perp_a$ with $\mathrm{codim}_{W^\perp_a}(J^\perp _a) = 1$. We have $J^\perp _a = \langle x_0^ix_1^jx_2^k \mid i\leq 1 \text{ and } i+j+k=a\rangle$ and $(J^2)^\perp_a = \langle x_0^ix_1^jx_2^k \mid i\leq 3 \text{ and } i+j+k = a \rangle$. Therefore, the choices of a codimension one subspace of $J_a$ containing $J^2_a$ are in natural bijection with the choices of  $\omega \in T_a$  all of whose monomials are divisible by $x_0^2$ but none of them is divisible by $x_0^4$. Given $\omega \in T_a$ we denote by $\partial_i \omega$ the subspace of $T_{a-i}$ spanned by all the derivatives of $\omega$ of order $i$. 
We have:
\begin{align*}
\mathrm{HF}(S/(y_0^2, y_1^6)) &= 1 \ \ 3 \ \ 5 \ \ 7 \ \ 9 \ \ 11 \ \ 12 \ \ 12 \ \ \cdots\\
\mathrm{HF}(S/(y_0^2, y_1^7)) &= 1 \ \ 3 \ \ 5 \ \ 7 \ \ 9 \ \ 11 \ \ 13 \ \ 14 \ \ 14 \ \ \cdots\\
\mathrm{HF}(S/(y_0^2, y_1^8)) &= 1 \ \ 3 \ \ 5 \ \ 7 \ \ 9 \ \ 11 \ \ 13 \ \ 15 \ \ 16 \ \ 16 \ \ \cdots\\
\mathrm{HF}(S/(y_0^2, y_1^9)) &= 1 \ \ 3 \ \ 5 \ \ 7 \ \ 9 \ \ 11 \ \ 13 \ \ 15 \ \ 17 \ \ 18 \ \ 18 \ \ \cdots.
\end{align*}
The form $\omega$ uniquely defines an ideal $I$ with generic Hilbert function such that $I_{\geq a} = (W) + J_{\geq a+1}$ if and only if $I_{a-i}:= (\partial_i \omega + J_{a-i}^\perp)^\perp_{a-i}$ has codimension $r = 2(a+1)$ in $S_{a-i}$ for all $i > 0$ such that $h_{r, \PP^2}(a-i) = r > \dim_\CC S_{a-i}$. That is, if and only if $\dim_\CC (\partial_i \omega + J_{a-i}^\perp) = 2(a+1)$ for those $i$.
Therefore, our claim is equivalent to the condition that for a general $\omega$ we have
\[
\dim_\CC (\partial_i \omega + J_{a-i}^\perp) = 2(a+1) \text { for } \begin{cases}
    i=1 \text{ if } a=5\\
    i=1,2 \text{ if } a=6, 7\\
    i=1,2,3 \text{ if } a=8.
\end{cases} 
\]
This is verified by the following \texttt{Macaulay2} script, run with the parameters 
\[
(a,e,r)\in \{(5,1,3), (6,1,3), (6,2,5), (7,1,3), (7,2,5), (8,1,3), (8,2,5), (8,3,7)\}.
\]
\begin{mybox}
{\color{blue}
\begin{verbatim}
(a,e,r) = (5,1,3);
T=QQ[x_0..x_2];
N=flatten entries super basis(a, ideal(x_0^2)/ideal(x_0^4));
n=#N;
A=QQ[b_1..b_n];
T=A[x_0..x_2];
N=apply(N, i-> sub(i, T));
w=0;
for i from 0 to n-1 do w=w+b_(i+1)*N#i;
Mon=flatten entries super basis (a-e, ideal(x_0^2));
Dif = flatten entries super basis(e, T);
for i from 0 to #Dif-1 do (
  C_i=(toList coefficients(diff(Dif#i, w), Monomials=>Mon))#1;
)
M=C_0;
for i from 1 to #Dif-1 do M=M|C_i;
I=sub(minors(r, M), A);
assert (dim I < n)
\end{verbatim}
}
\end{mybox}
\end{proof}

\section{Border identifiability and multigraded regularity}\label{sec:Multigraded reg}

We establish a criterion for border identifiability employing multigraded regularity, defined in \S\ref{sec:Hilbert schemes}. 
With this notion of regularity, Maclagan and Smith proved the following. 

\begin{proposition}[{\cite[Proposition 6.7]{MS}}]\label{prop: MS Prop. 6.7}
Let $Z\subset X$ be a zero-dimensional subscheme of length $r$ and let $I_Z$ be its $B$-saturated ideal in the Cox ring $S$. Then $\bb{m}\in \mathrm{reg}(S/I_Z)$ if and only if the space of forms vanishing on $Z$ has codimension $r$ in the space of forms of multidegree $\bb{m}$ in $S$. 
\end{proposition}

We are now ready to state our criterion to decide border identifiability. 

\begin{theorem}\label{thm:idenfiability via toric regularity} Let $X\subset \PP(T_{\bb v})$ and $r=\brk_X(F)$ be the border rank of $F\in T_{\bb{v}}$. Suppose that there exists $\bb{u} \in \mathbf{\mathcal  K}$ such that
\[
\HF(S/\Ann(F), \bb{u}) = \HF(S/\Ann(F), \bb{u}+\bb{c}_1+\cdots + \bb{c}_l) = r.
\]
If there exists a $B$-saturated ideal $I\in \VSPb(F, r)$, then $\VSPb(F, r) = \{I\}$.
\end{theorem}
\begin{proof}
Let $\bb{e} = \bb{c}_1+\cdots +\bb{c}_l$. We claim that for every choice of $\varepsilon_i \in \{0,1\}$ for $i\in \{1,2,\ldots, l\}$ we have $\HF(S/I, \bb{u} + \varepsilon_1\bb{c}_1 + \cdots + \varepsilon_l \bb{c}_l) = r$. Indeed, since $I$ is $B$-saturated, for each $i$ there is a nonzerodivisor on $S/I$ of degree $\bb{c}_i$ (see the proof of Lemma~\ref{lem:equality_of_sat}). Thus $\HF(S/I, \bb{v}+\bb{c}_i) \geq \HF(S/I, \bb{v})$ for every $\bb{v}\in \ZZ^s$. Our claim follows from the assumed equality of Hilbert functions. By Proposition \ref{prop: MS Prop. 6.7}, we have $\bb{u}+\varepsilon_1\bb{c}_1+\cdots + \varepsilon_l\bb{c}_l \in \mathrm{reg}(S/I).$
We claim that $\bb{u}+\bb{e} \in \mathrm{reg}(I)$. Since $I \subseteq S$ and $S$ is an integral domain, we have $H^0_B(I)_{\bb{p}} = 0$ for every $\bb{p}\in \mathbb{Z}^s$. According to Definition \ref{def:CM-regularity of MS}, we need to show that $H^i_B(I)_{\bb{p}} = 0$ for every $i \geq 1$ and all $\bb{p}\in \bigcup (\bb{u}+\bb{e}-\lambda_1\bb{c}_1 - \cdots - \lambda_l\bb{c}_l + \mathbf{\mathcal  K})$ where the union is over all $\lambda_1, \ldots, \lambda_l\in \NN$ such that $\lambda_1+\cdots +\lambda_l = i-1$. \cite[Corollary~3.6]{MS} states that if $\bb{u}\in \mathbf{\mathcal  K}$, then the local cohomology group $H^i_B(S)_{\bb{u}}$ vanishes. Now, using the long exact sequence of local cohomology groups associated to the short exact sequence $0\to I \to S \to S/I \to 0$ we conclude that for every $\bb{u}\in \mathbf{\mathcal  K}$ and every $i\in \NN$ we have
\begin{equation}\label{eq:reg2}
    H_B^i(S/I)_{\bb{u}} \cong H^{i+1}_B(I)_{\bb{u}}.
\end{equation}

Assume first that $i=1$ and $\bb{w}\in \mathbf{\mathcal  K}$. Since $\lambda_j\in \NN$, the only zero linear combination is when they are all zero; so we need to show that $H_B^1(I)_{\bb{u}+\bb{e}+\bb{w}} = 0$. By \eqref{eq:reg2} we have
$H_B^1(I)_{\bb{u}+\bb{e}+\bb{w}} \cong H^0_B(S/I)_{\bb{u}+\bb{e}+\bb{w}}$ which is zero since $\bb{u}\in \mathrm{reg}(S/I)$ and $\bb{e}+\bb{w}-\bb{c}_1 \in \mathbf{\mathcal  K}$. 
Let $\bb{w}\in \mathbf{\mathcal  K}$, $i\geq 2$ and $\lambda_1,\ldots, \lambda_l$ be nonnegative integers whose sum is $i-1$. We have to prove that $H_B^{i}(I)_{\bb{u}+\bb{e}-\lambda_1\bb{c}_1 - \cdots -\lambda_l \bb{c}_l +\bb{w}} = 0$. Equivalently, by \eqref{eq:reg2}, this amounts to check $H_B^{i-1}(S/I)_{\bb{u}+\bb{e} -\lambda_1\bb{c}_1 - \cdots -\lambda_l \bb{c}_l + \bb{w}} = 0$. Up to permuting the vectors $\bb{c}_j$, we may assume that $\lambda_1 \geq 1$. We have 
\[
\bb{u} + \bb{e} - \lambda_1\bb{c}_1 - \cdots - \lambda_l\bb{c}_l + \bb{w}= (\bb{u}+\bb{e}-\bb{c}_1) - (\lambda_1-1)\bb{c}_1-\cdots - \lambda_l\bb{c}_l + \bb{w}
\]
which shows that $H_B^{i-1}(S/I)_{\bb{u}+\bb{e} -\lambda_1\bb{c}_1 - \cdots -\lambda_l \bb{c}_l + \bb{w}} = 0$ since $\bb{u}+\bb{e}-\bb{c}_1 \in \mathrm{reg}(S/I)$.

Let $J \in \VSPb(F, r)$. Since $\Ann(F)_{\bb{u}+\bb{e}} = I_{\bb{u}+\bb{e}}$, we have $J_{\bb{u}+\bb{e}} = I_{\bb{u}+\bb{e}}$. It follows that $J$ contains $(I_{\bb{u}+\bb{e}})$. The latter ideal is equal to $(I_{|\bb{u}+\bb{e}+ \mathcal  K})$ by \cite[Theorem~5.4]{MS} since $\bb{u}+\bb{e}\in \mathrm{reg}(I)$.
Due to the equality of Hilbert functions we have $J_{|\bb{u}+\bb{e}+\mathbf{\mathcal  K}} = I_{|\bb{u}+\bb{e}+\mathbf{\mathcal  K}}$. Hence, by Lemma~\ref{lem:equality_of_sat} we obtain $\overline{J} = \overline{I} = I$. In particular, $J\subseteq I$ which implies that $I=J$.
\end{proof}

This result immediately yields the following corollary, which we also prove in an alternative
elementary way. 

\begin{corollary}\label{cor:border identifiability for P^n}
Let $X = \PP^n$ and $S = \CC[y_0,\ldots, y_n]$ with $\deg(y_i) = 1$. Let $F\in T_d$ with $\brk(F) = r$ and suppose there exists $a\in \ZZ$ such that $\HF(S/\Ann(F),a) = \HF(S/\Ann(F),a+1) = r$. If $\VSPb(F, r)$ contains a saturated ideal $I$, then $\VSPb(F, r)  = \{I\}$.
\begin{proof}[Alternate proof] 
For every $J\in \VSPb(F, r)$ we have  $I_a = J_a$ and $I_{a+1} = J_{a+1}$. In particular, $(I_{a+1}) \subseteq J$. We claim that $(I_{a+1}) = I_{\geq a+1}$.
Indeed, if $c$ is the smallest integer such that $\HF(S/I,c) = r$, then by \cite[Theorem~1.69]{ik} the ideal $I$ has no minimal generators in degree greater than $c+1$. The claim follows since $c\leq a$.  The ideals $I$ and $J$ have the same Hilbert function so $I=\overline{I} = \overline{J}$. Hence $J=I$ by the equality of their Hilbert functions.\end{proof}
\end{corollary}

\begin{corollary}\label{cor: p^n and deg 2s+1}
Let $d=2s+1$ and $r = \binom{n+s}{s}$. Then  
a general $[F]\in \sigma_{r}(\nu_d(\PP^n))$ is border identifiable. 
\begin{proof}
By \cite[Lemma 1.17]{ik}, a general $[F]\in \sigma_{r}(\nu_d(\PP^n))$ has apolar ideal
$\Ann(F)$ with Hilbert function
\[
\mathrm{HF}(S/\Ann(F), k) = \min\lbrace\dim_{\CC} S_k ,\dim_{\CC} S_{d-k} \rbrace, \mbox{ for } 0\leq k\leq d.
\]
Thus $\mathrm{HF}(S/\Ann(F),s)= \mathrm{HF}(S/\Ann(F), s+1)=r$. Proposition~\ref{prop:vspbar_of_general} and Corollary \ref{cor:border identifiability for P^n} prove the statement. 
\end{proof}

\end{corollary}

\begin{lemma}\label{lem:sufficient condition for all nonsat ideals}
Let $F\in T_{\bb{v}}$ and suppose that $\max \{\HF(S/\Ann(F), \bb{v}) \mid \bb{v}\in \mathbb{Z}\} =r$. Let $\bb{b}_1, \ldots, \bb{b}_k$ be the degrees of the minimal generators of the irrelevant ideal $B$. If there exist $\bb{0}\neq \bb{u}, \bb{u'} \in \ZZ^s$, with $\bb{u'} - \bb{u} -\bb{b}_i \in \mathcal{K}$ for every $i = 1,2,\ldots, k$ such that $\HF(S/\Ann(F),\bb{u}) < \dim_\CC S_{\bb{u}}\leq r$ and $\HF(S/\Ann(F), \bb{u'}) = r$, then $\VSPb(F, r)$ consists only of nonsaturated ideals. 
\begin{proof}
By assumption, there exists $0\neq \omega \in \Ann(F)_{\bb{u}}$. For every $I\in \VSPb(F, r)$ and every $i\in\{1,\ldots, k\}$ we
have $S_{\bb{u'}-\bb{u}-\bb{b}_i}S_{\bb{b}_i}\omega \in \Ann(F)_{\bb{u'}} = I_{\bb{u'}}$. Let $f_i \in S_{\bb{u'}-\bb{u}-\bb{b}_i}$ be a
nonzerodivisor on $S/\overline{I}$ (see the proof of Lemma~\ref{lem:equality_of_sat}). We conclude that $S_{\bb{b}_i}\omega \in \overline{I}$ for every
$i$ and therefore $\omega \in \overline{I}$. In particular $I\neq \overline{I}$. 
\end{proof}
\end{lemma}

\begin{example}
Let $X=\PP^1\times \PP^1$ with Cox ring $S = \CC[\alpha_0, \alpha_1, \beta_0, \beta_1]$, where $\deg(\alpha_0) = \deg(\alpha_1) = (1,0)$ and $\deg(\beta_0) = \deg(\beta_1) = (0,1)$. 
Let $T=\CC[a_0, a_1, b_0, b_1]$ be the graded dual ring. Consider $F = a_0^4 b_1^4 + a_1^4b_0^4 + a_1^4b_1^4\in T_{(4,4)}$ and let $r=3$. One 
verifies that $\HF(S/\Ann(F), (2,0)) = 2 < 3 = \dim_{\CC} S_{(2,0)}$, but $\HF(S/\Ann(F), (3,1)) = 3$.
      Finally observe that $(\alpha_0\alpha_1, \alpha_0\beta_0, \beta_0\beta_1)$ is a radical and $B$-saturated ideal of $S$ contained in $\Ann(F)$ with Hilbert polynomial $3$. Hence 
      $\VSPb(F,3)$ is nonempty and by Lemma~\ref{lem:sufficient condition for all nonsat ideals} it consists only of nonsaturated ideals.
\end{example}

\section{Minimal border rank and wildness}\label{sec:minimal border rank}

\subsection{Concise minimal border rank tensors}
We start  with recalling a useful result about multigraded Hilbert functions of zero-dimensional schemes inside products of projective spaces. 
\begin{proposition}[{\cite[Proposition~1.9]{SVT06}, \cite[Lemma~4.24]{Man22}}]\label{prop:properties_of_MHF}
Let $X = \PP^{n_1}\times \cdots \times\PP^{n_s}$ and let $Z\subset X$ be a zero-dimensional scheme 
with $B$-saturated ideal $I_Z$. Let $\bb{e}_j\in \NN^s\subset \mathrm{Pic}(X)$ be the $j$-th standard basis vector. Then:
\begin{enumerate}
\item[(i)] for all $\bb{v}\in \NN^s$ and 
all $1\leq j\leq s$, one has 
$\mathrm{HF}(S/I_Z,\bb{v})\leq \mathrm{HF}(S/I_Z, \bb{v}+\bb{e}_j)$;

\item[(ii)] if $\mathrm{HF}(S/I_Z,\bb{v})=\mathrm{HF}(S/I_Z,\bb{v}+\bb{e}_j)$, 
then $\mathrm{HF}(S/I_Z, \bb{v}+\bb{e}_j)=\mathrm{HF}(S/I_Z, \bb{v}+2\bb{e}_j)$; 

\item[(iii)] $\mathrm{HF}(S/I_Z,\bb{v})\leq \mathrm{length}(Z)$ for all $\bb{v}\in \NN^s$. 
\end{enumerate}
\end{proposition}

Let $X = \PP^{m-1}\times \PP^{m-1}\times \PP^{m-1}$, $S$ be its Cox ring and $T$ the graded dual of $S$. Then $T_{\bb{1}}\cong \CC^m\otimes \CC^m\otimes \CC^m$,
where $\bb{1}=(1,1,1)$.

\begin{lemma}\label{lem:generic_function}
Let $K$ be a $B$-saturated ideal of $S$ such that $\HF(S/K, (a,b,c)) = m$ for every $(a,b,c) \in \{0,1\}^3$ with $a+b+c\in \{1,2\}$. If $\HF(S/K, \bb{1}) = m$, then $S/K$ has generic Hilbert function $h_{m, X}$.
\end{lemma}
\begin{proof}
Let $(a,b,c) \in \NN^3\setminus \{(0,0,0)\}$. We need to show that $\HF(S/K, (a,b,c)) = m$. Assume first that only one of $\{a,b,c\}$ is positive. We may and do assume that it is $a$. From $\HF(S/K, (0,1,0)) = \HF(S/K, (1,1,0)) =m$ we conclude using Proposition~\ref{prop:properties_of_MHF}(ii) that $\HF(S/K,(a,1,0))=m$. From  Proposition~\ref{prop:properties_of_MHF}(i), we derive
\[
m= \HF(S/K, (1,0,0)) \leq \HF(S/K, (a,0,0)) \leq \HF(S/K, (a,1,0))=m,
\]
which implies that $\HF(S/K, (a,0,0))=m$.

In what follows, we repeatedly use Proposition~\ref{prop:properties_of_MHF}(i) and (ii).
Assume that $a,b>0$ and $c=0$. We have  already established that $\HF(S/K, (a,0,0))= \HF(S/K, (a, 1, 0)) =m$. Thus, we obtain $\HF(S/K, (a,b,0))=m$. 

Finally, assume that $a,b,c >0$. It is enough to show that $\HF(S/K, (a,b,1)) = m$ to conclude that $\HF(S/K, (a,b,c)) = m$.
We have $\HF(S/K, (0,b,1))=m$, so it is sufficient to show that $\HF(S/K, (1,b,1))=m$. This follows since $\HF(S/K, (1,0,1))=\HF(S/K, \bb{1})=m$.
\end{proof}

\begin{theorem}\label{thm:minimalBRTensors}
Let $F\in T_{\bb{1}}$ be concise and of minimal border rank, i.e. $\brk(F)=m$. Let $I = (\Ann(F)_{(1,1,0)})+(\Ann(F)_{(1,0,1)}) + (\Ann(F)_{(0,1,1)})\subset S$ and $K=\overline{I}$. Then the following statements hold:
\begin{enumerate}
\item[(i)] If $\HF(S/I, \bb{1}) \neq m$, then $F$ is wild. 

\item[(ii)] If $\HF(S/I, \bb{1}) =m$, then $F$ is not wild if and only if $I_{(a,b,c)} = K_{(a,b,c)}$ for every $(a,b,c) \in \mathcal{S}$, where $\mathcal{S} = \{(1,0,0), (0,1,0), (0,0,1), (1,1,0), (1,0,1), (0,1,1),(1,1,1)\}$.
\end{enumerate}
\end{theorem}
\begin{proof}
(i) Since $\brk(F)=m$, it follows from the border apolarity Theorem~\ref{mainbb} (see also \cite[Theorem 5.5]{bb19}) that $\HF(S/I, \bb{1}) \geq m$. If $\HF(S/I, \bb{1}) > m$, then, in the terminology of \cite{JLP}, the tensor $F$ is not $111$-sharp. Therefore, it is wild by \cite[Theorem~9.2]{JLP}. 

(ii) We then assume that $F$ is $111$-sharp, i.e.  $\HF(S/I, \bb{1}) =m$. 
    Suppose first that $F$ is not wild. According to Definition \ref{def:wild}, this implies that its cactus rank satisfies $\crk(F)\leq \mathrm{srk}(F) = \brk(F)=m$. By the cactus apolarity lemma (see Definition \ref{def:cactus}), there is a $B$-saturated homogeneous ideal $J\subseteq \Ann(F)\subset S$ having the multigraded Hilbert polynomial equal to $\crk(F)$. By Proposition~\ref{prop:properties_of_MHF}(iii), for every $(a,b,c)\in \NN^3$ we have $\HF(S/J, (a,b,c)) \leq \crk(F)$. If $(a,b,c)\in \{0,1\}^3$ with $a+b+c\in\{1,2\}$, then from
    \[
    m=\HF(S/\Ann(F), (a,b,c)) \leq \HF(S/J, (a,b,c)) \leq \crk(F) \leq m
    \]
    we conclude that $\crk(F) = m$ and $J_{(a,b,c)} = \Ann(F)_{(a,b,c)} = I_{(a,b,c)}$. Hence $I\subseteq J$ and therefore $K\subseteq J$. Using Proposition~\ref{prop:properties_of_MHF}(i)(iii) we get the inequalities 
    \[
     m = \HF(S/J, (1,1,0)) \leq \HF(S/J, \bb{1}) \leq \crk(F) = m,
    \]
    from which we conclude that $\HF(S/J, \bb{1})=m$.  It follows that $I_{(a,b,c)} = J_{(a,b,c)}$ for every $(a,b,c) \in \mathcal{S}$,  which implies that for any such $(a,b,c)$ we have also $I_{(a,b,c)} = K_{(a,b,c)}$. 

For the converse, assume that $I_{(a,b,c)} = K_{(a,b,c)}$ for every $(a,b,c)\in \mathcal{S}$. By Lemma~\ref{lem:generic_function}, the algebra $S/K$ has generic Hilbert function $h_{m,X}$. 

Let $J \in \VSPb(F, m)$. We have $J_{(a,b,c)} = I_{(a,b,c)}$ for every $(a,b,c)\in \mathcal{S}$. Therefore, $I\subseteq J$ and so $K=\overline{I}\subset \overline{J}$. 
However, $K$ and $J$ have the same multigraded Hilbert polynomial, so $\overline{J} = \overline{K} = K$. 
Hence $J\subseteq \overline{J} = K$. Since these two have the same Hilbert function, $J = K$. 
This implies that $\VSPb(F, m)=\{K\}$.  In particular, $\mathrm{Proj} (S/K)$ is a smoothable scheme. It follows that $F$ has minimal smoothable rank, i.e. $\mathrm{srk}(F)=m = \brk(F)$. Thus, $F$ is not wild. 
\end{proof}

\begin{corollary}\label{cor:minimalBRWild}
 If $F$ is a nonwild concise minimal border rank tensor in $T_{\bb{1}}$, then $\VSPb(F, m) = \{K\}$ where $K$ is the saturation of $I = (\Ann(F)_{(1,1,0)}) + (\Ann(F)_{1,0,1}) + (\Ann(F)_{(0,1,1)})$.
\end{corollary}
\begin{proof}
By Theorem~\ref{thm:minimalBRTensors},  $I_{(a,b,c)} = K_{(a,b,c)}$ for every $(a, b, c) \in \mathcal{S}$. Therefore, $\VSPb(F, m) = \{K\}$ by the argument presented in the last paragraph of the proof of Theorem~\ref{thm:minimalBRTensors}.
\end{proof}

\subsection{Tensors of minimal border rank three}

In this section we answer a question posed by Buczy\'nska and Buczy\'nski \cite[\S 5.2]{bb19}
finding the description of $\VSPb(F,3)$ for all minimal border rank three tensors $F$. Let $X = \PP^{n_1}\times \cdots \times \PP^{n_s}$ and let $\bb{1} = (1,\ldots, 1)\in \ZZ^s$, which corresponds to a very ample line bundle $D$. Then the closed embedding induced by $\mathcal L = \mathcal O_X(D)$ is the {\it Segre embedding} $\mathrm{Seg}(X)\subset \PP(T_{\bb{1}})$. Given $[G]\in \mathrm{Seg}(X)$, denote by $\widehat{T}_{[G]}\mathrm{Seg}(X)$ the tangent space at $G$ to the affine cone of $\mathrm{Seg}(X)$.

\begin{theorem}\label{theo:min border rank three}
Let $X = \PP^2\times \PP^2\times \PP^2$ and let $F$ be a border rank three concise tensor in $\CC^3\otimes \CC^3\otimes \CC^3\cong T_{\bb{1}}$.
The variety $\VSPb(F, 3)$ is a single point, unless $F=G'+H'$ where $H'\in \widehat{T}_{[H]}\mathrm{Seg}(\PP^2\times \PP^2\times \PP^2)$, $G'\in \widehat{T}_{[G]}\mathrm{Seg}(\PP^2\times \PP^2\times \PP^2)$ with $[G], [H]$ being two distinct points on a line in $\mathrm{Seg}(\PP^2\times \PP^2\times \PP^2)$. In the latter case, $F$ is wild and $\VSPb(F, 3)$ is isomorphic to $\PP^3$. 
\end{theorem}
\begin{proof}
    By \cite[Theorem~1.2]{BL14}, it is sufficient to consider the case when $F$ is one of the four tensors whose normal forms (i)--(iv) are classified by Buczy\'nski and Landsberg \cite[pages 477--478]{BL14}.
    Let $I$ be the ideal generated by the degree $(1,1,0), (1,0,1)$ and $(0,1,1)$ 
    parts of $\Ann(F)$ and let $K = \overline{I}$.
    If $F$ is one of the tensors described in (i)--(iii) we conclude using Theorem~\ref{thm:minimalBRTensors} that $F$ is not wild. It follows from Corollary~\ref{cor:minimalBRWild} that $\VSPb(F, 3) = \{K\}$.

    Assume that $F$ has normal form (iv), i.e. $F=a_2\otimes b_1\otimes c_2 + a_2\otimes b_2\otimes c_1+ a_1\otimes b_1\otimes c_3 + a_1\otimes b_3 \otimes c_1 + a_3\otimes b_1\otimes c_1$. It follows from Theorem~\ref{thm:minimalBRTensors} that $F$ is wild (this was already known, see \cite[proof~of~Proposition~2.4]{bb15}). We show that $\VSPb(F, 3) \cong \PP^3$. Let $\alpha_i, \beta_j, \gamma_k$ be the generators of $S$ dual to $a_i, b_j, c_k$ respectively.

    First, we identify a subset $\mathcal{X}$ of $\mathrm{Hilb}_S^{h_{3, X}}$ which contains $\VSPb(F,3)$. As before, we consider the ideal  $I=(\Ann(F)_{(1,1,0)}) + (\Ann(F)_{(1,0,1)}) + (\Ann(F)_{(0,1,1)})$. We have $\HF(S/I, (2,1,0)) = \HF(S/I, (0,2,1)) = \HF(S/I, (0,1,2)) = 3$. We compute that $(I\colon (\beta_1, \beta_2, \beta_3))_{(0,0,2)} = \langle \gamma_2^2, \gamma_2\gamma_3, \gamma_3^2\rangle $, $(I\colon (\gamma_1,\gamma_2,\gamma_3))_{(0,2,0)} = \langle \beta_2^2, \beta_2\beta_3, \beta_3^2\rangle$ and $(I\colon (\beta_1,\beta_2,\beta_3))_{(2,0,0)} = \langle \alpha_1\alpha_3, \alpha_2\alpha_3, \alpha_3^2\rangle$. It follows that if $J\in \VSPb(F, 3)$, then 
    $J_{(2,0,0)} \oplus J_{(0,2,0)} \oplus J_{(0,0,2)}\subseteq \langle \alpha_1\alpha_3, \alpha_2\alpha_3, \alpha_3^2\rangle \oplus \langle \beta_2^2, \beta_2\beta_3, \beta_3^2\rangle \oplus \langle \gamma_2^2, \gamma_2\gamma_3, \gamma_3^2\rangle$. Since these three vector subspaces have codimension $3$ in $S_{(2,0,0)}, S_{(0,2,0)}$ and $S_{(0,0,2)}$, respectively, we conclude that the containment is an equality. Therefore, $J \supseteq I'$  where 
    \[
    I'= I + (\alpha_1\alpha_3, \alpha_2\alpha_3, \alpha_3^2, \beta_2^2, \beta_2\beta_3, \beta_3^2, \gamma_2^2, \gamma_2\gamma_3, \gamma_3^2).
    \]
    Furthermore, $\HF(S/I', (3,0,0)) = 4$. Therefore, $J$ contains a unique minimal generator of degree $(3,0,0)$ that is a cubic in the variables $\alpha_1,\alpha_2$. We show that for every nonzero cubic $C$ in $\alpha_1, \alpha_2$ the ideal $I'+(C)$ has the generic Hilbert function. Consider the lexicographic order with $\gamma_3\prec \gamma_2\prec \gamma_1 \prec \beta_3 \prec \beta_2\prec \beta_1 \prec \alpha_3\prec \alpha_2\prec \alpha_1$. By computing the $S$-polynomials we verify that the set
    \[
    \{
    \gamma_3^2, \gamma_2\gamma_3, \gamma_2^2, \beta_3\gamma_3, \beta_3\gamma_2, \beta_2\gamma_3, \beta_2\gamma_2, \beta_1\gamma_3-\beta_3\gamma_1, \beta_1\gamma_2-\beta_2\gamma_1, \beta_3^2, \beta_2\beta_3, \beta_2^2,
    \]
    \[
    \alpha_3\gamma_3, \alpha_3\gamma_2, \alpha_3\beta_3, \alpha_3\beta_2, \alpha_2\gamma_3, \alpha_2\gamma_2-\alpha_3\gamma_1, \alpha_2\beta_3, \alpha_2\beta_2-\alpha_3\beta_1, \alpha_1\gamma_3-\alpha_3\gamma_1,
    \]
    \[
    \alpha_1\gamma_2, \alpha_1\beta_3-\alpha_3\beta_1, \alpha_1\beta_2, \alpha_3^2, \alpha_2\alpha_3, \alpha_1\alpha_3, x\alpha_1^3+y\alpha_1^2\alpha_2+z\alpha_1\alpha_2^2+w\alpha_2^3
    \}
\]
forms a Gr\"obner basis for every choice of $(x,y,z,w) \in \CC^4\setminus \{(0,0,0,0)\}$.
There are four possible monomial ideals as initial ideals depending on which one of scalars $x,y,z$ and $w$ are zero. These ideals are generated by the initial monomials of all but the last generator displayed above and one of the four $\{\alpha_1^3, \alpha_1^2\alpha_2, \alpha_1\alpha_2^2, \alpha_2^3\}$. In all four cases, we verify in \texttt{Macaulay2} that the quotient algebra has the generic Hilbert function $h_{3, X}$. Let $\mathcal{X} = \{I'+(C) \mid C\in \CC[\alpha_1, \alpha_2]_3 \setminus \{0\}\}$. Before checking that $\VSPb(F,3) =\mathcal{X}$ we prove that the latter set is closed and with its reduced scheme structure is isomorphic to $\PP^3$. Let $g$ be the Hilbert function of $S/I'$. Consider the flag multigraded Hilbert scheme $\mathrm{Hilb}_S^{h_{3,X}, g}$ with its two projections $p_{h_{3,X}}\colon \mathrm{Hilb}_S^{h_{3,X}, g}\to \mathrm{Hilb}_S^{h_{3,X}}$ and $p_g\colon \mathrm{Hilb}_S^{h_{3,X}, g} \to \mathrm{Hilb}_S^g$. We have $\mathcal{X} = p_{h_{3,X}} (p_g^{-1}(I'))$ and it is therefore closed.
From now on, $\mathcal{X}$ will be the corresponding reduced closed subscheme.

Consider the flag variety $\mathrm{Flag}(6,7, S_{(3,0,0)})$ and its two projections $\pi_6\colon \mathrm{Flag}(6,7, S_{(3,0,0)}) \to \mathbb{G}(6, S_{(3,0,0)})$ and $\pi_7\colon \mathrm{Flag}(6,7, S_{(3,0,0)}) \to \mathbb{G}(7, S_{(3,0,0)})$. Let $\mathcal{Y}$ be the fiber of $\pi_6$ over $(\alpha_3)_3$. Observe that $\mathcal{Y}\cong \PP(S_{(3,0,0)}/(\alpha_3)_3)\cong \PP^3$. We claim that $\pi_7|_{\mathcal{Y}}\colon \mathcal{Y}\to \pi_7(\mathcal{Y})$ is an isomorphism.
Since the target is reduced and the map is surjective, it is sufficient to show that it is a closed immersion, or equivalently, that $\pi_7|_{\mathcal{Y}}\colon \mathcal{Y}\to \mathbb{G}(7, S_{(3,0,0)})$ is a closed immersion. By~\cite[Corollary~12.94]{GW}  it is enough to prove that the map of $\CC[\varepsilon]/(\varepsilon^2)$-points is injective. Let $R$ be a $\CC$-algebra.
The $R$-points of $\mathcal{Y}$ can be identified with those $R$-submodules $W\subset R\otimes_\CC S_{(3,0,0)}$ for which $(\alpha_3)_3\subset W$ and $(R\otimes_\CC S_{(3,0,0)})/W$ is a locally free $R$-module of rank $\dim_\CC S_{(3,0,0)}-7 = 3$. Under this identification, the map on $R$-points takes $W$ to $W$ and so it is injective. 

Consider the natural map $\pi\colon \mathrm{Hilb}^{h_{3,X}}_S \to \mathbb{G}(7, S_{(3,0,0)})$. The morphism $\pi$ restricted to $\mathcal{X}$ factors through $\pi_7(\mathcal{Y})$. Furthermore, the induced map $\mathcal{X}\to\pi_7(\mathcal{Y})$ is a bijective morphism of $\CC$-varieties whose target is normal. Hence it is an isomorphism by Zariski's Main Theorem.

We have established that $\mathcal{X}$ is an irreducible closed subset of $\mathrm{Hilb}_S^{h_{3,X}}$. The discussion above yields $\VSPb(F, 3) \subseteq \mathcal{X}$. In order to show the opposite inclusion consider four distinguished points of $\mathcal{X}$: $I'+(\alpha_1^3), I'+(\alpha_1^2\alpha_2), I'+(\alpha_1\alpha_2^2), I'+(\alpha_2^3)$. It can be checked in \texttt{Macaulay2} that the tangent space to the multigraded Hilbert scheme $\mathrm{Hilb}_S^{h_{3, X}}$ at all those four points is $18$-dimensional. Since this is also the dimension of $\mathrm{Slip}_{3,X}$ and every ideal in $\mathcal{X}$ is apolar to $F$, in order to show that $\VSPb(F, 3) = \mathcal{X}$, it is sufficient to show that at least one of the four points considered above is in $\mathrm{Slip}_{3,X}$.
By border apolarity Theorem~\ref{mainbb}, there is an ideal $J \in \VSPb(F,3)\subseteq \mathcal{X}$. Consider the weight vector $\mathbf{w}=(1,2,3,1,2,3,1,2,3)$. The initial ideal of $J$ with respect to $\mathbf{w}$ is a point of $\mathrm{Slip}_{3,X}$. Since all the quadratic generators of $J$ displayed above are homogeneous with respect to $\mathbf{w}$, the initial ideal is one of the four $I'+(\alpha_1^3), I'+(\alpha_1^2\alpha_2), I'+(\alpha_1\alpha_2^2), I'+(\alpha_2^3)$. 
\end{proof} 
\begin{proof}[Alternate proof of the last case] 
In this proof we use part of the first argument, highlighting the border rank decomposition perspective. The first proof shows that any ideal $I\in \VSPb(F,3)$ is such that there exists a unique cubic $C_{I}\in \langle \alpha_1^3, \alpha_1^2\alpha_2, \alpha_1\alpha_2^2, \alpha_2^3 \rangle$ up to scaling that is 
the generator of the vector space 
$(I/(\alpha_3))_{(3,0,0)}$. Such a cubic may be regarded as a point  $C_I\in \PP(\langle \alpha_1^3, \alpha_1^2\alpha_2, \alpha_1\alpha_2^2, \alpha_2^3 \rangle)$. This defines a morphism 
\[
\psi_{(3,0,0)}:\VSPb(F,3)\longrightarrow \PP(\langle \alpha_1^3, \alpha_1^2\alpha_2, \alpha_1\alpha_2^2, \alpha_2^3 \rangle)\cong \PP^3,
\]
\[
I\longmapsto [C_I].
\]
This morphism is clearly injective on $\CC$-points. It is also surjective on $\CC$-points. To see this, first notice that $\psi_{(3,0,0)}$ is projective and so closed. Thus it is enough to show its image contains a general point of $\PP(\langle \alpha_1^3, \alpha_1^2\alpha_2, \alpha_1\alpha_2^2, \alpha_2^3 \rangle)$.

Let $\ell_1 = e_1a_1+e_2a_2$
and $\ell_2 = f_1a_1+f_2a_2$ be two linearly independent vectors in $\CC^3$ and let $\ell_3 \in \langle \ell_1,\ell_2\rangle$. Up to scaling $\ell_1$ and $\ell_2$, we may assume $\ell_3 = \ell_1+\ell_2$. We now exhibit a minimal border rank decomposition for $F$ utilising $\ell_1$ and $\ell_2$. 
For any $t\in \mathbb{A}^1_{\CC}\setminus \lbrace 0\rbrace$, consider the expression 
\[
G_t = \frac{1}{t}\left[(ta_3 - \ell_3)\otimes b_1\otimes c_1 + \ell_1\otimes(b_1+t\cdot \delta_2b_2 + t\cdot \delta_3 b_3)\otimes (c_1+t\cdot \rho_2c_2 + t\cdot \rho_3 c_3)  \right]+
\]
\[
+ \frac{1}{t}\left[\ell_2\otimes (b_1+t\cdot \tau_2b_2 + t\cdot \tau_3 b_3)\otimes (c_1+t\cdot \eta_2c_2 + t\cdot \eta_3 c_3)\right]. 
\]
Given $f_1,f_2,e_1,e_2\in \CC$, one has that $\lim_{t\rightarrow 0} G_t = F$ if and only if there exist complex values for the  parameters $\delta_2,\delta_3, \ldots, \eta_2,\eta_3$ such that the following relations are satisfied
\[
\begin{pmatrix}
e_1 & f_1 \\
e_2 & f_2 \\
\end{pmatrix}\cdot 
\begin{pmatrix}
\delta_2 & \delta_3 & \rho_2 & \rho_3 \\
\tau_2 & \tau_3 & \eta_2 & \eta_3 \\
\end{pmatrix}= 
\begin{pmatrix}
0 & 1 & 0 & 1 \\
1 & 0 & 1 & 0 \\
\end{pmatrix}.
\]
Since $\ell_1$ and $\ell_2$ are linearly independent, such parameters' values do exist and are unique. 
By Theorem \ref{borderdecgivespointinVSPb} and using the following \texttt{Macaulay2} script (where, for simplicity of notation, 
we work with $a_i, b_j, c_k$ instead) 
\begin{mybox}
{\color{blue}
\begin{verbatim}
S=QQ[a_1..a_3]**QQ[b_1..b_3]**QQ[c_1..c_3];
S'=S[e_1,e_2,f_1,f_2, delta_2, delta_3, rho_2, rho_3,
tau_2, tau_3, eta_2, eta_3,t];
I=intersect(ideal((f_1+e_1)*a_3+t*a_1, (f_2+e_2)*a_3+t*a_2, b_2,b_3,c_2,c_3),
ideal(a_3, e_2*a_1-e_1*a_2, delta_2*t*b_1-b_2, delta_3*t*b_1-b_3,
rho_2*t*c_1-c_2, rho_3*t*c_1-c_3), ideal(a_3, f_2*a_1-f_1*a_2,tau_2*t*b_1-b_2,
tau_3*t*b_1-b_3, eta_2*t*c_1-c_2, eta_3*t*c_1-c_3 ));
J=trim sub(saturate(I, t), {t => 0});
R=S'/sub(ideal(a_3),S'); 
J'= trim sub(J,R); 
cJ = (select(flatten entries mingens J', (i -> (degree i)#1 == 3)))#0;
assert(cJ ==(e_2*a_1 - e_1*a_2)*(f_2*a_1-f_1*a_2)*((e_1+f_1)*a_2-(e_2+f_2)*a_1))
\end{verbatim}
}
\end{mybox}
\noindent we find that this border rank decomposition corresponds to an ideal $J\in \VSPb(F,3)$ whose unique generator modulo the ideal $(\alpha_3)$ in degree $(3,0,0)$ is the cubic 
\[
C_J=(e_2\alpha_1-e_1\alpha_2)(f_2\alpha_1-f_1\alpha_2)((e_1+f_1)\alpha_2-(e_2+f_2)\alpha_1). 
\]
Up to scaling, this cubic form is a general point in $\PP(\langle \alpha_1^3, \alpha_1^2\alpha_2, \alpha_1\alpha_2^2, \alpha_2^3 \rangle)\cong\PP^3$, because any two of the three linear forms are linearly independent. Hence $\psi_{(3,0,0)}$ is surjective. Since the morphism $\psi_{(3,0,0)}$ is bijective on $\CC$-points and the target is a normal scheme, Zariski's Main Theorem implies that $\psi_{(3,0,0)}$ is an isomorphism of $\CC$-schemes. 
\end{proof}

\subsection{Wildness}

A special and simple feature of the structure of $\VSPb$'s for wild elements (see Definition \ref{def:wild}) is that their closed points are nonsaturated ideals. 

\begin{lemma}\label{vspwild}
Let $F\in T_{\bb{v}}$ be wild with $\brk(F) = r$. Then any ideal $J\in \VSPb(F, r)$ is nonsaturated. 
\begin{proof}
On the contrary, suppose there exists a $B$-saturated ideal $J\in \VSPb(F, r)$. By Proposition~\ref{surjective map from Slip to Hilb^r} the scheme $Z$ defined by $J$ is smoothable. One has $\mathrm{srk}(F)\leq \mathrm{length}(Z) = r$, a contradiction.
\end{proof}
\end{lemma}
\noindent The previous lemma cannot be used as a criterion for wildness, because its converse is false even for $X = \PP^2$. The following example illustrates this.

\begin{example}\label{nonwild_no_saturated_ideals}
Let $F = x_0x_1^2x_2^3\in T_6$. Then $\underline{{\bf rk}}(F) = \mathrm{srk}(F) = \mathrm{crk}(F)=6$. In particular, $F$ is not wild. 
However, $\VSPb(F,6)$ consists of a unique nonsaturated ideal. 
\begin{proof}
Since $\mathrm{Hilb}^r(\PP^2)$ is irreducible, it follows that $\mathrm{srk}(F) = \mathrm{crk}(F)$. We have $\mathrm{Ann}(F) = (y_0^2, y_1^3, y_2^4)$ and its Hilbert function is  $\mathrm{HF}(S/\mathrm{Ann}(F)) = 1 \ \ 3 \ \ 5 \ \ 6 \ \ 5 \ \ 3 \ \ 1.$ As a result, we get $6\leq \brk(F) \leq \mathrm{srk}(F) = \mathrm{crk}(F)$. 
From the containment $(y_0^2,  y_1^3)\subset \mathrm{Ann}(F)\subset S$ we deduce that $\mathrm{crk}(F) = 6$ and therefore $\brk(F) = \mathrm{srk}(F) = \mathrm{crk}(F) = 6$. Any ideal $J\in \VSPb(F,6)$ satisfies $J_{\leq 2} = 0 $ and $J_{\geq 3} = (y_0^2, y_1^3)_{\geq 3}$ and hence $(y_0^2, y_1^3)_{\geq 3}$ is a unique such ideal. Note that it is not saturated.
\end{proof}
\end{example}

In order to find an equivalence between the presence of nonsaturated ideals in $\VSPb$ and wildness of $F$, we then have to impose further assumptions. 

\begin{proposition}\label{iff with a nonsat ideal in vsp for P^n}
Let $X = \PP^n$. Let $F\in T_d$ be such that there exists $a\in \ZZ$ with 
$\mathrm{HF}(S/\Ann(F),a) = \mathrm{HF}(S/\Ann(F),a+1) = r = \brk(F)$, and, for all $k<a$, one has
$\mathrm{HF}(S/\Ann(F),k) = h_{r,\PP^n}(k)$. Then 
$F$ is wild if and only if $\VSPb(F,r)$ contains a nonsaturated ideal. Furthermore, in that case $\VSPb(F,r)$ consists only of nonsaturated ideals.
\begin{proof}
If $F$ is wild, then $\VSPb(F,r)\neq \emptyset$ contains only nonsaturated ideals by Lemma \ref{vspwild}. 
For the converse, since $\VSPb(F,r)$ contains a nonsaturated ideal, by Theorem \ref{thm:idenfiability via toric regularity} it cannot
contain any saturated ideal. Suppose now by contradiction that $\mathrm{srk}(F) = \brk(F) = r$. 
Then there exists a $B$-saturated ideal $J\subset \Ann(F)$ of a smoothable scheme with Hilbert polynomial equal to $r$. By assumption on 
the Hilbert function of $\Ann(F)$, $J$ has the generic Hilbert function $h_{r,\PP^n}$. Hence $J\in \mathrm{Slip}_{r,\PP^n}$. Therefore 
$J\in \VSPb(F,r)$, a contradiction. Thus $F$ is wild. 
\end{proof}
\end{proposition}

This result along with a characterization of wildness for minimal border rank forms  in terms of Hessians yields the following corollary. 

\begin{corollary}\label{vspwithonlysat}
Let $X = \PP^n$. Let $d =  3$ or $d\geq n+2$ and $F\in T_d$ be a minimal border rank concise form, i.e. $\brk(F) = n+1$. Then $\mathrm{Hess}(F)\neq 0$ if and only if $\VSPb(F,n+1)$ consists of a unique saturated ideal. When this holds, the unique 
saturated ideal is $(\Ann(F)_2)$. 
\begin{proof}
By \cite[Theorem 4.9]{HMV} $\mathrm{Hess}(F)\neq 0$ if and only if $F$ is not wild. 
We claim that in both cases considered in the statement $\mathrm{HF}(S/\Ann(F),2)=n+1$. If $d=3$ it is clear. If $d\geq n+2$ we obtain the claim by using Macaulay's bound applied to $\mathrm{HF}(S/\Ann(F),d-2)$ and the symmetry of the Hilbert function of $\Ann(F)$. 
If $F$ is not wild, then $\VSPb(F,n+1)\neq \emptyset$ contains only saturated ideals by Proposition ~\ref{iff with a nonsat ideal in vsp for P^n}. Since, in particular, $\VSPb(F,n+1)$
contains a saturated ideal, then 
by Corollary~\ref{cor:border identifiability for P^n} it contains a unique one. For the converse, suppose $\VSPb(F,n+1)$ contains a unique 
saturated ideal. Again by Proposition \ref{iff with a nonsat ideal in vsp for P^n}, $F$ is not wild.   
\end{proof}
\end{corollary}

We finish off this short subsection with an example when $X = \PP^3$. Continuing the program started in \cite{HMV} to better understand $\VSPb$'s of wild forms, we offer the description of $\VSPb(F,10)$ where $F$ is the well-known {\it Ikeda quintic} $F = x_0^3x_1x_2+x_0x_1^3x_3+x_2^3x_3^2\in T_5$. The Ikeda quintic does not have vanishing Hessian but it has {\it second higher vanishing Hessian}; see \cite{MW} and \cite[Chapter 7]{russo}. This description was not accessible 
with the methods in \cite{HMV}.

\begin{example}[$\VSPb$ of the Ikeda quintic]
Let $F = x_0^3x_1x_2+x_0x_1^3x_3+x_2^3x_3^2\in T_5$. It is a wild form with $\underline{{\bf rk}}(F)=10$. 
We have $\VSPb(F, 10) \cong \PP^7$.
\begin{proof}
    Let $J = (\Ann(F)_{\leq 3})$. Observe that $J_{\geq 6}$ is a monomial ideal and $\HF(S/J, d) = h_{10,\PP^3}(d)$ for every $d\leq 6$.
    Therefore, if $I\in \VSPb(F, 10)$, then we can take the Gr\"obner degeneration of $I_{\geq 6}$ to obtain an ideal $I'\in \VSPb(F,10)$ such that
    $I'_{\leq 6} = J_{\leq 6}$ and $I'_{\geq 6}$ is a monomial ideal. There are $81$ ideals $I$ in $\mathrm{Hilb}_S^{h_{10,\PP^3}}$ such that  
    $I_{\leq 6} = J_{\leq 6}$ and $I_{\geq 6}$ is a monomial ideal. These are computed using the \texttt{Macaulay2} script presented below. None of them is saturated but only for $8$ of them, we have $\dim_\CC \mathrm{Ext}^1_S(\overline{I}/I, S/\overline{I})_0 \neq 0$. 
It follows from \cite[Theorem~3.4]{JM} that only these $8$ ideals might be in $\mathrm{Slip}_{10, \PP^3}$. These ideals are the elements of the set $A=\{J + (y_0^ay_1^b) \mid 0\leq a,b\text{ and }a+b = 7\}$.
If $I\in \VSPb(F, 10)$, then for every monomial order $\prec$ the ideal obtained from $I$ by deforming $I_{\geq 6}$ to a monomial ideal according to $\prec$ is in $A$. It follows that $I$ is contained in 
    \[
      \mathcal{X} = \{J + (G) \mid G \text{ is a nonzero element of } \langle y_0^7, y_0^6y_1, \ldots, y_1^7\rangle\}. 
    \]
    We show that in fact $\VSPb(F, 10) = \mathcal{X}$. Indeed, since $\underline{{\bf rk}}(F)=10$, by the border apolarity Theorem~\ref{mainbb}, $\VSPb(F,10) \neq \emptyset$.
    Therefore, by the above arguments, there exists $K\in A\cap \VSPb(10, F)$. We verify that the dimension of the tangent space to $\mathrm{Hilb}_S^{h_{10,\PP^3}}$ at every ideal from $A$ is $30$ which is equal to dimension of $\mathrm{Slip}_{10, \PP^3}$. It follows that every closed irreducible subset of $\mathrm{Hilb}_S^{h_{10,\PP^3}}$ passing through the point $K$ is contained in $\mathrm{Slip}_{10, \PP^3}$. In particular, $\mathcal{X} \subseteq \mathrm{Slip}_{10, \PP^3}$ which implies that $\VSPb(F, 10) = \mathcal{X}$.
\begin{small}    
\begin{mybox}
{\color{blue}
\begin{verbatim}
S=QQ[x_0,x_1,x_2,x_3]
m=ideal vars S;
r=10;
hrn = (i,r) -> (return min(hilbertFunction(i, ideal(0_S)), r););
cohrn = (i,r) -> (return (hilbertFunction(i, ideal(0_S)) - hrn(i,r)););
testInDegree = (i,r, I) -> (
  if (hilbertFunction(i, I) != hrn(i, r)) then return false;
  if (hilbertFunction(i+1, I) < hrn(i+1, r)) then return false;
  return true;
);
howManyNewGenerators = (i, r, I) ->(return hilbertFunction(i, I) - hrn(i,r););
addingInDegree = (i, r,E) -> (
  F:=flatten entries super basis(i, ideal vars S);
  G:={};
  for k from 0 to #E-1 do(
    H:=subsets(F, howManyNewGenerators(i, r, E#k));
    if (#H==0) then G=append(G, E#k);
    for l from 0 to #H-1 do (
      I:=E#k+ideal(H#l);
      if (testInDegree(i,r,I) == true) then G=append(G, I);   
    );
  );
  return G;
);
apolarIdeals  = (r,J) -> (
  E:={J}; F:={};
  for i from 0 to r do E=addingInDegree(i, r, E);
  for i from 0 to #E-1 do (if (dim E#i == 1) and (degree E#i == r) 
    then F=append(F, E#i));
  return F;
);
F=x_0^3*x_1*x_2 + x_0*x_1^3*x_3 + x_2^3*x_3^2
e=select(flatten entries mingens(inverseSystem(F)), i -> ((degree i)#0 < 4 ));
E=apolarIdeals(10,ideal e);
\end{verbatim}
}
\end{mybox}    
\end{small}

\end{proof}
\end{example}

\begin{remark}
Let $X = \PP^n$. Let $F\in T_d$ be such that  $\mathrm{HF}(\Ann(F),s)=\underline{{\bf rk}}(F)= \binom{s+n}{s}$ for 
some $1\leq s\leq \lfloor d/2\rfloor$ and such that $\Ann(F)$ has unimodal Hilbert function. In particular, $F$ is $s$-concise, i.e. $\Ann(F)_{\leq s} = 0$. The condition that the higher Hessian $\mathrm{Hess}^{(s,s)}(F)$ vanishes implies that $F$ is wild \cite[Theorem 3.14]{DG}. We do not know whether the converse is true as in the minimal border rank regime 
\cite[Theorem 4.9]{HMV}. When $n=2$, if an $F$ satisfying this assumption existed, it would solve the problem of the first author \cite[Problem 1]{Man} in the positive. However, proving the equivalence and finding such an $F$ would also solve a long-standing open problem of Maeno and Watanabe in the negative \cite[Remark 5.1]{MW}: there would exist Gorenstein algebras $A=S/\Ann(F)$ with $\dim A_1=3$ {\it without} the Strong Lefschetz Property \cite[Corollary 7.2.21]{russo}.  
\end{remark}

\section{Botany of forms and their $\VSPb$'s}\label{sec:botany of forms}

In this section, we assume $X = \PP^n$, so $S$ and $T$ are polynomial rings with the standard grading. 
The surjective morphism $\phi_{r,\PP^n}: \mathrm{Slip}_{r,\PP^n}\rightarrow \mathrm{Hilb}_{sm}^{r}(\PP^n)$ 
is given on closed points by $J\mapsto \mathrm{Proj}(S/J)$.

\subsection{Binary forms}\label{ssec:binary}
Here let $n=1$. It follows from \cite[Lemma~4.1]{hs} that $\mathrm{Hilb}_S^{h_{r,\PP^1}}$ and $\mathrm{Hilb}^r(\PP^1)$ represent the same functor. In particular, $\mathrm{Hilb}_S^{h_{r,\PP^1}} \cong \mathrm{Hilb}^r(\PP^1) \cong \PP^r$ where the latter isomorphism is well-known \cite[Proposition 7.3.3~and~Example~7.1.3]{Fan05}. 

\begin{proposition}\label{prop:binary forms}
Let $F\in T_d$ be a binary form. Then $\mathrm{crk}(F) = \mathrm{srk}(F) = \brk(F)=r$ and $\VPS(F,r)\cong \VSPb(F,r)$
where the isomorphism is induced by the isomorphism $\phi_{r,\PP^1}$ and either $\VPS(F,r)\cong \PP^1$ or it is a point. If $\VPS(F,r) \cong \PP^1$, then $d$ is even. Moreover, if $\rk(F)=r$ then  $\VSP(F,r)\cong \VPS(F,r)$. 
\begin{proof}
By a theorem of Sylvester, for any $F\in T_d$ its annihilator is a complete intersection $\Ann(F)=(g_1,g_2)\subset S$ with $\deg(g_1)\leq \deg(g_2)$ and $\deg(g_1)+\deg(g_2) = d+2$. 
It is clear that $r:=\min\lbrace \deg(g_1), \deg(g_2)\rbrace = \deg(g_1) = \mathrm{crk}(F)$. Since $\mathrm{Hilb}^r(\PP^1)$ is irreducible, one has 
$\mathrm{crk}(F) = \mathrm{srk}(F)$. By the border apolarity Theorem \ref{mainbb}, we have $\brk(F)=r$ as well
and the equality is proven. 

Either $\deg(g_1) = \deg(g_2)$ or not. In the first case, $\dim_{\CC} \Ann(F)_{\deg(g_1)}=2$ and
so any $I = (g)\subset \Ann(F)$ with $\deg(g) = \deg(g_1)=r$ is a point in $\VPS(F,r)$. We have a morphism 
\[
\PP^1 = \PP(\Ann(F)_{\deg(g_1)})\longrightarrow \VPS(F,r)
\]
defined on closed points by $g \mapsto \mathrm{Proj}(S/(g))$, which is an isomorphism. Note that $d = 2(\deg(g_1)-1)$ is even.  In the second case, the only point of $\VPS(F,r)$ is given by $I=(g_1)$. 

If $\rk(F) = r$, then $\Ann(F)_{\deg(g_1)}$ contains a square-free form. If we are in the first case, then $\VSP^{0}$ which is set-theoretically given by all the square-free forms inside the pencil $\PP(\Ann(F)_{\deg(g_1)})$, is dense in $\PP(\Ann(F)_{\deg(g_1)})$.  If we are in the second case, then the unique point of $\VPS(F,r)$ is a radical ideal. 
\end{proof}
\end{proposition}

\subsection{Cubic forms}\label{ssec:cubics}
\subsubsection{Ternary cubics}

\begin{proposition}\label{likegenericubic}
If $F$ is a plane cubic such that $\mathrm{Ann}(F)$ is generated by three quadrics, then $\VSPb(F,4)\cong \PP^2$. 
\begin{proof}
Let $\mathrm{Ann}(F) = (q_1,q_2,q_3)$. 
The quadrics $q_1,q_2,q_3$ form a regular sequence. Let $q'$ and $q''$ be two linear combinations of $q_1,q_2,q_3$ such that 
$\dim_\CC\langle q', q''\rangle = 2$.  Then $q'$ and $q''$ form a regular sequence, i.e. the ideal they generate $I=(q',q'')$ has codimension $2$
and it is saturated. Therefore, any choice of a $2$-dimensional subspace generates a saturated ideal with Hilbert function $h_{4, \PP^2}$. Hence $\VSPb(F,4)\cong \mathbb G(2,3)\cong \PP^2$.
\end{proof}
\end{proposition}

The $\mathrm{PGL}(3,\CC)$-orbits of ternary cubic forms with their ranks and border ranks are reported 
in Table \ref{table: ternary cubic forms} \cite[Theorem 8.1]{LT10}. 

\begin{table}
\begin{center}
\begin{tabular}{|l l c r|}
\hline
Description of $F$ & $F$ in normal form & $\mathrm{rk}(F)$ & $\brk(F)$ \\
\hline
triple line & $x_0^3$ & $1$ & $1$ \\  
\hline
three concurrent lines & $x_0x_1(x_0+x_1)$ & $2$ & $2$ \\  
\hline
double line + line & $x_0^2x_1$ & $3$ & $2$ \\  
\hline
irreducible Fermat & $x_1^2x_2 - x_0^3 -  x_2^3$ & $3$ & $3$ \\
\hline
irreducible & $x_1^2x_2 - x_0^3 - x_0x_2^2$ & $4$ & $4$ \\ 
\hline
cusp & $x_1^2x_2 - x_0^3$ & $4$ & $3$ \\
\hline
triangle & $x_0x_1x_2$ & $4$ & $4$ \\
\hline
conic + transversal line & $x_0(x_0^2+x_1x_2)$ & $4$ &$4$ \\
\hline
irreducible, smooth ($\lambda^3 \neq -27/4$ and $\lambda \neq 0$) & $x_1^2x_2 - x_0^3 - \lambda x_0x_2^2 - x_2^3$ & $4$ &$4$ \\
\hline
irreducible, singular ($\lambda^3 = -27/4$) & $x_1^2x_2 - x_0^3 - \lambda x_0x_2^2 - x_2^3$ & $4$ &$4$ \\
\hline
conic + tangent line & $x_1(x_0^2 + x_1x_2)$ & $5$ & $3$ \\
\hline
\end{tabular}
\caption{Ranks and border ranks of ternary cubic forms under $\mathrm{PGL}(3,\CC)$-action.}\label{table: ternary cubic forms}
\end{center}
\end{table}

\begin{theorem}\label{thm:ternary_cubic}
Let $F$ be any ternary cubic form. Then either $\VSPb(F,\brk(F))\cong \PP^2$ or 
it is a point. 
\begin{proof}
It is sufficient to prove the result for the normal forms presented in Table \ref{table: ternary cubic forms}.
When $\brk(F)\leq 2$, $F$ is a binary form and we apply Proposition \ref{prop:binary forms}. 
For the cases when $\brk(F) = 3$ (i.e. minimal border rank), we use Corollary~\ref{vspwithonlysat} for $n=2, d=3$. In the cases $\brk(F)=4$ (i.e. generic border rank), we employ Proposition~\ref{likegenericubic}.
\end{proof}
\end{theorem}

\subsubsection{Reducible cubics}\label{ssec: reducible cubics}
Concise reducible cubic forms in $T=T[X]$ for $X=\PP^n$ fall in the following four classes up to the $\mathrm{PGL}(n+1,\CC)$-action \cite{CCV16}: 
\begin{enumerate}

\item[(i)] $A = x_0(x_0^2+x_1^2+\cdots + x_n^2)$;

\item[(ii)] $B = x_0(x_1^2+x_2^2+\cdots + x_n^2)$;

\item[(iii)] $C = x_0(x_0x_1+x_2^2+\cdots+x_n^2)$;

\item[(iv)] monomials (which are all nonconcise if and only if $n\geq 3$).

\end{enumerate}

So the fourth cases appear only for $n\leq 2$, and hence their $\VSPb$'s are described by Theorem~\ref{thm:ternary_cubic}.
In this subsection we focus on cases (i)--(iii). The cubic form $B$ is the symmetrization of the {\it small Coppersmith-Winograd tensor} $T_{cw,n}$. The cubic form $C$ is the symmetrization of the {\it big Coppersmith-Winograd tensor} $T_{CW,n-1}$. 

\begin{proposition}\label{typeB}
Let $n\geq 2$ and let $B\in T_3$ be the symmetrization of the small Coppersmith-Winograd tensor. Then 
$\VSPb(B,n+2)$ is $\PP^2$ for $n=2$ and a single point for $n\geq 3$. 
\begin{proof}
By \cite[Proposition~3.4.9.1]{Landsberg17} we have $\brk(B) = n+2$.
For $n=2$, the statement is a special case of Proposition \ref{likegenericubic}. Let $n\geq 3$ and let $I\subset \Ann(B)$ be an ideal with generic Hilbert function $h_{n+2,\PP^n}$. Then $I_{2}$ has codimension one in $\Ann(B)_2$. 
Write $\Ann(B)=(y_0^2,y_iy_j,y_i^2-y_n^2 \mid 1\leq i < j \leq n)$. 
Fix a graded lexicographic monomial ordering with $y_0\succ \cdots \succ y_n$. 

We now prove that $y_0^2\notin J =\mathrm{in}_{\succ}(I)$. If not, then $J_2 = \langle y_0^2, m\in \mathcal M\rangle$, where $\mathcal M$ is the set of the initial monomials of the generators of $\Ann(B)$ except one. 
Suppose $y_iy_j\notin \mathcal M$ with $i\neq j$. Then $y_{\ell}^2\in J$ for all $0\leq \ell\leq n-1$. If $i,j\neq n$, then the quotient $S_k/J_k$ for $k\geq 4$ is spanned by at most two monomials: $y_0y_n^{k-1}$ and $y_n^k$. Thus $\mathrm{HF}(S/I,k) \leq 2$ for $k\geq 4$, a contradiction. If $y_iy_n\notin \mathcal M$, then the quotient $S_k/J_k$ for $k\geq 3$ is spanned by at most four monomials: $y_0y_iy_n^{k-2}, y_iy_n^{k-1}, y_0y_n^{k-1}$ and $y_n^k$. Thus $\mathrm{HF}(S/I,k) \leq 4<n+2$ for $k\geq 3$, a contradiction. If $y_i^2\notin \mathcal M$ for some $0<i<n$, then the quotient 
$S_k/J_k$ for $k\geq 3$ is spanned by at most four monomials: $y_0y_i^{k-1},y_0y_n^{k-1},y_i^{k},y_n^{k}$. Thus $\mathrm{HF}(S/I,k)\leq 4<n+2$ for $k\geq 3$, a contradiction. 

Therefore $y_0^2\notin J$.
This implies that $I =(y_iy_j,y_i^2-y_n^2)$, which is the unique point of $\VSPb(F,n+2)$. Indeed, it is easy to check that the basis of the quotient $S_k/I_k$ for $k\geq 3$ is 
spanned by $y_0^k, y_0^{k-1}y_1,\ldots, y_0^{k-1}y_{n},y_0^{k-2}y_1^2$.
\end{proof}
\end{proposition}

\begin{proposition}\label{typeC}
Let $n\geq 2$ and let $C\in T_3$ be the symmetrization of the big Coppersmith-Winograd tensor. Then 
$\VSPb(C,n+1)$ is a single point.
\begin{proof}
By \cite[Exercise 3.4.9.3]{Landsberg17} we have $\brk(C)=n+1$, i.e. $C$ has minimal border rank. Corollary~\ref{vspwithonlysat} shows that $\VSPb(F,n+1)$ contains a single point which is the saturated ideal $(\Ann(C)_2)$. 
\end{proof}
\end{proposition}

\begin{proposition}\label{typeA}
Let $n\geq 2$ and let $A\in T_3$ be as above. Then 
$\VSPb(A,n+2)$ is $\PP^2$ for $n=2$ and a single point for $n\geq 3$.
\begin{proof}
For $n=2$, this follows from Proposition \ref{likegenericubic}.
Assume that $n\geq 3$. Since $\Ann(F)$ is generated by quadrics we have $\brk(A) \geq n+2$. On the other hand, the completely analogous proof to the one for the cubic $B$ shows that $I =(y_iy_j,y_i^2-y_n^2)$ is the unique apolar ideal of $A$ with generic Hilbert function $h_{n+2, \PP^n}$. Furthermore, it is in $\mathrm{Slip}_{n+2, \PP^n}$. We conclude that $\brk(A) = n+2$ and $\VSPb(A, n+2) = \{I\}$. 
\end{proof}
\end{proposition}

\subsection{Ternary monomials}\label{ssec:monomials}

We shift gears to monomials in $T=\CC[x_0,x_1,x_2]$. Before proceeding, we record a result about $\VSP$'s of monomials with the same exponents, which seems to 
be not explicitly written anywhere in the literature. 

\begin{proposition}[{\cite[Proposition~25]{bbt13}}]\label{vspmonomials}
Let $F = x_0^k\cdots x_n^k$. Then  $\VSP(F, \rk(F))\cong \PP^{n}$. 
\begin{proof}
The dimension of $\VSP(F, \rk(F))$ follows from \cite[Proposition~25]{bbt13}. By \cite[Theorem 4]{bbt13}, $\VSP^{0}(F, \rk(F))\cong \mathbb (\CC^{*})^n$, and hence its closure $\VSP(F,\rk(F))$ is irreducible. 
More explicitly, the open variety of sums of powers $\VSP^{0}(F, \rk(F))\subset \VSP(F, \rk(F))$ may be described as the locally closed set of zero-dimensional schemes defined
by ideals of the form 
$(y_0^{k+1}-\alpha_0y_i^{k+1},\ldots, y_n^{k+1}-\alpha_n y_i^{k+1})\subset \CC[y_0,\ldots,y_n]$ with $\prod_{i=1}^n \alpha_i\neq 0$. We show that $\VSP(F, \rk(F))$ contains an isomorphic copy of $\PP^n$ as a subvariety. This is enough as they are both $n$-dimensional  and $\VSP(F,\rk(F))$ is irreducible. 

For each $0\leq i\leq n$, let $U_i\subset \mathrm{Hilb}_{sm}^{\rk(F)}(\PP^n)$ be the subset whose points are the zero-dimensional schemes corresponding to the ideals of the form $I_i=(y_0^{k+1}-\alpha_0y_i^{k+1},\ldots, y_n^{k+1}-\alpha_n y_i^{k+1})$ with $\alpha_j\in \CC$. One has $U_i\subset \VSP(F, \rk(F))$ because the latter is the closure of $\VSP^{0}(F, \rk(F))$ in $\mathrm{Hilb}_{sm}^{\rk(F)}(\PP^n)$. 

Let $\mathbb A^n\cong H_i\subset \PP^n$ be defined by $H_i = \lbrace \lambda_i\neq 0\rbrace = \left\lbrace \left(\frac{\lambda_0}{\lambda_i}, \ldots, 1, \ldots, \frac{\lambda_n}{\lambda_i}\right)\right\rbrace$, for each $0\leq i\leq n$. Then, for each $0\leq i\leq n$, define the map $\psi_i: H_i\rightarrow U_i$ given by 
\[
\left(\frac{\lambda_0}{\lambda_i}, \ldots, 1, \ldots, \frac{\lambda_n}{\lambda_i}\right)
\mapsto 
\left(y_0^{k+1}-\frac{\lambda_0}{\lambda_i}y_i^{k+1},\ldots, y_n^{k+1}-\frac{\lambda_n}{\lambda_i} y_i^{k+1}\right).
\]
This is an isomorphism of $n$-dimensional affine spaces. The maps $\psi_i$ and $\psi_j$ glue together on the intersections $H_i\cap H_j$ giving an isomorphism $\psi$ between $\PP^n$ and its image $\psi(\PP^n)\subseteq \VSP(F,\rk(F))$. This concludes the proof.   
\end{proof}
\end{proposition}

Let $F = x_0^ax_1^bx_2^c$ be a monomial with exponents $0<a\leq b\leq c$ and let $r=(a+1)(b+1)=\brk(F)$. In the following we will give an explicit description of $\VPS(F,r)$ and show that $\VSPb(F,r)$ is of fiber type, i.e. $\VSPb(F,r) = \phi^{-1}_{r,\PP^2}(\VPS(F,r))$. 
This is proven in Proposition \ref{prop:case c>= a+b}, Theorem \ref{thm:case c=a+b-1} and Theorem \ref{prop:vps_vspb_>=-2}. 

We shall assume familiarity with Macaulay's representation of integers \cite[\S 49]{Peeva} and Macaulay's bound theorem \cite[Theorem 49.5]{Peeva}:
for any graded ideal $J$ in a standard graded polynomial ring $S$, one has that $\mathrm{HF}(S/J,k+1)\leq \mathrm{HF}(S/J,k)^{\langle k\rangle}$. 

\subsubsection{Monomials $x_0^ax_1^bx_2^c$ with $c\geq a+b-1$}

\begin{lemma}\label{lem: T/J in deg a+b}
Let $J = (y_0^{a+1},y_1^{b+1})$ for some positive integers 
$a,b$. We have $\mathrm{HF}(S/J,a+b) = r$. 
\begin{proof}
This is easily seen either from a monomial counting, as in the proof of Theorem \ref{borderrankmonimprove}, or as follows. 
Consider the $a+b$ degree piece of the minimal free resolution (a Koszul resolution) as an $S$-module of $S/J$:
\[
0\longrightarrow S(-a-b-2)\longrightarrow S(-a-1)\oplus S(-b-1)\longrightarrow S\longrightarrow S/J\longrightarrow 0. 
\]
Then one finds $\mathrm{HF}(S/J,a+b) = \mathrm{HF}(S,a+b)-\mathrm{HF}(S(-a-1),a+b)) - \mathrm{HF}(S(-b-1),a+b) = r$. 
\end{proof}
\end{lemma}

\begin{proposition}\label{prop:case c>= a+b}
If $c\geq a+b$, then $\VPS(F,r)=\lbrace \mathrm{Proj}(S/J)\rbrace$ where
$J = (y_0^{a+1}, y_1^{b+1})$ and 
$\VSPb(F,r)$ is of fiber type. 
\begin{proof}
We show that if $K\subset \Ann(F)$ is a homogeneous ideal with Hilbert polynomial $r$ and Hilbert function bounded 
from above by $r$, then $\overline{K}=J$. By assumption on the 
Hilbert function and Lemma~\ref{lem: T/J in deg a+b},
we have $K_{a+b} = \Ann(F)_{a+b}$. Since $\Ann(F)_{a+b} = J_{a+b}$, 
it follows that $(J_{a+b})\subset K$. Since $J_{\geq a+b} = (J_{a+b})$ and $J,K$ have the same Hilbert polynomial we conclude that $\overline{K}=\overline{J}=J$. The inclusion $\VSPb(F,r) \subseteq \phi_{r,\PP^2}^{-1}(\VPS(F,r))$ follows. Lemma \ref{lem:trivial_containment} establishes 
that the containment is an equality. 
\end{proof}
\end{proposition}

\begin{lemma}\label{lem: Hilbpoly at most r}
Suppose that $I$ is a monomial ideal contained in $J = (y_0^{a+1},y_1^{b+1},y_2^{c+1})$ with $0<a, b\leq c$. Let $d\geq \max\{c+1, a+2\}$ and assume that $I_d$ contains $(y_1^{b+1},y_2^{c+1})_d$ and has codimension $h$ in $J_d$. If $h^{\langle d-a-1\rangle}\leq h$ then the Hilbert polynomial of $S/I$ is at most $h$. 
\begin{proof}
Let $K$ be the monomial ideal generated by all monomials of degree $d-a-1$ that multiplied by $y_0^{a+1}$ belong to $I_d$. Since $I_d$ is of codimension $h$ in $J_d$ and it contains $(y_1^{b
+1},y_2^{c+1})_d$, there are exactly $h$ monomials in $(y_0^{a+1})_d$ that are not in $I_d$. It follows that $K_{d-a-1}$ is of codimension $h$. Therefore, by Macaulay's bound $\mathrm{HF}(S/K,k)\leq h$ for $k\gg 0$. Since $(y_1^{a+1},y_2^{b+1})_d\subset I_d$,
we have $\mathrm{HF}(S/I,k)\leq \mathrm{HF}(S/K,k-a-1)\leq h$ for $k\gg 0$. 
\end{proof}
\end{lemma}

\begin{lemma}\label{lem:containments}
Let $c=a+b-1$. If $I\subset \Ann(F)$ is a homogeneous ideal with Hilbert function of $S/I$ bounded above by $r$ and the Hilbert polynomial of $S/I$ equals $r$, then 
$I_{a+b}$ is of codimension $1$ in $\Ann(F)_{a+b}$. 
Furthermore, if $b>1$ \textnormal{(}respectively $a>1$\textnormal{)}
then $I_{a+b}$ contains $(y_0^{a+1})_{a+b}$
\textnormal{(}respectively is equal to $(y_0^{a+1},y_1^{b+1})_{a+b}$\textnormal{)}. 
\begin{proof}
Note that, by definition, $r\geq 4$. By Lemma \ref{lem: T/J in deg a+b}, we obtain that
$\mathrm{HF}(S/\Ann(F),a+b)=r-1$. The claim about the codimension of $I_{a+b}$ follows easily from this. Assume that $b>1$ and that $(y_0^{a+1})_{a+b}$ is not contained in $I_{a+b}$. Let $\mathcal M$ be the set of all degree $a+b$ monomials in $\Ann(F)$ and let $m\in \mathcal M$ be a monomial in $(y_0^{a+1})_{a+b}$ that is not in $I_{a+b}$. Since $I_{a+b}$ is of codimension $1$ in $\Ann(F)_{a+b}$ we deduce that for every $m'\in \mathcal M\setminus\lbrace m\rbrace$, there exists $\alpha\in \CC$
such that $(m'+\alpha m)_{m'\in \mathcal M\setminus\lbrace m\rbrace}$ is a $\CC$-vector space basis of $I_{a+b}$. We consider the initial ideal $J=\mathrm{in}_{\prec}(I)$ of $I$ with respect to the graded reverse lexicographic order with $y_0\prec y_1\prec y_2$. If $m'\in (y_1^{b+1},y_2^{a+b})$ then $m\prec m'$ and we conclude that $J$ contains $m'$. Therefore, it is a monomial ideal containing $y_2^{a+b}, (y_1^{b+1})_{a+b}$ and all but one monomial in $(y_0^{a+1})_{a+b}$. By Lemma \ref{lem: Hilbpoly at most r} applied for $h=1$, the Hilbert polynomial of $S/J$ is at most $1$, which is a contradiction. If $a>1$ then $b>1$, so the argument above shows that $(y_0^{a+1})_{a+b}$ is contained in $I_{a+b}$. If $(y_1^{b+1})_{a+b}$ is not contained in $I_{a+b}$ then, as before, using Lemma~\ref{lem: Hilbpoly at most r} with the roles of $y_0$ and $y_1$ exchanged, we conclude that 
$S/J$ has Hilbert polynomial at most $1$. 
\end{proof}
\end{lemma}

\begin{theorem}\label{thm:case c=a+b-1}
If $c = a+b-1$, then $\VSPb(F,r)$ is of fiber type. Furthermore,
\begin{enumerate}
\item[(i)] if $a=b=1$, then $\VSPb(F,4)\cong \VPS(F,4)\cong \PP(\Ann(F)_2)=\PP^2$;

\item[(ii)] if $a=1<b$, then $\VPS(F,2(b+1))=\lbrace \mathrm{Proj}(S/(y_0^2,sy_1^{b+1}+ty_2^{b+1})), [s:t]\in \PP^1\rbrace\cong \PP^1$;

\item[(iii)] if $a\geq 2$, then $\VPS(F,r) = \lbrace \mathrm{Proj}(S/(y_0^{a+1},y_1^{b+1}))\rbrace$. 
\end{enumerate}
\begin{proof}
Assume that $a=b=1$. If $I\in \VSPb(F,4)$ or $I = \overline{I}$ and $\mathrm{Proj}(S/I)\in \VPS(F,4)$, then by Lemma \ref{lem:containments} we derive that 
$I_2$ is of codimension $1$ in $\Ann(F)_2$. Since any choice of a codimension one subspace of $\Ann(F)_2$ gives a saturated ideal with Hilbert function $h_{4,\PP^2}$, the proof of this case is completed. 

Assume that $a=1$ and $b>1$. Then $b=c$. Let $I\subset \Ann(F)$ be a homogeneous ideal such that $S/I$ has Hilbert polynomial $2(b+1)$. Assume that 
$I\in \VSPb(F, 2(b+1))$ or $I = \overline{I}$. By Lemma \ref{lem:containments}, we obtain that $I_{1+b}$ is a subspace of codimension $1$ of $\Ann(F)_{1+b}$ that contains $(y_0^2)_{1+b}$. It follows that it is of the form 
\[
I_{1+b} = (y_0^2)_{1+b} + (sy_1^{b+1}+ty_2^{b+1})_{1+b}, \mbox{ for some } [s:t]\in \PP^1.
\]
We see that $\overline{(I_{1+b})}$ is equal to $(y_0^2, sy_1^{b+1}+ty_2^{1+b})$ which is contained in $\Ann(F)$ and defines a length $2(b+1)$ subscheme of $\PP^2$. It follows that $\overline{I}=(y_0^2,sy_1^{b+1}+ty_2^{b+1})$. This describes 
$\VPS(F,2(b+1))$. The other containment between $\VSPb(F, 2(b+1))$ and $(\phi_{r,\PP^2})^{-1}(\VPS(F,2(b+1)))$ follows from Lemma~\ref{lem:trivial_containment}. 

Assume $a\geq 2$. Let $I\subset \Ann(F)$ be an ideal such that $S/I$ has Hilbert polynomial $r$. Assume that $I\in \VSPb(F,r)$ or $I=\overline{I}$. By Lemma \ref{lem:containments}, we get $I_{a+b}=(y_0^{a+1},y_1^{b+1})_{a+b}$. It follows that 
$\overline{I} = (y_0^{a+1},y_1^{b+1})$. Hence
$\VPS(F,r)$ is as claimed and $\VSPb(F,r)$ is contained in the fibre of $\phi_{r,\PP^2}$ over the unique point of $\VPS(F,r)$. Again, the other containment follows from Lemma \ref{lem:trivial_containment}.
\end{proof}

\end{theorem}

\subsubsection{Monomials $x_0^ax_1^bx_2^c$ with $c\leq a+b-2$}
Let $d=a+b-c$ and assume that $d\geq 2$. Suppose that $I\subseteq \mathrm{Ann}(F)$ is a monomial ideal such that $\HF(S/I, a+b) = \HF(S/I, a+b+1)= r$. We have
$\mathrm{Ann}(F)_{a+b} = y_0^{a+1}S_{b-1} \oplus y_1^{b+1}S_{a-1}\oplus y_2^{c+1}S_{d-1}$ and
$\mathrm{Ann}(F)_{a+b+1} = y_0^{a+1}S_b \oplus y_1^{b+1}S_{a}\oplus y_2^{c+1}S_{d}$. Therefore, there are $V_0\subseteq S_{b-1}, V_1\subseteq S_{a-1},V_2\subseteq S_{d-1}$ and $W_0\subseteq S_{b}, W_1\subseteq S_{a}, W_2\subseteq S_{d}$ such that
\[
I_{a+b} = y_0^{a+1}V_0 \oplus y_1^{b+1}V_1\oplus y_2^{c+1}V_2 \text{ and } I_{a+b+1} = y_0^{a+1}W_0 \oplus y_1^{b+1}W_1\oplus y_2^{c+1}W_2, 
\]
where $S_1V_i \subseteq W_i$ for $i=0,1,2$. Let $h_i = \mathrm{codim}(V_i)$ and $H_i =\mathrm{codim}(W_i)$ (where codimensions 
are taken in the respective vector spaces). 
From the assumption on the Hilbert function of $S/I$ we get $h_0+h_1+h_2 = \dim_\CC S_{d-1}$ and $H_0+H_1+H_2 = \dim_\CC S_{d}$. To see this, note that 
\[
h_0+h_1+h_2 = \dim_{\CC} S_{b-1}+\dim_{\CC} S_{a-1} + \dim_{\CC} S_{d-1} - (\dim_{\CC} V_0 + \dim_\CC V_1 + \dim_\CC V_2)=
\]
\[
= \dim_{\CC} S_{b-1}+\dim_{\CC} S_{a-1} + \dim_{\CC} S_{d-1} - \dim_\CC S_{a+b} +r = \dim_\CC S_{d-1}. 
\]
The second equality can be proved in a similar way.

By Macaulay's bound we have 
\[
H_0 \leq h_0^{\langle b-1\rangle},  H_1 \leq h_1^{\langle a-1\rangle} \text{ and } H_2\leq h_2^{\langle d-1\rangle}.
\]
We claim that two out of the three values of $h_0,h_1,h_2$ are zero. We use the following result about Macaulay exponents.

\begin{lemma}[{\cite[Lemma~6.12]{bb19}}]\label{bblem}
Suppose $q,r,\ell,e$ are nonnegative integers with $\ell\geq e >0$. Then 
\[
q^{\langle \ell \rangle} + r^{\langle e \rangle}\leq (q+r)^{\langle e \rangle}.
\]
\end{lemma}

We apply Lemma~\ref{bblem} with $q+r=\dim_\CC S_{d-1}$ and $e=d-1$. In that case, we observe that the inequality from the lemma is strict if $\ell = e$ and $q,r > 0$ (Lemma~\ref{lem:numerics_of_MC})
or if $r=0$ and $\ell > d-1$ (Lemma~\ref{lem:numerics_of_MC2}). 

\begin{lemma}\label{lem:numerics_of_MC}
If $s$ and $t$ are positive integers with $s+t = \dim_\CC S_{d-1}$, then $s^{\langle d-1 \rangle} + t^{\langle d-1 \rangle} < \dim_\CC S_d.$ 
\end{lemma}
\begin{proof}
Let 
\[
s = \sum_{i=1}^{d-1} \binom{s(i)}{i} \text{ and } t = \sum_{i=1}^{d-1} \binom{t(i)}{i}
\]
be the Macaulay representations of the integers $s$ and $t$. Assume that $0<s<\dim_\CC S_{d-1}=\binom{d+1}{2} = \binom{d+1}{d-1}$. Hence
$s(d-1) \leq d$ and $t(d-1)\leq d$. As a result, for every $i$ we have $s(i)\leq i+1$ and $t(i)\leq i+1$.
It follows that $s^{\langle d-1 \rangle} - s$ is the number $n_s$ of those $i$ for which $s(i) = i+1$. We define $n_t$ analogously.
We have $\dim_\CC S_d - \dim_\CC S_{d-1} = d+1$ so we need to show that $n_s + n_t < d+1$. By Lemma \ref{bblem}, we have $n_s+n_t\leq d+1$ so we may assume that $n_s > 0$ and $n_t = (d+1-n_s)$ and argue by contradiction. By monotonicity of the coefficients $s(i)$ and $t(i)$ we get 
$s \geq \sum_{i=d-n_s}^{d-1} \binom{i+1}{i}  = \binom{d+1}{2} - \binom{d-n_s+1}{2}$. Similarly, $t \geq \binom{d+1}{2} - \binom{d-n_t+1}{2} = \binom{d+1}{2} - \binom{n_s}{2}$. We claim that the sum $s+t$ is then strictly larger than $\dim_\CC S_{d-1} = \binom{d+1}{2}$.
It is sufficient to show that $\binom{d+1}{2} > \binom{d-n_s+1}{2} + \binom{n_s}{2}$, which is implied by the inequality
$n_s(d-n_s+1) = n_sn_t> 0.$ 
\end{proof}

\begin{lemma}\label{lem:numerics_of_MC2}
If $\ell > d-1$, then $(\dim_\CC S_{d-1})^{\langle \ell \rangle} < (\dim_\CC S_{d-1})^{\langle d-1 \rangle}$.
\end{lemma}
\begin{proof}

Let 
\[
\dim_\CC S_{d-1} = \binom{d+1}{2} = \sum_{i=1}^{\ell} \binom{s(i)}{i}
\]
be the Macaulay representation. We consider two cases. If $\ell \geq \dim_\CC S_{d-1}$, then for every $i\in \{1,\ldots, \ell\}$ we have $s(i) \leq i$, so  
$(\dim_\CC S_{d-1})^{\langle \ell \rangle} = \dim_\CC S_{d-1} <  \dim_\CC S_{d} = (\dim_\CC S_{d-1})^{\langle d-1 \rangle}$.

Assume that $\ell < \dim_\CC S_{d-1}$. It follows that $s(\ell) \geq \ell+1$.  Since $\ell > d-1$ we have in fact $s(\ell) =\ell+1$. 
Thus, there are integers $k$ and $k'$ with $\ell \geq k \geq k' \geq 1$ such that 
$s(i) = i+1$ for every $k\leq i \leq \ell$, $s(i) = i$ for every $k'\leq i \leq k-1$ and $s(i)= i-1$ for every $1\leq i \leq k'-1$.
As a result, $(\dim_\CC S_{d-1})^{\langle \ell \rangle} - \dim_\CC S_{d-1} = \ell-k+1$.
We have $(\dim_\CC S_{d-1})^{\langle d-1 \rangle} - \dim_\CC S_{d-1} = d+1$ so to conclude the proof we need to show that $\ell - k +1 < d+1$. Assume by contradiction that 
$k\leq \ell-d$. We get
\[
\binom{d+1}{2} = \dim_\CC S_{d-1} \geq \sum_{i=k}^\ell \binom{i+1}{i} = \binom{\ell+2}{2} - \binom{k+1}{2}
\geq  \binom{\ell+2}{2} - \binom{\ell-d+1}{2}.
\]
We claim that the last expression is strictly larger than $\binom{d+1}{2}$ which gives a contradiction. Indeed, after rearranging the terms we obtain an equivalent inequality $(d+1)(\ell -d+1 ) > 0$ which is true by the assumption that $\ell > d-1$.
\end{proof}

\begin{lemma}\label{lem:inequality_of_MC}
If $h_0,h_1,h_2$ are nonnegative integers with $h_0+h_1+h_2 = \dim_\CC S_{d-1}$ and $h_0^{\langle b-1\rangle} + h_1^{\langle a-1\rangle} + h_2^{\langle d-1 \rangle} \geq \dim_\CC S_{d}$, then $\{h_0,h_1,h_2\} = \{0, \dim_\CC S_{d-1}\}$. 
Furthermore, 
\begin{enumerate}
    \item[(i)] if $b<c$, then $h_0 = h_1 = 0$;
    \item[(ii)] if $a < b$, then $h_0 = 0$.
\end{enumerate}
\end{lemma}
\begin{proof}
By assumptions we have $d-1 = a+b-c-1 \leq a-1 \leq b-1$.
By Lemma \ref{bblem}, it follows that
\begin{equation}\label{eq:MC}
h_0^{\langle b-1\rangle} + h_1^{\langle a-1\rangle} + h_2^{\langle d-1 \rangle} \leq h_0^{\langle d-1\rangle} + h_1^{\langle d-1\rangle} + h_2^{\langle d-1 \rangle}.
\end{equation}
Furthermore, by Lemma \ref{bblem} we get 
\begin{equation}\label{eq:mac_coeff_ineq}
h_0^{\langle d-1\rangle} + h_1^{\langle d-1\rangle} + h_2^{\langle d-1 \rangle}\leq (h_0+h_1+h_2)^{\langle d-1 \rangle} =  (\dim_\CC S_{d-1})^{\langle d-1 \rangle} = \dim_\CC S_{d}.
\end{equation}
We show that if at least two of the numbers $h_0,h_1$ and $h_2$ are nonzero then the inequality in  \eqref{eq:mac_coeff_ineq} is strict. 
Using Lemma \ref{bblem} again we can reduce to the numerical problem considered in Lemma~\ref{lem:numerics_of_MC}. Therefore, exactly one of $h_0, h_1$ and $h_2$ is nonzero. Suppose we have a counterexample to (i) or (ii). Then inequality \eqref{eq:MC} is strict by Lemma~\ref{lem:numerics_of_MC2}. Using inequality \eqref{eq:mac_coeff_ineq} we obtain a contradiction.
\end{proof}

\begin{lemma}\label{lem:apolar_monomial_ideals}
Let $I\subseteq \mathrm{Ann}(F)$ be a monomial ideal with $\HF(S/I, a+b) = \HF(S/I,a+b+1) = r$.
If $b<c$, then $I_{a+b} = (y_0^{a+1}, y_1^{b+1})_{a+b}$. If $a < b=c$, then $I_{a+b} = (y_0^{a+1}, y_1^{b+1})_{a+b}$ or $I_{a+ b} = (y_0^{a+1}, y_2^{b+1})_{a+b}$. If  $a=b=c$, then $I_{a+b}$ is one of the three $(y_0^{a+1}, y_1^{a+1})_{a+b}$, $(y_0^{a+1}, y_2^{a+1})_{a+b}$ or $(y_1^{a+1}, y_2^{a+1})_{a+b}$.
\end{lemma}
\begin{proof}
Let $V_i, W_i, h_i$ and $H_i$ be as defined at the beginning of the subsection. 
We have $h_0+h_1+h_2 = \dim_\CC S_{d-1}$ and by Lemma~\ref{lem:inequality_of_MC} two of these numbers are zero. 
If $c > b$, then Lemma~\ref{lem:inequality_of_MC}(i) gives also $h_0=h_1=0$.
If $c=b > a$, then $h_0 = 0$ and this leads to the two cases considered above.
If $a=b=c$, there are three possible choices depending on which two of the numbers $h_0,h_1$ and $h_2$ are zero.
\end{proof}

\begin{lemma}\label{lem:apolar_ideals}
Assume that $F=x_0^ax_1^bx_2^c$ with $1\leq a\leq b\leq c$ and $a+b-c \geq 2$ and that $I\subseteq \Ann(F)$ is a homogeneous ideal such that $\HF(S/I, a+b)=\HF(S/I, a+b+1)=r$.
\begin{enumerate}
    \item[(i)]\label{lempart1} If $b < c$, then $I_{a+b} = (y_0^{a+1}, y_1^{b+1})_{a+b}$.
    \item[(ii)]\label{lempart2} If $a < b=c$, then $I_{a+b}=(y_0^{a+1}, sy_1^{b+1}+ty_2^{b+1})_{a+b}$ for some $[s:t]\in \PP^1$.
    \item[(iii)]\label{lempart3} If $a=b=c$, then $I_{a+b} = (W)_{2a}$ for some $W\in \mathbb{G}(2, \langle y_0^{a+1}, y_1^{a+1}, y_2^{a+1}\rangle)$.
\end{enumerate}
\end{lemma}
\begin{proof}
    First we show that if $a < c$, then $(y_0^{a+1})_{a+b}\subseteq I_{a+b}$. Let $\prec$ be the graded reverse lexicographic monomial order with $y_0 \prec y_1 \prec y_2$. Since $F$ is a monomial, we have $\mathrm{in}_{\prec}(I)\subseteq \Ann(F)$ and by Lemma~\ref{lem:apolar_monomial_ideals} we have $(y_0^{a+1})_{a+b} \subseteq \mathrm{in}_{\prec}(I)$. Since every degree $a+b$ monomial in $(y_0^{a+1})$ is smaller than every monomial in $\Ann(F)_{a+b}$ that is not divisible by $y_0^{a+1}$ we conclude that $(y_0^{a+1})_{a+b} \subseteq I$. Analogously, we show that if $b<c$, then $(y_1^{b+1})_{a+b}\subseteq I_{a+b}$. Which finishes the proof of part (i) since $\dim_\CC I_{a+b} = \dim_\CC (y_0^{a+1}, y_1^{b+1})_{a+b}$.
    
For the remaining cases we consider the set $\mathcal{M} = \{m_1, \ldots, m_{\dim_{\CC} S_{a-1}}\}$ of all monomials in $S$ of degree $a-1$.  
    
Assume that $a < b=c$. By Lemma~\ref{lem:apolar_monomial_ideals} we may assume without loss of generality that $\left(\mathrm{in}_{\prec}(I)\right)_{a+b} = (y_0^{a+1}, y_2^{b+1})_{a+b}$ which implies that also 
$\left(\mathrm{in}_{\prec}(I)\right)_{a+b+1}=  (y_0^{a+1}, y_2^{b+1})_{a+b+1}$ since the two ideals have the same Hilbert function in degree $a+b+1$ and the latter has no minimal generators in that degree.
From the assumption on the degree $(a+b)$-part of the initial ideal we get $I_{a+b} = (y_0^{a+1},f_m \mid m\in \mathcal{M})_{a+b}$ where
\begin{equation}\label{eq1}
f_m = y_2^{b+1}m + g_m \text{ with } g_m \in y_1^{b+1}S_{a-1}.
\end{equation}
\begin{claim}\label{cl:1}
There exists $s\in \CC$ such that $g_m = sy_1^{b+1}m$ for every $m\in \mathcal{M}$.
\end{claim}
\begin{proof}[Proof of Claim~\ref{cl:1}]
    We first show that for every $m\in \mathcal{M}$, $g_m$ is of the form $s_my_1^{b+1}m$ for some $s_m\in \CC$. 
    For every $m\in \mathcal{M}$ let $h_m \in S_{a-1}$ be the polynomial for which $g_m = y_1^{b+1}h_m$.
    
    It is sufficient to show by induction on the degree $\delta$ of a monomial $p$ dividing $m$ that $p$ divides $h_m$. We start with $\delta=1$. Let $p=y_i$ and $m=y_im'$. Let $j\neq i$ be another element of the index set $\{0,1,2\}$ and let $\widehat{m} = y_jm'$.
    Consider the polynomial 
    \[
    y_jf_m-y_if_{\widehat{m}} =  y_jg_m - y_ig_{\widehat{m}} = y_1^{b+1}(y_jh_{m} - y_ih_{\widehat{m}}).
    \]
    Its initial monomial is in $\left(\mathrm{in}_{\prec}(I)\right)_{a+b+1}=  (y_0^{a+1}, y_2^{b+1})_{a+b+1}$ which implies that 
    $y_jh_{m} - y_ih_{\widehat{m}}=0$ since it is a polynomial of degree $a$. Thus, $p=y_i$ divides $h_m$.

    Assume that for every $m\in \mathcal{M}$ and for every monomial $q$ of degree $\delta \geq 1$ dividing $m$ the monomial $q$ divides $g_m$.
    Let $p$ be a monomial of degree $\delta+1$ that divides $m$. Let $m=pt$ and $y_i$ divides $p$ so that $p=y_iq$.
    Let $j\neq i$ be an element of the set $\{0,1,2\}$ and $\widehat{m}=y_jqt$. By induction, $q$ divides $h_m$ and $h_{\widehat{m}}$. Therefore, there are polynomials $k_m$ and $k_{\widehat{m}}$ such that $h_m = qk_m$ and $h_{\widehat{m}} = qk_{\widehat{m}}$.
    It follows that
    \[
    y_jf_m - y_if_{\widehat{m}} = y_jg_m - y_ig_{\widehat{m}} = y_1^{b+1}(y_jh_{m} - y_ih_{\widehat{m}}) = y_1^{b+1}q (y_jk_m - y_ik_{\widehat{m}}).
    \]
    Its initial monomial is in $\left(\mathrm{in}_{\prec}(I)\right)_{a+b+1}=  (y_0^{a+1}, y_2^{b+1})_{a+b+1}$ which implies that 
    $y_jk_{m} - y_ik_{\widehat{m}}=0$ since $q(y_jh_{m} - y_ih_{\widehat{m}})$ is of degree $a$. We conclude that $y_i$ divides $k_m$ so $p=y_iq$ divides $h_m = qk_m$.

    Having established that for every $m\in \mathcal{M}$, there is $s_m\in \CC$ such that $h_m = s_mm$. We are left with showing that all the constants coincide.
    This follows from the fact that in the inductive proof above we showed that $y_jh_m = y_ih_{\widehat{m}}$ if $\widehat{m}=\frac{y_j}{y_i}m$.
    Therefore, for such $m$ and $\widehat{m}$ we have $s_m = s_{\widehat{m}}$. Since we can connect any two monomials in $\mathcal{M}$ by a sequence of monomials each of which is obtained by changing exactly one linear factor we conclude that $s_m = s_{m'}$ for all $m,m'\in\mathcal{M}$.
    \end{proof}

    It follows from \eqref{eq1} and Claim~\ref{cl:1} that $I_{a+b} = (y_0^{a+1}, y_2^{b+1}+sy_1^{b+1})_{a+b}$ which is of the form described in the statement.

    Finally assume that $a=b=c$. By Lemma~\ref{lem:apolar_monomial_ideals} we may assume without loss of generality that $\left(\mathrm{in}_{\prec}(I)\right)_{2a} = (y_1^{a+1}, y_2^{a+1})_{2a}$. 
     It follows that $I_{a+b} = (f_m,f'_m \mid m\in \mathcal{M})_{a+b}$ where
\[
f_m = y_1^{a+1}m + g_m \text{ with } g_m \in y_0^{a+1}S_{a-1}
\]
and
\[
f'_m = y_2^{a+1}m + g'_m \text{ with } g'_m \in y_0^{a+1}S_{a-1}.
\]
Arguing as in Claim~\ref{cl:1} we conclude that there are $s,t \in \CC$ such that $g_m = sy_0^{a+1}m$ and $g'_m = ty_0^{a+1}m$. As a result
\[
f_m = m(y_1^{a+1} + sy_0^{a+1}) \text{ and }f'_m = m(y_2^{a+1} + ty_0^{a+1}) \text{ for every } m\in \mathcal{M}, 
\]
so $I_{2a} = (y_1^{a+1} + sy_0^{a+1}, y_2^{a+1}+ty_0^{a+1})_{a+b}$ which is of the claimed form.
\end{proof}

\begin{theorem}\label{prop:vps_vspb_>=-2}
Assume that $F=x_0^ax_1^bx_2^c$ with $1\leq a\leq b\leq c$ and $a+b-c\geq 2$. 
Then $\VSPb(F,r)$ is of fiber type. Furthermore, 
\begin{enumerate}
    \item[(i)] If $b < c$, then $\VPS(F,r)  = \{\mathrm{Proj}(S/(y_0^{a+1}, y_1^{b+1}))\}$.
    \item[(ii)] If $a < b=c$, then $\VPS(F,r)  = \{\mathrm{Proj}(S/(y_0^{a+1}, sy_1^{b+1}+ty_2^{b+1}))\mid [s:t]\in \PP^1\} \cong \PP^1$.
    \item[(iii)] If $a=b=c$, then $\VPS(F,r) = \{\mathrm{Proj}(S/(W)) \mid [W]\in \mathbb{G}(2, \langle y_0^{a+1}, y_1^{a+1}, y_2^{a+1}\rangle)\}\cong \PP^2$. In this case, $\VPS(F,r) = \VSP(F,r)$. 
\end{enumerate}
\end{theorem}
\begin{proof}
    We have $\phi_{r,\PP^2}^{-1}(\VPS(F,r)) \subseteq \VSPb(F,r)$ by Lemma~\ref{lem:trivial_containment}. Observe that the ideals $(y_0^{a+1}, y_1^{b+1})$, $(y_0^{a+1}, sy_1^{b+1}+ty_2^{b+1})$ and $(W)$ considered in Lemma~\ref{lem:apolar_ideals} are all generated in degree at most $a+b$, their quotient algebras have Hilbert function equal to $r$ in all degrees at least $a+b$ and the ideals are saturated. Indeed, the first property is clear and by taking an initial ideal and possibly renaming the variables it is sufficient to observe that the last two properties hold for the ideal $(y_0^{a+1}, y_1^{b+1})$.
    Therefore, if $I\in \VSPb(F,r)$, by Lemma~\ref{lem:apolar_ideals}, we have $I_{a+b} = J_{a+b}$ for a saturated ideal $J$ such that $\HF(S/J,e) = r$ for all $e\geq a+b$ and $(J_{\leq a+b})=J$. 
    It follows that  $I_{\geq a+b} = J_{\geq a+b}$ since the first ideal contains the latter and they have equal Hilbert functions. In particular, we have $\overline{I}=J$ and thus $\overline{I}$ is contained in $\Ann(F)$. This shows that $\VSPb(F,r)$ is of fiber type. 

    Let $\mathrm{Proj}(S/J)\in \VPS(F,r)$ for some saturated ideal $J$. Since $\VSPb(F,r)$ is of fiber type, there exists $I\in\VSPb(F, r)$ such that 
$\overline{I} = J$. Hence $J_{a+b} = (\overline{I})_{a+b}$ is of the form given in Lemma~\ref{lem:apolar_ideals}. It follows that $J$ is as claimed in the statement since $J = \overline{(J_{a+b})}$. The very last statement in (iii) follows from Proposition \ref{vspmonomials}. 
\end{proof}

The next example shows that when $F$ is a monomial with $r=\rk(F) = \brk(F)$ (i.e. $F = x_0^kx_1^kx_2^k$ for some $k\geq 1$), then $\VSP(F,r)$ and $\VSPb(F,r)$ might have even different dimensions.

\begin{example}\label{ex:case_2_2_2}
Let $F=x_0^2x_1^2x_2^2$. Then $\brk{(F)} = \rk{(F)} = 9$. By Theorem \ref{prop:vps_vspb_>=-2}, $\VPS(F,r)\cong \PP^2$ and, by Proposition \ref{vspmonomials}, one has $\VSP(F, 9)\cong \PP^2$ . Again by Theorem \ref{prop:vps_vspb_>=-2} we have $\VSPb(F,r) = \phi_{r,\PP^2}^{-1}(\VPS(F,r))$. We claim that the latter is smooth irreducible and $3$-dimensional. Indeed, for every $K\in \VPS(F,r)$ we have $K = (C_1,C_2)_{\geq 4}$ for some $2$-dimensional subspace $\langle C_1, C_2 \rangle$ of $\Ann(F)_3$.  Therefore, the fiber over $K$ consists of all ideals of the form
$(aC_1+bC_2) + K$ for some $[a,b]\in \PP^1$. Hence $\VSPb(F,r)$ is irreducible and $3$-dimensional. Since $\VSPb(F,r)$ (which is always taken with its reduced structure) is a fibration on a smooth basis with smooth equidimensional fibers, it is irreducible and smooth.  
\end{example}

\subsubsection{Border rank of monomials: an improvement by one}

Given a monomial $F = x_0^{a_0}\cdots x_n^{a_n}$, with $0<a_0\leq a_1\leq \cdots \leq a_n$, its border rank is at the time of this writing unknown. We do know that $\brk(F)\leq \prod_{i=0}^{n-1}(a_i+1)$. This value is what one conjecturally expects for the border rank to be.

This is known to be true in some cases. For instance, whenever $a_n\geq \sum_{i=0}^{n-1} a_j-1$, 
then $\brk(F) = \prod_{i=0}^{n-1}(a_i+1)$ \cite[Example 6.8]{bb19}. We improve this lower bound on $a_n$ by one.

\begin{theorem}\label{borderrankmonimprove}
Let $F = x_0^{a_0}\cdots x_n^{a_n}$
with $0<a_0\leq a_1\leq \cdots \leq a_n$ and $a_n\geq \left(\sum_{i=0}^{n-1} a_i\right)-2$. Then 
\[
\brk(F) = \prod_{i=0}^{n-1}(a_i+1). 
\]
\begin{proof}
As mentioned, by \cite[Example 6.8]{bb19}, it is enough to assume that $a_n = \left(\sum_{i=0}^{n-1}a_i\right)-2$. We have that $\Ann(F) = \left(y_0^{a_0+1},\ldots, y_n^{\left(\sum_{i=0}^{n-1} a_i\right)-1}\right)$. Let $r =\prod_{i=0}^{n-1}(a_i+1)$.  We first claim that $\mathrm{HF}(S/\Ann(F), a_n+1) = r-2$. 
Let $y_0^{b_0}\cdots y_n^{b_n}$ be a monomial of degree $a_n+1$ and not belonging to $\Ann(F)$. Then 
\begin{equation}\label{ineqonbs}
0\leq b_j\leq a_j \mbox{ for } 0\leq j\leq n-1, \ \ \mbox{ and } 0\leq b_n < \left(\sum_{i=0}^{n-1}a_i\right)-1.
\end{equation}
Note that $\left(\sum_{i=0}^{n-1} b_i\right) + b_n = a_n+1$, hence $b_n = \left(\sum_{i=0}^{n-1}a_i - b_i\right)-1$. Thus for any choice of $0\leq b_j\leq a_j$ with $0\leq j\leq n-1$, we have such a monomial, except when 
$b_j = 0$ or $b_j = a_j$ for each $0\leq j\leq n-1$. Indeed, in these two cases, the inequalities on $b_n$ are not verified. Therefore there are $r-2$ such monomials, which shows the claim.  

Now, suppose by contradiction that $\brk(F)\leq r-1$. By \cite[Corollary~6.3]{bb19} there exists a monomial ideal $I \subseteq \Ann(F)$ such that $S/I$ has Hilbert function $h_{r-1, \PP^n}$.
In particular, $I_{a_n+1}$ is of codimension $1$ in $\Ann(F)_{a_n+1} = J_{a_n+1} \oplus \langle y_n^{a_n+1}\rangle$, where $J = (y_0^{a_0+1}, \ldots, y_{n-1}^{a_{n-1}+1})$. 

We consider two cases. If $I_{a_n+1} = J_{a_n+1}$, then $I$ has one minimal generator in degree $a_n+2$ that is divisible by $y_n^{a_n +1}$ since $\HF(S/J,a_n+2) = r$.
If $y_iy_n^{a_n+1}\in I$, then $y_i \in \overline{I}$. It follows that the Hilbert polynomial of $S/I$ is at most $\prod_{i=1}^{n-1} (a_i+1) < r-1$.
The other possibility is that $y_n^{a_n+1} \in I_{a_n+1}$ and $J_{a_n+1}\cap I_{a_n+1}$ is of codimension $1$ in $I_{a_n+1}$. We argue similarly showing that, for some $i\in \{0,1,\ldots, n\}$, we have $y_i\in \overline{I}$. This proves again that the Hilbert polynomial of $S/I$ is at most $\prod_{i=1}^{n-1} (a_i+1) < r-1$.
\end{proof}
\end{theorem}

\subsection{Ternary forms of low even degree}\label{ssec: ternary of low even deg}
Let $X = \PP^2, S = \CC[y_0,y_1,y_2]$ and $T = \CC[x_0,x_1,x_2]$. 
In this section, by a {\it nondegenerate} ternary form of degree $d=2p-2$, we mean a ternary form $F\in T_d$ such that 
$\Ann(F)_{\leq p-1}=0$, following the terminology of \cite[\S 1]{rs00}.  There exists a dense set of such forms, so this is a generality-type assumption.  

\begin{proposition} \label{plane curve codim}
Let $F$ be a nondegenerate ternary form of degree $d = 2p-2$ with $p\geq 2$, then for every choice of $p+1$ linearly independent forms of degree $p$ in $\Ann(F)$ there is no linear form that divides all of them.
\end{proposition}
\begin{proof}
Suppose by contradiction this is not the case. Then we have $\langle g_0, \ldots, g_p \rangle \subset \Ann(F')$ where $g_i$ are linearly independent forms of degree  $p-1$ and $F' = l\circ F$ is a form of degree $2p-3$. The Hilbert function of the quotient $S / \Ann(F')$ must satisfy:
\[
\HF(S / \Ann(F'),p-2) = \HF(S / \Ann(F'),p-1) \leq \binom{p+1}{2} - (p+1) < \binom{p}{2}.
\]
Hence we have a form $g$ of degree $p-2$ in $\Ann(F')$. This implies that the degree $(p-1)$ form $l\cdot g$ is in $\Ann(F)$, which is impossible for a nondegenerate ternary form of degree $2p-2$.
\end{proof}

\begin{theorem} \label{plane low even degree}
Let $F$ be a nondegenerate ternary form of degree $d = 2p-2$ with $2\leq p \leq 5$, we have an isomorphism between $\VPS(F, r_p)=\VSP(F,r_p)$ and $\VSPb(F,r_p)$ given by $\phi_{r_p,\PP^2}$ with $r_p:=\binom{p+1}{2} = \brk(F)$. 
\end{theorem}

\begin{proof}
The equality $\VPS(F, r_p)=\VSP(F,r_p)$ follows from \cite[Theorem~1.7~and~Lemma~1.8]{rs00}. By the assumption on $F$, the ideal of every zero-dimensional scheme in the variety of apolar schemes $\mathrm{VPS}(F,r_p)=\mathrm{VSP}(F,r_p)$ must have the generic Hilbert function. Then, by definition, the preimage under $\phi_{r_p,\PP^2}$ of each of these is a single ideal in $\VSPb(F,r_p)$. Now, we will show that every ideal $J \subset \mathrm{Ann}(F)$ with generic Hilbert function is saturated and so lies in $\phi_{r_p,\PP^2}^{-1}(\mathrm{VPS}(F,r_p))$. By Lemma~\ref{isobetweenVSP-VPSbar} this proves the statement.

Suppose that $J$ is not saturated and let $K$ be its saturation. By the first part of the proof, $K$ is not contained in $\Ann(F)$. Therefore, $\HF(S/K, 2p-2) < r_p$,  otherwise we would have $K_{2p-2} = J_{2p-2}\subset \Ann(F)$ and hence $K\subset \Ann(F)$, by Proposition \ref{inclusion in onedeg}, a contradiction.

We claim that there is a linear form $\ell$ dividing every form in $K_{2p-2}$. To prove the claim, one employs Macaulay's bound to list all the possible Hilbert functions of $K$. Then for each of them we check utilizing \cite[Theorem~7.3.4]{Schenck} that a linear form $\ell$ as above exists. 
Let $\ell'$ be a linear form such that $\dim_\CC \langle \ell, \ell'\rangle = 2$. Suppose that $F\in J_p$, then $(\ell')^{p-2}F \in J_{2p-2}\subset K_{2p-2}$. Hence $\ell$ divides $(\ell')^{p-2}F$ and thus it divides $F$. This contradicts Proposition~\ref{plane curve codim}. 
\end{proof}

\begin{remark}\label{vspevenlowdegplane}
Keep the notation from above. The varieties $\VSP(F,r_p)$ are very interesting. 
Mukai \cite{Mukai1,Mukai2} and Ranestad-Schreyer \cite[Theorem 1.7]{rs00} proved that 
\begin{enumerate}
\item[(i)] for $p=2$, $\VSPb(F,3)\cong \VSP(F,3)$ is a Fano threefold of index $2$ and degree $5$ in $\PP^6$;

\item[(ii)] for $p=3$, $\VSPb(F,6)\cong \VSP(F,6)$ is a Fano 
threefold of index $1$ and degree $22$ in $\PP^{13}$;

\item[(iii)] for $p=4$, $\VSPb(F,10)\cong \VSP(F,10)$ is a $K3$ surface of genus $20$ of degree $38$ canonically embedded in $\PP^{20}$; 

\item[(iv)] for $p=5$, $\VSPb(F,15) \cong\VSP(F,15)$ consists of $16$ reduced points.
\end{enumerate}

\end{remark}

\subsection{A reducible $\VSPb$ of a ternary form}\label{ssec: reducible vspbar}

The $\VSPb$ might be reducible; see for instance \cite[Theorem 7.9]{HMV} for examples of reducible $\VSPb$ of wild forms. We exhibit an example of a reducible $\VSPb(F,\brk(F))$ when $F$ is a ternary form. This addresses and solves the question whether $\VSPb(F,\brk(F))$ can be reducible even in $\PP^2$. Note that this is not the case when $X=\PP^1$. Let $X = \PP^2$ and $F=x_0^{11}+x_1^5x_2^6\in T_{11}$. It is easy to check that $\brk(F)=7$. Let $J=(y_0y_1, y_0y_2, y_1^6)\subset \Ann(F)$.

\begin{proposition}
We have $\VSPb(F,7) = \phi_{7,\PP^2}^{-1}(\mathrm{Proj}(S/J))$ and it is reducible.
\begin{proof}

We have $\mathrm{Ann}(F)_6 = J_6$ and $\mathrm{HF}(S/J, 6) = 7$. Therefore, the saturation of every $[I]\in \VSPb(F,7)$ is $J$.
From Lemma~\ref{lem:trivial_containment} we obtain $\VSPb(F,7) = \phi_{7,\PP^2}^{-1}(\mathrm{Proj}(S/J))$.

    If $I\in \VSPb(F, 7)$ then $(y_0^2y_1, y_0^2y_2)_4\subseteq I_4$ by \cite[Example~4.2]{JM}. Since $\mathrm{HF}(S/J,4) = 6$ and $I$ has generic Hilbert function, one has $I_4^\perp = J_4^\perp + \CC \omega$ where $\omega = ax_0x_1^3+bx_0x_1^2x_2+cx_0x_1x_2^2+dx_0x_2^3$ for some $a,b,c,d\in \CC$ not all zero.
    Note that $\partial\omega/\partial x_0 \in J_3^\perp$. 
    
    For every $[\omega]\in \PP\langle x_0x_1^3,x_0x_1^2x_2,x_0x_1x_2^2,x_0x_2^3\rangle$ there is an element $I\in \mathrm{Hilb}_{S}^{h_{7, \PP^2}}$ with $I_4^\perp = J_4^\perp + \CC\omega$. Indeed, it is enough to choose a $7$-dimensional subspace of $T_3$ that contains $J_3^\perp + \langle \partial\omega/\partial x_1, \partial\omega/\partial x_2\rangle$.
The ideal $I$ constructed in this way is uniquely determined by $I_{\geq 4}$ if and only if 
$\dim_\CC J_3^\perp + \dim_\CC \langle \partial\omega/\partial x_1, \partial \omega/\partial x_2\rangle = 7$, which is equivalent to the condition that $\langle \partial\omega/\partial x_1, \partial \omega/\partial x_2\rangle$ is 2-dimensional. 
Therefore, if it is nonempty, the locus in $\VSPb(F,7)$ of ideals corresponding to such $\omega$ is open and irreducible and so its closure $V_J$ is an irreducible component of $\VSPb(F,7)$. Observe that by construction, every element $I$ in $V_J$ satisfies $(y_0^2y_1, y_0^2y_2) \subseteq I$.

Let $\prec$ be the lexicographic monomial order with $y_0\prec y_1 \prec y_2$. Consider the ideal $K=(y_0y_2^2, y_0^2y_2, y_0^2y_1, y_0y_1^3, y_1^6)$ corresponding to $\omega = x_0x_1^2x_2$. We claim that it belongs to $V_J$. To justify this, we need to show that $K$ is in $\mathrm{Slip}_{r, \PP^2}$. This follows from the fact that it is the initial ideal of the saturated ideal $L = (y_0y_2^2+y_1^3, y_0^2y_2, y_0^2y_1)$.

Consider the ideal $K'=(y_0y_2^2, y_0y_1y_2, y_0^2y_2, y_0^2y_1^2, y_0^3y_1, y_0y_1^4, y_1^6)$. It is not in $V_J$ but it is apolar to $F$. It is sufficient to show that it is in $\mathrm{Slip}_{7,\PP^2}$ to conclude that $\VSPb(F, 7)$ is reducible.
This follows from the fact that it is the initial ideal of the saturated ideal $L'=(y_0^2y_2+y_0y_1^2,y_0y_1y_2+y_1^3,y_0y_2^2+y_1^2y_2+y_0^2y_1)$.
\end{proof}
\end{proposition}

\subsection{A Schubert variety as $\VSPb$}\label{sec:Schubert variety}

Let $\mathbb{G}(k,n) = \mathbb{G}(\PP^{k-1},\PP(V))$ be the Grassmannian
of $k$-dimensional subspaces of an $n$-dimensional vector space $V$. Fix a complete flag $\mathcal V: 0=V_0\subset V_1\subset \ldots \subset V_n= V$,
where $\dim(V_i)=i$, and a sequence $\bb{a}=(a_1,\ldots, a_k)$ of integers 
with $n-k\geq a_1\geq \cdots \geq a_k\geq 0$. For such a sequence $\bb{a}$, 
the {\it Schubert variety} $\Sigma_{\bb{a}}(\mathcal V)\subset \mathbb{G}(k,n)$ is the closed subset 
\[
\Sigma_{\bb{a}}(\mathcal V) = \lbrace \Lambda\in \mathbb{G}(k,n) \ | \ \dim(\Lambda \cap V_{n-k+i-a_i})\geq i \mbox{ for all } i\rbrace. 
\]
The definition of $\Sigma_{\bb{a}}(\mathcal V)$ depends on the specific flag. However, its class in the cohomology ring of $\mathbb{G}(k,n)$ does not. Hence, whenever the flag is fixed, the Schubert variety is denoted $\Sigma_{\bb{a}}$. 

The Schubert varieties $\Sigma_{\bb{a}}$ are irreducible, rational and of codimension $\sum_{i=1}^k a_i$ inside $\mathbb G(k,n)$ \cite[Theorem 4.1]{eh16}. In the next result, we show an instance of $\VSPb$ that is a Schubert variety. When $\bb{a}=(a)$, then we write $\Sigma_{a} = \Sigma_{\bb{a}}$.

\begin{proposition}\label{prop: schubert}
Let $X = \PP^2$ and let $F=x_0^{11}+x_0x_1^6x_2^4+x_2^{11}\in T_{11}$. Then $\brk(F)=12$ and $\VSPb(F, 12)$ is isomorphic to the  Schubert variety $\Sigma_1 \subset \mathbb{G}(3, 5)$.
\end{proposition}
    \begin{proof}
      We have $\HF(5, S/\Ann(F)) = \HF(6, S/\Ann(F)) = 12$, therefore $\brk(F)\geq 12$. On the other hand, the ideal
      $J=(y_0^2y_1, y_0^2y_2, y_1y_2^5, y_0y_2^5)$ is a saturated ideal contained in $\Ann(F)$ that defines a length $12$ subscheme of $\PP^2$. So $\mathrm{crk}(F)\leq 12$. Since $X=\PP^2$, $\mathrm{crk}(F) = \mathrm{srk}(F)\geq \brk(F) = 12$. Hence $\brk(F)=12$. 

      Assume that $I\in \VSPb(F, 12)$. Then $I_{6} = \Ann(I)_6$ which is equal to $J_6$. Since $J$ is generated in degrees at most $6$ and its quotient algebra has Hilbert function $12$ in all degrees larger than $4$ we get $I_{\geq 6} = J_{\geq 6}$. In particular, $I$ belongs to the fiber of $\phi_{12,\PP^2}\colon \mathrm{Slip}_{12,\PP^2} \to \mathrm{Hilb}^{12}(\PP^2)$ over $\mathrm{Proj}(S/J)$. It follows from Lemma~\ref{lem:trivial_containment} that $\VSPb(F, 12)$ is in fact equal to this fiber.

Recall we have the morphism $\psi_{12,\PP^2}: \mathrm{Hilb}_S^{h_{12,\PP^2}}\to \mathrm{Hilb}^{12}(\PP^2)$ (whose restriction on $\mathrm{Slip}_{12,\PP^2}$ is $\phi_{12,\PP^2}$). Let $\mathcal F$ be its fiber over $\mathrm{Proj}(S/J)$. We have proven that $\VSPb(F,12)$ is the intersection of 
 $\mathrm{Slip}_{12,\PP^2}$ with $\mathcal{F}$.
 
 Since $h_{12,\PP^2}(3) = \dim_\CC S_3$ and $h_{12,\PP^2}(d) = 12 = \HF(S/J,d)$ for every $d>4$, the closed points of $\mathcal{F}$ are parameterized by the choices of subspaces of $J_4$ of suitable dimension. We have $\dim_\CC S_4 = 15$ and $\HF(S/J,4) = 10$, so these may be identified with the closed points of $\mathbb{G}(\PP^2, \PP(J_4))=\mathbb{G}(3,5)$. To simplify notation, in what follows we identity $\mathcal F$ with $\mathbb{G}(3,5)$. 

    Consider the full flag $0 = V_0 \subset V_1 \subset V_2 \subset \cdots \subset V_5 = J_4$ where $V_i$ is spanned by $i$ smallest monomials in $J_4$ in the graded reverse lexicographic order with $y_0 \prec y_1 \prec y_2$. There are 10 monomial ideals in $\mathcal{F}$ corresponding to the choices of 3 out of 5 monomials in $J_4$. Since we chose an ordering of monomials in $J_4$, the elements of $\mathcal F$ correspond to full-rank $3\times 5$-matrices. More precisely, given $I\in \mathcal F$, the $(i,j)$th entry of the matrix corresponding to $I$ is the coefficient in front of the $j$th largest monomial in $J_4$ of the $i$th minimal generator of $I$ of degree $4$. 
    
For a sequence of integers $1\leq i_1 < i_2 < i_3 \leq 5$, let $\mathcal F_{(i_1,i_2,i_3)}$ denote the locally closed subset of  $\mathcal F$, whose corresponding matrices can be expressed in the reduced row echelon form with coefficients $1$ at entries $(s, i_s)$ for $s=1,2,3$ and zeroes for every entry $(s,t)$ with $s\in \{1,2,3\}$ and $1\leq t< i_s$. These sets cover $\mathcal F$ and furthermore, we claim that the matrices whose row spans define an element in the Schubert variety $\Sigma_1$ are precisely the matrices that belong to one of the sets $\mathcal F_{(i_1,i_2,i_3)}$ with $(i_1, i_2, i_3) \neq (1,2,3)$. Indeed, by the choice of our flag and the definition of Schubert varieties we have
\[
\Sigma_1 = \lbrace I_4 \in \mathbb{G}(3, 5) \mid I_4 \cap \langle y_0^3y_1, y_0^3y_2 \rangle \neq \{0\} \rbrace.
\]
This is the union of those $\mathcal{F}_{(i_1,i_2, i_3)}$ for which $\{4,5\} \cap \{i_1,i_2, i_3\} \neq \emptyset$. 

    To show that $\mathcal F\cap \mathrm{Slip}_{12,\PP^2}$ coincides with $\Sigma_1$, we proceed in two steps. First, we prove that $\mathcal F_{(1,2,3)}$ is disjoint from $\mathrm{Slip}_{12,\PP^2}$.
    Let $I$ be any element in $\mathcal F_{(1,2,3)}$. By our choice of the monomial order, we have $K=\mathrm{in}_{\prec}(I) = (y_0^2y_1^2, y_0^2y_1y_2, y_0^2y_2^2)+J_{\geq 5}$ and it is enough to show that $K$ is not in $\mathrm{Slip}_{12,\PP^2}$. By construction, $\overline{K}=J$. We compute that $\mathrm{Ext}_S^1(J/K, S/J)_0 = 0$. By \cite[Theorem~3.4]{JM}, $K$ is not in $\mathrm{Slip}_{12,\PP^2}$. Therefore $\VSPb(F,12)=\mathcal F\cap \mathrm{Slip}_{12,\PP^2} \subseteq \Sigma_1$. 
    
    $\Sigma_1$ is closed, irreducible and of codimension one in $\mathcal F$. Therefore, by \cite[Lemma~2.6]{Man} to derive the desired equality $\VSPb(F, 12) = \Sigma_1$, it is sufficient to find a point in $\mathcal F\cap \mathrm{Slip}_{12,\PP^2}$ such that the dimension of its tangent space to $\mathrm{Hilb}_S^{h_{12,\PP^2}}$ is $25 = \dim \mathrm{Slip}_{12,\PP^2}+1$. Let $\widehat{K} = (y_0^2y_1^2, y_0^2y_2^2, y_0^3y_2) + J_{\geq 5}$. We compute that $\dim_\CC \mathrm{Hom}_S(\widehat{K}, S/\widehat{K})_0 =25$ so we are left with showing that $\widehat{K}$ is in $\mathrm{Slip}_{12,\PP^2}$. Let $\widehat{I} = (y_0^2y_1^2+y_0y_2^3, y_0^3y_2, y_0^2y_2^2, y_0^4y_1, y_1y_2^5, y_0y_2^5)$. Its initial ideal with respect to graded reverse lexicographic order with $y_0\succ y_1\succ y_2$ is $\widehat{K}$ so it is sufficient to observe that  $\widehat{I}$ is in $\mathrm{Slip}_{12,\PP^2}$. This follows from \cite[Theorem~1.3]{Man} since the saturation of $\widehat{I}$ differs from $\widehat{I}$ only in degree $4$.
  \end{proof}

The Ext criterion seems to have a tendency to describe Schubert loci. One would be tempted 
to speculate that in some regimes {\it all} the $\VSPb$'s are Schubert loci of some Grassmannians. 
However, we refrain from formulating a precise conjecture at this stage. 

\section{Outlook and open problems}\label{sec:Outlook}

Here we explicitly pose some natural follow-up questions and directions from this work. We collect
them in one place for the convenience of the reader. For unexplained notation, see \S\ref{sec:prelim}.

\begin{question}\label{qus1}
Does there exist $F\in T[X]$, for some smooth projective toric variety $X$ and some very ample embedding of $X$, such that $\VSPb(F, \brk_X(F))$ has {\it both} a component whose general point is a saturated ideal and a component whose general point is nonsaturated? 
(Perhaps a starting point here would be to look for border varieties of sums $F+G$, where $F$ is wild and $G$ is not.) 
\end{question}

\begin{question}
Are there examples of forms $F$ whose $\VSPb(F, \brk(F))$ has the property stated in Question \ref{qus1}?
\end{question}

\begin{question}
Are there examples of the behaviour described in \ref{qus1} for a general element of some $\sigma_r(X)\subset \PP(T_{\bb{v}})$?    
\end{question}

\begin{question}\label{qus4}
Let $X = \PP^{m-1}\times \PP^{m-1}\times \PP^{m-1}$ and let $T$ be the dual to its Cox ring. 
Let $F\in T_{\bb{1}}\cong \CC^m\times \CC^m \times \CC^m$ be a nonwild minimal border rank $3$-tensor. Is its $N$-th Kronecker power $F^{\boxtimes N}\in \CC^{m^N}\otimes \CC^{m^N}\otimes \CC^{m^N}$ nonwild either? If it so, 
is it possible to express the unique ideal in $\VSPb(F^{\boxtimes N},m^N)$ in terms of the unique ideal in $\VSPb(F,m)$? This question tries to indicate some investigation 
between border apolarity and asymptotic ranks, which seems missing at the moment. 
\end{question}

\begin{question}
Keep the notation from Question \ref{qus4}. 
Let $F\in T_{\bb{1}}$ be a minimal border rank tensor. 
Is it true that $F$ is wild if and only if $\VSPb(F, \brk(F))$ is positive-dimensional? 
Can one draw similar conclusions as those of Corollary \ref{cor:minimalBRWild} when $X$ is an arbitrary smooth projective toric variety with a given embedding? 
\end{question}

\begin{question}
Keep the notation from Question \ref{qus4}. 
Suppose $F\in T_{\bb{1}}$ is a degeneration of another tensor $F'$ with $r=\brk(F)\leq r'=\brk(F')$. Is it true that there exists a surjective morphism from $\VSPb(F',r')$ to $\VSPb(F,r)$? Perhaps the open-ended problem behind this specific question is to what extent we can relate $\VSPb$'s under degenerations of the $F$'s. 
\end{question}

\begin{question}
Let $r=\brk(F)$ and suppose that $\VSPb(F,r)$ is of fiber type. What are the general properties of the morphism (say, over $\CC$)
\[
\phi_{r,X}: \VSPb(F,r)\longrightarrow \VPS(F,r)\cap \mathrm{Hilb}_{sm}^r(X)?
\]
Can this have disconnected fibers on a general point? Is the map smooth?
If not, are there easy necessary conditions for smoothness? 
\end{question}

\begin{question}
Let $X_{Seg} = \PP^{m-1}\times \cdots \times \PP^{m-1}$ ($d$ factors) and let $T$ be the dual to its Cox ring. 
Let $F\in T_{\bb{1}}\cong (\CC^m)^{\otimes d}$ be a symmetric tensor. There is another natural 
variety related to $F$, i.e. $X_{Ver} = \nu_d(\PP^{m-1})$. Let $r\geq \max\lbrace \brk_{X_{Ver}}(F), \brk_{X_{Seg}}(F)\rbrace$.
\begin{enumerate}
\item[(i)] What is the relation (if any) between $\VSPb_{X_{Ver}}(F,r)$ and $\VSPb_{X_{Ser}}(F,r)$? 
\item[(ii)] Is there a criterion on the level of border varieties to guarantee that $\brk_{X_{Ver}}(F)=\brk_{X_{Seg}}(F)$? 
\end{enumerate}
\end{question}

\begin{opquestion} 
Develop criteria for membership in $\mathrm{Slip}_{r,X}$. A most tantalizing one
is to prove or disprove the {\it flag condition}. 
Let $X$ be a smooth projective toric variety with Cox ring $S=S[X]$, and let $r$ be a positive integer. Assume
that $I\in \mathrm{Slip}_{r,X}$. Is there a flag of ideals 
$I=I_r\subseteq I_{r-1}\subseteq \cdots \subseteq I_0=S$ such that, for all $k$, $I_k\in 
\mathrm{Slip}_{k,X}$? The point is that this condition should imitate
what happens for secant varieties. See \cite[Problem 1.10]{Man22}. 
\end{opquestion}

For more questions and problems, see \cite[\S 1]{Man22}. As perhaps clear from these (at times very easily stated) questions, border apolarity has plenty of natural open interesting directions,
whose answers may have an important impact on our understanding of border rank.

\bibliographystyle{amsplain}

\begin{small}

\end{small}

\end{document}